\documentclass{article}
\usepackage{graphicx}
\usepackage{amsfonts}
\usepackage{amsthm}
\usepackage{mathtools}
\usepackage{mathrsfs}
\usepackage{textcomp}
\usepackage{caption}
\usepackage{setspace}
\usepackage{bm}
\usepackage{authblk}
\usepackage{float}
\restylefloat{table}
\usepackage{hyperref}

\usepackage{tikz-cd}

\newtheorem{theorem}{Theorem}
\newtheorem{lemma}{Lemma}
\newtheorem{cor}{Corollary}

\newtheorem{definition}{Definition}
\newtheorem*{remark}{Remark}

\newcommand*\pFqskip{8mu}
\catcode`,\active
\newcommand*\pFq{\begingroup
\catcode`\,\active
\def ,{\mskip\pFqskip\relax}%
\dopFq
}
\catcode`\,12
\def\dopFq#1#2#3#4#5{%
{}_{#1}F_{#2}\biggl[\genfrac..{0pt}{}{#3}{#4};#5\biggr]%
\endgroup
}

\newcommand\extrafootertext[1]{%
\bgroup
\renewcommand\thefootnote{\fnsymbol{footnote}}%
\renewcommand\thempfootnote{\fnsymbol{mpfootnote}}%
\footnotetext[0]{#1}%
\egroup
}

\begin{document}
\title{\bf Change of Basis from\\
Bernstein to Zernike}
\markright{Change of Basis}
\author{D.A.\ Wolfram}
\affil{\small College of Engineering \& Computer Science\\The Australian National University, Canberra, ACT 0200\\
\medskip
{\rm David.Wolfram@anu.edu.au}
}
\date{}

\maketitle
\maketitle

\begin{abstract}
We increase the scope of previous work on change of basis between finite bases of polynomials by defining ascending and descending bases and introducing three techniques for defining them from known ones.

The minimum degrees of polynomials in an ascending basis can increase such as with bases of Bernstein and Zernike Radial polynomials. They have applications in computer-aided design and optics. 

We give coefficient functions for mappings from the monomials to descending bases of Bernstein polynomials, and ascending ones of Zernike Radial polynomials and prove their correctness. 

Allowing for parity, we define eight general change of basis matrices and the related equations for composing their coefficient functions.

A main example is the change of basis from shifted Legendre polynomials to Bernstein polynomials considered by R.\ Farouki~\cite{farouki2000}. The analysis enables us to find a more general hypergeometric function for the coefficient function, and recurrences for finding coefficient functions by matrix row and column. We show the column coefficient functions are equivalent to the Lagrange interpolation polynomials of their elements. We also provide a solution to an open problem~\cite{farouki2000} by showing there is no general closed form expression for the coefficient function from Gosper's Algorithm.

Truncation, alternation and superposition increase the scope further. Alternation enables, for example, a change of basis between Bernstein and Zernike Radial polynomials. A summary shows that the groupoid of change of basis matrices is a small category, and triangular and alternating change of basis matrices are morphisms in full subcategories of it. Truncation is a covariant functor.

\end{abstract}

\extrafootertext{MSC: Primary 15A03; Secondary  20N02, 33C45, 20L05}


\pagebreak
\tableofcontents

\section{Introduction and Motivation}

This paper concerns change of basis between finite bases of polynomials over $\mathbb{R}[x]$.
It increases the scope of my previous work on change of basis that mostly involved classical orthogonal polynomials~\cite{wolfram2107,wolfram2108}.

The increased scope includes Bernstein polynomials and Zernike Radial polynomials.
It is also increased by introducing techniques that produce polynomial bases from other polynomial bases.

\subsection{Bernstein and Zernike Radial Polynomials}

For every classical orthogonal polynomial, such as $T_n(x)$ or $L_n(x)$, the greatest degree is $n$ and the least degree implicitly is either $0$ or $1$. With bases of classical orthogonal polynomials, e.g., $\{T_0(x), T_1(x), \ldots, T_n(x)\}$, generally the greatest degree of each polynomial changes between polynomials and the basis can be ordered so that it decreases or descends to $0$ or to $1$.

Unlike the classical orthogonal polynomials, Bernstein polynomials and Zer\-nike Radial polynomials expressed using monomials have least degrees that can be greater than $1$ and are specified. For example, $R_5^3(x) = 5x^5 - 4x^3$. 

A key observation of Part I of this paper is that 
bases of Bernstein and Zernike Radial polynomials allow other possibilities. With them, the greatest degree can be fixed at $n \geq 0$ and the minimum degree can increase or ascend to $n$, e.g., $\{b^7_3(x), b^7_4(x), b^7_5(x), b^7_6(x), b^7_7(x)]$. This leads to the generalization of ascending and descending bases. We show that change of basis matrices between ascending bases are lower diagonal matrices and ones for descending bases are upper diagonal matrices. We show in Theorem~\ref{tri-inv}, that the set of change of basis matrices where the domain and range bases are all ascending or all descending bases forms a connected groupoid of lower or upper triangular change of basis matrices, respectively.

Bernstein polynomials $b_k^n(x)$ form a polynomial basis that are used in defining a B{\'e}zier curve:
\[
{\bf r}(t) = \sum\limits_{k=0}^{n} {\bf p}_k b_k^n(t)
\] where $t \in [0,1]$, e.g., \cite[equation (9)]{farouki}. This is a parametric vector equation where the evaluation of the Bernstein polynomials gives scalar coefficients of the $n+1$ two-dimensional control points ${\bf p}_k$.
The B{\'e}zier curve can be evaluated by using De Casteljau's algorithm~\cite{farouki} that was developed in 1959 and used for computer-aided design of car bodies. These techniques are now widely used in computer graphics and computer-aided design software~\cite{farouki,li}. 

Change of basis with Bernstein polynomials enables functions or curves expressed with other bases to be represented as B{\'e}zier curves, and the converse. Some previous work has considered particular cases of change of basis involving Bernstein polynomials. These include shifted Legendre polynomials~\cite{farouki2000,li}, Chebyshev polynomials of the first kind~\cite{li}, and shifted and generalised Chebyshev polynomials of the second kind~\cite{alqudah}.

Zernike polynomials were defined by Frits Zernike in 1934. They are orthogonal polynomials on the unit disk and have two independent variables: a radius and an azimuthal angle, e.g.,~\cite{bhatia}, \cite[\S 4]{boyd2}, \cite{lak}. They are used to measure and describe optical wavefront aberrations such as defocus, astigmatism, and coma, and have applications in optometry, telescope mirror alignment such as in the James Webb Space Telescope~\cite{howard}, and adaptive optics~\cite{hampson}. Adaptive optics is used, for example, in large terrestrial telescopes such as at the Keck Observatory~\cite{mawet} to reduce wavefront aberrations caused by atmospheric turbulence, and de-twinkle stars. It enabled the study of planets that orbit stars outside the solar system. An example is $\beta$ Pictoris b by the European Southern Observatory's Very Large Telescope in Chile~\cite{lagrange} using its Nasmyth Adaptive Optics System (NAOS).

A wavefront can be modelled with Zernike polynomials. It has the form of a sum terms each of which has 
a polynomial factor called a Zernike Radial polynomial that uses the radius as an independent variable.
Boyd and Yu~\cite{boyd2} considered four variants of Chebyshev polynomials of the first kind as alternatives to Zernike Radial polynomials.

For example, change of basis enables a radial polynomial to be expressed with Chebyshev polynomials of the first kind or Zernike Radial polynomials:
\[
\frac74 T_8(r) + \frac{73}{32} T_6(r) - \frac5{16} T_4(r) + \frac{39}{32} T_2(r) + \frac1{16} T_0(x) = 4 R_8^2(r) - 2 R_4^0(r) + 3R_6^2(r).
\]

In this paper, we consider change of basis between Zernike Radial polynomials and other polynomial bases in general, but do not compare their properties. Advantages of using two equivalent formalizations of a wavefront function are that in solving partial differential equations it can be preferable computationally to use a basis other than Zernike Radial polynomials, e.g., for faster convergence~\cite{boyd2}. A formalization using Zernike Radial polynomials is also useful because it enables standard and comparable measurements especially in optometry and opthalmology, and widely-understood aberrations are associated with orthogonal terms~\cite{vsia}. 

Additionally, deformable mirror actuator software in adaptive optics can use Zernike polynomial coefficients as inputs. Li and Jiang \cite[\S 4.1]{liSPIE}  define this use of Zernike coefficients. 
Fern{\'a}ndez and Artal~\cite{fernandez} used the coefficients to control an OKO Technologies deformable mirror, and Zhang et. al.~\cite[\S 2.2]{zhang} used a Boston Micromachines Corporation deformable mirror with 140 mirror actuators.
If processing of data from wavefront sensors uses a different polynomial basis, then a change of basis would be required to find the Zernike coefficients.

\subsection{Coefficient Functions and Polynomial Bases}

We consider a general formalization of change of basis with finite polynomial bases. 
This builds on previous work by Wolfram~\cite{wolfram2107,wolfram2108} where an intermediate or exchange basis of monomials is used. This enables $n^2$ change of basis combinations to be simplified to $2n$ changes of bases that can then be composed.

In part I, we consider four general kinds of products of change of basis matrices. They occur when all bases are descending bases, or ascending bases. They can have definite parity or non-definite parity. Coefficient functions~\cite{wolfram2107} are algebraic expressions which evaluate to connection coefficients for changes of basis.

These definitions enable us to prove the correctness of new coefficient functions that are for descending bases of Bernstein polynomials in the proof of Theorem~\ref{invdesa}, and ascending bases of Zernike Radial polynomials in the proof of Theorem~\ref{beta-RM}.

In Part II, we use these cases to find four others, when one basis is ascending and the other is descending. They can be formed by composition of coefficient functions, and we provide equations for the eight cases. A summary of them is given in Table~\ref{CoBM}.

As a main example, we apply this to the analysis of an open problem of Farouki~\cite{farouki2000} of whether there is a closed form expression for the coefficient function for the change of basis from shifted Legendre polynomials to Bernstein polynomials. We show that this change of basis matrix is from a descending basis to an ascending basis without definite parity. This leads to expressing the coefficient function as an indefinite hypergeometric sum. Gosper's Algorithm does not find a  general closed form expression for the sum. However, from Zeilberger's algorithm, we can find coefficient functions for any specific row or column of the change of basis matrix.

The second way the scope has increased is presented in Part III. This introduces techniques to produce new polynomial bases from others. The new bases can be used in mappings with polynomial bases whose minimum degrees can be greater than $1$,

These techniques are truncation, alternation and superposition of polynomials. Change of basis matrices and their inverses and associated coefficient functions are used to analyse these techniques.

Truncation is a technique that can be used to form ascending bases from classical orthogonal polynomials by subtracting terms that have degree less than an integer $l$ or greater than an integer $u$. For example, $T_7(x)_{5, 3} = -112 x^5 + 56x^3$ where $l=3$ and $u=5$, and the terms $64 x^7$ and $-7x$ have been subtracted from $T_7(x)$. We formalise how truncation can be extended to change of basis matrices. 

A main result of this is Theorem~\ref{local} that matrix inversion and truncation of diagonal change of basis matrices is commutative. In practical terms, it shows that truncation should be done before inversion to reduce complexity.

Alternation produces bases without definite parity from ones that have definite parity. This example takes 
\[\{U_0(x), U_2(x), U_4(x), U_6(x)\} \: \mbox{ and } \: \{U_1(x), U_3(x), U_5(x), U_7(x)\}
\] to form their union which spans the same vector space as $\{1, x, x^2, \ldots, x^7\}$. 

From Theorem~\ref{inv-alt}, we show that the inverse of a diagonal change of basis matrices of an alternating basis is the same matrix but with the coefficient function application replaced by its inverse. This enables the use of the inverse coefficient function to find the elements of the inverse matrix.

We show in Theorem~\ref{conv-supalt} that superposition and alternation are converses. A basis formed by superposition has polynomials that are the sum of pairs of polynomials in an alternating basis, e.g. 
$\{U_0(x), U_1(x)+U_0(x), U_2(x)+U_1(x), \ldots , U_7(x) + U_6(x)\}$. These are Chebyshev polynomials of the fourth kind $V_n(x)$ where $0\leq n \leq 7$ and they do not have definite parity. 

Truncation can also be applied to polynomials in bases formed from alternation or superposition.

\part{Ascending and Descending Representations}

Theorem 1 of Wolfram~\cite{wolfram2107}, gives conditions for composing finite change of basis matrices. In general,
\begin{equation} \label{mtv}
M_{tv} = M_{ts} M_{sv}
\end{equation} where $s$, $t$ and $v$ are bases of the same vector space, and the matrices are change of basis matrices.

The exchange basis is $s$. We choose this to be either $\{x^m, x^{m+1}, \ldots, x^n\}$, or $\{x^m, x^{m+2}, \ldots, x^n\}$ where $0 \leq m \leq n$ and $m$ and $n$ have the same parity.

\begin{definition}
A basis polynomial is a polynomial over $\mathbb{C}[x]$ of the form $f_n(x)$ where $n \geq 0$ such that its degree in $x$ is $n$.
\end{definition}

The Bernstein polynomials, Zernike Radial polynomials, monomials, and all classical orthogonal polynomials are basis polynomials.
We are concerned with defining basis polynomials for $v$ and $t$ that can be used in equation~(\ref{mtv}) above with the exchange basis $s$.

Basis polynomials such as classical orthogonal polynomials~\cite{dlmf} implicitly have minimum degree of $0$ or $1$~\cite{dlmf}. A finite basis formed from them, such as $\{T_0(x), T_1(x), \ldots, T_N(x)\}$ comprises polynomials with a fixed minimum degree of $0$ or $1$, and maximum degree of $N$. The polynomials can be ordered by descending degree to the minimum of $0$. An example is the representation of a function expressed using Chebyshev polynomials of the first kind, i.e.,
\[
p_N(x) = \sum\limits_{n=0}^N c_n^{\mbox\it cheb} T_n(x) 
\] where $x \in [0,1]$ from Hale and Townsend~\cite[equation (1.1)]{hale}.

If the basis polynomials have a variable minimum degree $m$ where $0 \leq m \leq n$ and $n$ is the maximum degree, then we can define polynomial bases of them where the maximum degree is fixed at $n$ and the minimum degree ascends to it.

For example, suppose that $p(x) = 16 x^7 -12 x^5 + 5 x^4 + 3 x^2$.
This polynomial has an ascending representation in Bernstein polynomials of
\[
p(x) = \frac{1}7b_{2}^7(x) + \frac{3}7 b_{3}^7(x) + b_{4}^7(x) + \frac{11}7 b_{5}^7(x) + \frac{6}7 b_{6}^7(x) + 12 b_{7}^7(x).
\]

Its descending representation is
\[
p(x) = -\frac{16}{21} b_{2}^7(x) + \frac{16}3 b_{2}^6(x) -\frac{74}5 b_{2}^5(x) + \frac{43}2 b_{2}^4(x) -18 b_{2}^3(x) + 12 b_{2}^2(x).
\]

The ascending and descending representations of $p(x)$ in Zernike Radial polynomials is simplified by expressing $p(x)$ by $(16 x^7 - 12 x^5) +(5 x^4 + 3 x^2)$ where the two parts have definite parity.
The ascending representation is
\[
p(x) = (2R_7^5(x) + 2R_7^7(x)) + (-R_4^2(x) + 9 R_4^4(x)).
\]

The descending representation is
\[
p(x) = (\frac{16}7 R_7^5(x) + \frac{12}7 R_5^5(x)) + (\frac{5}4 R_4^2(x) + \frac{27}4 R_2^2(x)).
\]

These can be combined, e.g.,
\[
p(x) = (2R_7^5(x) + 2R_7^7(x)) + (\frac{5}4 R_4^2(x) + \frac{27}4 R_2^2(x)).
\]

Parity is a separate attribute of basis polynomials. Six of the fifteen classical orthogonal polynomials~\cite{wolfram2108} have definite parity, e.g., Gegenbauer polynomials and Chebyshev polynomials of the first and second kinds. The Zernike Radial polynomials also have definite parity. For each of them, the minimum and maximum degree have the same parity, and when expressed using monomials, the smallest difference of degree between distinct terms is $2$.
For example
\begin{align*}
U_6(x) =& 64 x^6 - 80 x^4 + 24 x^2 -1\\
R_7^3(x) = & 21 x^7 - 30 x^5 + 10 x^3.
\end{align*}

The other classical orthogonal polynomials and the Bernstein polynomials do not have definite parity in general. For each of them, the minimum and maximum degree can have different parities, and when expressed using monomials, the smallest difference of degree between distinct terms is $1$:
\begin{align*}
L_4(x) =& \frac1{24}( x^4 - 16 x^3 + 72 x^2 - 96x + 24)\\
V_5(x) =& 32 x^5 - 16 x^4 - 32 x^3 + 12 x^2 + 6x -1\\
b_2^5(x) =& -10 x^5 + 30 x^4 - 30 x^3 + 10 x^2.
\end{align*}

\section{Ascending and Descending Bases }
Truncation subtracts from a polynomial the terms that have degrees higher or lower than particular values. For example,
\begin{align*}
U_6(x)_{8,4} =& 64 x^6 - 80 x^4.\\
L_4(x)_{3,1} =& \frac1{24}(- 16 x^3 + 72 x^2 - 96x)\\
b_2^5(x)_{4,3} =& 30 x^4 - 30 x^3.
\end{align*}

Truncation enables ascending and descending bases to be defined from any polynomial basis. 

Suppose $s = \{x^4, x^5, x^6, x^7, x^8\}$ so that $m = 4$ and $n=8$. 
If the source or target basis of a mapping is truncated Laguerre polynomials, it can be $\{L_4(x)_{4, 4}, L_5(x)_{5,5}, L_6(x)_{6, 4}, L_7(x)_{7,5}, L_8(x)_{8, 4}\}$. For example, 
\[
L_6(x)_{6, 4} = \frac1{720}(x^6 -36 x^5 +450 x^4).
\] If the source or target of the mapping is truncated Chebyshev polynomials of the first kind, it can be
\[
\{T_4(x)_{4, 4}, T_5(x)_{5,5}, T_6(x)_{6, 4}, T_7(x)_{7,5}, T_8(x)_{8, 4}\}.
\]
These polynomials omit terms of degree less than $4$. For example, $T_6(x)_{6, 4} = 32 x^6 - 48x^4$ and $T_7(x)_{7, 5} = 64x^7 - 112 x^5$.

The following definition formalises truncation. 

\begin{definition} \label{parcoeff}
Suppose that $f(x)$ is a polynomial over $\mathbb{R}[x]$ of degree $n$ and minimum degree $m$ where $0 \leq m \leq n$. If $f$ has definite parity, then expressed using the monomials it has the form
\begin{equation}
f(x) = \sum\limits_{k=0}^{\frac{n-m}2} \beta(n, m, k) x^{n-2k}.
\end{equation}

If $f$ does not have definite parity, then expressed using the monomials it has the form
\begin{equation}
f(x) = \sum\limits_{k=0}^{n-m} \alpha(n, m, k) x^{n-k}.
\end{equation}

The functions $\alpha$ and $\beta$ are coefficient functions~\cite{wolfram2107,wolfram2108}.

\smallskip

The notation $f(x)_{u,l}$ is defined where $u \geq l \geq 0$, and if $f(x)$ has definite parity then $u$, $l$ and $n$ have the same parity. Let $n' = \min \{u, n\}$ and $m' = \max \{l, m\}$. Then if $f(x)$ has definite parity,

\begin{equation}
f(x)_{u,l} = \sum\limits_{k=0}^{\frac{n'-m'}2} \beta(n, m, \frac{n-n'}2 + k) x^{n'-2k}.
\end{equation}

If $f$ does not have definite parity, then
\begin{equation} \label{NDtruncalpha}
f(x)_{u,l} = \sum\limits_{k=0}^{n' -m'} \alpha(n, m, n-n' + k) x^{n'-k}.
\end{equation}

\end{definition}

\begin{remark}\rm
In Definition~\ref{parcoeff}, if $n' < m'$ then $[l, u] \cap [m,n] = \emptyset$ and $f(x)_{u, l}$ is nothing.
\end{remark}

We now consider polynomial bases and change of basis matrices. The following definition enables the degree and minimum degree of a polynomial to be specified.

\begin{definition} \label{desasc}
Let $d$ be either $1$ or $2$. 
A polynomial basis over $\mathbb{R}[x]$ is a descending basis if it has the form
\[
\{f(x)_{m,m}, f(x)_{m+d,m}, f(x)_{m+2d, m}, \ldots, f_{n, m}(x)\}
\] where for each polynomial of the form $f(x)_{u,l}$ in this basis, its degree is $u$.
\medskip

A polynomial basis over $\mathbb{R}[x]$ is an ascending basis if it has the form
\[
\{f(x)_{n,m}, f(x)_{n,m+1}, \ldots , f(x)_{n,n-2d}, f(x)_{n,n-d}, f(x)_{n, n}\}.
\] where for each polynomial of the form $f(x)_{u,l}$ in this basis, its minimum degree is $l$.

\end{definition}

\bigskip

These bases span the same vector space as the following bases:
\begin{itemize}
\item If $d=1$, $\{x^m, x^{m+1}, \ldots, x^n\}$ where $0 \leq m \leq n$, or
\item If $d=2$, $\{x^m, x^{m+2}, \ldots, x^n\}$ where $0 \leq m \leq n$ and $m$ and $n$ have the same parity.
\end{itemize} 
These bases of monomials are both ascending and descending bases. An example of a descending basis with definite parity is $\{T_3(x)_{3,1}, T_5(x)_{5,1}, T_7(x)_{7,1}\}$ which has $n =7$, $m= 1$ and it spans the same vector space as $\{x, x^3, x^5, x^7\}$. An example of an ascending basis without definite parity is 
\[
\{b^{5}_2(x)_{7,2}, b^5_3(x)_{7,3}, b^6_4(x)_{7,4}, b^6_5(x)_{7,5}, b^7_6(x)_{7,6}, b^7_7(x)_{7,7}\}
\] where $n=7$, $m=2$ and $d=1$.

\section{Change of Basis Matrices}
\label{CoBS}

Let $t$ and $v$ be both descending or ascending polynomial basis that span the same vector space. We are concerned with the form of the change of basis matrix $M_{vt}$ and its inverse. From Wolfram~\cite[\S 3]{wolfram2107}, if $M_{vt}$ is a change of basis matrix, its columns of are the transposed coordinate vectors with respect to $v$ of the basis vectors of $t$. These coordinates are given by the evaluations of the coefficient functions $\alpha$ and $\beta$ of Definition~\ref{parcoeff}.

If the bases $t$ and $v$ do not have definite parity and they are descending bases and the polynomials have maximum degree of $n$ and minimum degree of $m$, then the change of basis matrix $M_{vt}$ has the following form.

\begin{figure}[ht]
\resizebox{\textwidth}{!}{%
$\displaystyle
\left[
\begin{array}{rrrrrr}
\alpha(m,m, 0) & \cdots & \alpha(n-3,m, n-m-3) & \alpha_!(n-2,m, n-m-2) & \alpha(n-1, m, n-m-1) & \alpha(n,m, n-m)\\
\vdots & \vdots & \vdots & \vdots & \vdots & \vdots \\
0 & \cdots & \alpha(n-3, m, 0) & \alpha(n-2, m, 1) & \alpha(n-1, m, 2)& \alpha(n ,m, 3)\\
0 &\cdots & 0 & \alpha(n-2, m, 0) & \alpha(n-1, m, 1)& \alpha(n, m, 2)\\
0 & \cdots &0 & 0 & \alpha(n-1, m, 0)& \alpha(n, m, 1)\\
0 & \cdots &0 & 0 & 0 & \alpha(n, m, 0)
\end{array} 
\right]
$}
\caption{$M_{vt}$ for a Descending Basis Without Definite Parity} \label{MDesNP}
\end{figure}
\medskip

This matrix has elements
\begin{equation}
m_{i,j} = 
\left\{
\begin{array}{ll} 
\alpha(m+j, m, j-i ) & \mbox{if $j \geq i$}\\
0 & \mbox{if $j< i$}
\end{array}
\right. \label{mDND}
\end{equation} where $0 \leq i,j \leq n-m$.

For Laguerre polynomials, $m=0$ and the coefficient function for the matrix $M_{LM}$ for the mapping from the monomials is given by
\[
\alpha(n, m, k) = (-n)_{n-k} (n -k+1)_k
\] from \cite[equation 18.18.19]{dlmf}, where $0 \leq k \leq n$. When $n=3$ and $m=1$, this gives

\begin{center}
\begin{doublespace}
\noindent\(\left[
\begin{array}{ccc}
-1 & -4 & -18 \\
\phantom{-}0 &\phantom{-} 2 & \phantom{-}18 \\
\phantom{-}0 & \phantom{-}0 & -6 \\
\end{array}
\right]\)
\end{doublespace}
\end{center} From the third column, we have
\begin{align*}
x^3 =& -6 L_3(x)_{3,1} +18 L_2(x)_{2,1} - 18 L_1(x)_{1,1}\\
=& (x^3 - 9x^2 + 18 x) +(9x^2 - 36 x) + 18 x.
\end{align*}

\medskip

If $v$ and $t$ are ascending bases, then $M_{vt}$ has the form

\medskip
\begin{figure}[ht]
\resizebox{\textwidth}{!}{%
$\displaystyle
\left[
\begin{array}{rrrrrr}
\alpha(n, m,n-m) & \cdots & 0 & 0 & 0 & 0\\
\vdots & \vdots & \vdots & \vdots & \vdots & \vdots \\
\alpha(n,m,3)& \cdots & \alpha(n, n-3, 3) & 0 & 0& 0\\
\alpha(n,m,2) &\cdots & \alpha(n, n-3, 2) & \alpha(n,n-2, 2) & 0& 0\\
\alpha(n,m,1) & \cdots &\alpha(n,n-3,1)& \alpha(n,n-2,1) & \alpha(n, n-1, 1)& 0\\
\alpha(n,m,0) & \cdots &\alpha(n,n-3, 0) & \alpha(n,n-2, 0) & \alpha(n,n-1, 0) & \alpha(n, n, 0)
\end{array} 
\right]
$}
\caption{$M_{vt}$ for an Ascending Basis Without Definite Parity} \label{MDAscNP}
\end{figure}

\medskip

Its elements are
\begin{equation}
m_{i,j} = 
\left\{
\begin{array}{ll} 
\alpha(n, m+j, n-m-i ) & \mbox{if $i \geq j$}\\
0 & \mbox{if $j> i$}
\end{array}
\right. \label{mAND}
\end{equation} where $0 \leq i,j \leq n-m$.

An example of this is the change of basis matrix for the mapping from the Bernstein polynomials $\{b_3^7(x), b_4^7(x), b_5^7(x), b_6^7(x), b_7^7(x)\}$ to the monomials $\{x^3, x^4, x^5, x^6, x^7\}$ where $n=7$ and $m=3$:

\begin{center}
\begin{doublespace}
\noindent\(\left[
\begin{array}{ccccc}
\phantom{00}35 & \phantom{-}0 & \phantom{-}0 & \phantom{-}0 & 0 \\
-140 & \phantom{-}35 & \phantom{-}0 & \phantom{-}0 & 0 \\
\phantom{-}210 & -105 & \phantom{-}21 & \phantom{-}0 & 0 \\
-140 & \phantom{-}105 & -42 & \phantom{-}7 & 0 \\
\phantom{00}35 & -35 & \phantom{-}21 & -7 & 1 \\
\end{array}
\right]\)
\end{doublespace}
\end{center}
The coefficient function $\alpha(n,m,k)$ is given in equation~(\ref{MB}).

\medskip
If the bases $v$ and $t$ have definite parity, then the change of basis matrix $M_{vt}$ is
\medskip

\begin{figure}[H]
\resizebox{\textwidth}{!}{%
$\displaystyle
\left[
\begin{array}{rrrrr}
\beta(m,m, 0) & \cdots & \beta(n-4,m, l -2) & \beta(n-2, m, l -1)& \beta(n,m, l)\\
0 & \ldots & \beta(n-4, m, l -1)& \beta(n-2, m, l -2) & \beta(n ,m, l-1)\\
\ldots & \ldots & \ldots & \ldots & \ldots \\
0 & \cdots & \beta(n-4, m, 1) & \beta(n-2, m, 2)& \beta(n ,m, 3)\\
0 &\cdots & \beta(n-4, m, 0) & \beta(n-2, m, 1)& \beta(n ,m, 2)\\
0 & \cdots & 0 & \beta(n-2, m, 0)& \beta(n ,m, 1) \\
0 & \cdots & 0 & 0 & \beta(n, m, 0)
\end{array} 
\right]
$} 
\caption{$M_{vt}$ for a Descending Basis With Definite Parity} \label{MDesP}
\end{figure}
\medskip 

This matrix has elements
\begin{equation}
m_{i,j} = 
\left\{
\begin{array}{ll} 
\beta(m+2j , m, j-i ) & \mbox{if $j \geq i$}\\
0 & \mbox{if $j< i$}
\end{array}
\right. \label{mDD}
\end{equation} where $0 \leq i,j \leq l$.

As an example, the change of basis matrix from Chebyshev polynomials of the first kind of even degree to the monomials where
$0 \leq n \leq 6$ is
\begin{gather*}
\begin{bmatrix}
1 & -1 & \phantom{-}1 & -1\\
\\
0 & \phantom{-}2& -8 & \phantom{-}18 \\
\\
0 & \phantom{-}0 & \phantom{-}8& -48 \\
\\
0& \phantom{-}0 & \phantom{-}0&\phantom{-}32
\end{bmatrix}
\end{gather*}
From the fourth column, we have
\[
T_6(x) = 32 x^6 - 48 x^4 + 18 x^2 - 1.
\]
\medskip

When $v$ and $t$ are ascending bases with definite parity and let $l = \frac{n-m}2$. The change of basis matrix $M_{vt}$ is

\begin{figure}[H]
\resizebox{\textwidth}{!}{%
$\displaystyle
\left[
\begin{array}{rrrrr}
\beta(n,m, l) & \cdots & 0 &0 & 0\\
\beta(n,m,l-1) & \ldots & 0& 0 & 0\\
\ldots & \ldots & \ldots & \ldots & \ldots \\
\beta(n,m,3) & \cdots & 0 & 0& 0\\
\beta(n,m,2) &\cdots & \beta(n, n-4, 2) & 0& 0\\
\beta(n,m,1) & \cdots & \beta(n, n-4, 1) & \beta(n, n-2, 1)& 0\\
\beta(n,m,0) & \cdots & \beta(n, n-4, 0) & \beta(n, n-2, 0) & \beta(n, n, 0)
\end{array} 
\right]
$} 
\caption{$M_{vt}$ for an Ascending Basis With Definite Parity} \label{MDAscP}
\end{figure}
\medskip

We have
\begin{equation}
m_{i,j} = 
\left\{
\begin{array}{ll} 
\beta(n, m+2j, l- i ) & \mbox{if $i \geq j$}\\
0 & \mbox{if $i < j$}
\end{array}
\right. \label{mAD}
\end{equation} where $0 \leq i,j \leq l$.

An example of a change of basis matrix in this form is for the mapping from the basis $t = \{x^3, x^5, x^7, x^9\}$ to the ascending basis of Zernike Radial polynomials $v= \{R_9^3(x), R_9^5(x), R_9^7(x), R_9^9(x)\}$ where $n=9$, $m=3$ and $l=3$. The change of basis matrix $M_{vt}$ is
\begin{center}
\begin{doublespace}
\noindent\(\left[
\begin{array}{cccc}
-\frac{1}{20} & \phantom{-}0 & \phantom{-}0 & 0 \\
\phantom{-}\frac{1}{4} & \phantom{-}\frac{1}{21} & \phantom{-}0 & 0 \\
-\frac{7}{10} & -\frac{1}{3} & -\frac{1}{8} & 0 \\
\phantom{-}\frac{3}{2} & \phantom{-}\frac{9}{7} & \phantom{-}\frac{9}{8} & 1 \\
\end{array}
\right]\)
\end{doublespace}
\end{center}
The coefficient function $\beta(n, m, k)$ is given in equation~(\ref{b4}).
From the second column, we have
\begin{align*}
x^5 =& \frac97 R_9^9(x) - \frac13 R_9^7(x) + \frac1{21}R_9^5(x)\\
=& \frac97 x^9 - \frac13 (9x^9 - 8x^7) + \frac1{21} (36 x^9 -56 x^7 +21 x^5).
\end{align*}

\begin{theorem}
The four kinds of change of basis matrices $M_{vt}$ are well-defined and they are invertible matrices.
\end{theorem}

\begin{proof}
The change of basis matrices $M_{vt}$ are well-defined because from Definition~\ref{desasc}, if $v$ and $t$ are descending bases, 
then each polynomial in $v$ or $t$ has the form $h(x)_{u,l}$ where the degree $u$ is unique with respect to the other polynomials in $v$ or $t$. If $v$ and $t$ are ascending bases, then each polynomial in $v$ or $t$ has a minimum degree that is similarly unique. 

In all four cases, the change of basis matrix $M_{vt}$ is a triangular matrix: an upper triangular matrix when $v$ and $t$ are descending bases, and a lower triangular one when $v$ and $t$ are ascending bases. 

Let $d$ be either $1$ or $2$ depending on whether $v$ does not have definite parity, or has definite parity, respectively.
If $v$ is a descending basis, then the column $j$ where $0 \leq j \leq \frac{m-n}d$ is the transposed coordinate vector of $f(x)_{ m+jd, m}$ with respect to the basis $t$, i.e.,
\[
f(x)_{m + jd, m} = \sum\limits_{k=0}^{j} m_{j-k, j} g(x)_{m + (j-k)d, m}
\]
where $v = \{g(x)_{m,m}, g(x)_{m+d, m}, \ldots, g(x)_{n-d, m}, g(x)_{n,m}\}$. By Definition~\ref{desasc}, the only polynomial in $t$ in the sum above with degree $m+jd$ is $g(x)_{m+jd, m}$ when $k=0$. It follows that $m_{j,j} \not=0$ because $f(x)_{m + jd, m}$ also has degree $m + jd$.

If $v$ is an ascending basis, then the column $j$ where $0 \leq j \leq \frac{m-n}d$ is the transposed coordinate vector of $f_{n, m+jd}(x)$ with respect to the basis $t$. Similarly, the element $m_{j,j} \not= 0$. 

It follows that the elements of the main diagonals of these matrices are non-zero. Hence their determinants are non-zero, and so $M_{vt}$ is an invertible matrix.
\end{proof}

From Lemma 1 and Lemma 2 of Wolfram~\cite{wolfram2107}, we have $M_{vt}^{-1} = M_{tv}$.

The following theorem shows that upper or lower triangularity of a change of basis matrix is preserved by matrix inversion and multiplication.

\begin{theorem} \label{tri-inv}
Let $V$ be a finite dimensional vector space of polynomials over $\mathbb{R}[x]$ that spans $s = \{x^m, x^{m+d}, \ldots, x^{n-d}, x^n\}$, where $0 \leq m \leq n$, $d$ is either $1$ or $2$, and if $d=2$ then $n$ and $m$ have the same parity.

Let $S$ be a set of bases of $V$ that are either ascending or descending bases, and where there are $k_1$ ascending bases and $k_2$ descending bases.

The set of change of basis matrices $M_{vt}$ where $v,t \in S$ and both are ascending bases is a connected sub-groupoid of order $k_1^2$ of lower-triangular matrices. If $v,t \in S$ and both are descending bases, the set of change of basis matrices is s a connected sub-groupoid of order $k_2^2$ of upper-triangular matrices.
\end{theorem}

\begin{proof}
Wolfram~\cite[Theorem 1]{wolfram2107} is that the set of change of basis matrices between $m$ bases that span the same vector space is a connected groupoid of order $m^2$ with the groupoid operations of matrix inversion and matrix multiplication of the form $M_{vr} M_{rt}$ where $r, t, v \in S$.

From equations (\ref{mDND}) -- (\ref{mAD}),
if $M_{vt}$ is a change of basis matrix between ascending or descending bases, it is a lower- or upper-triangular matrix, respectively. It is straightforward to show that the upper- or lower triangularity of the elements of these sub-groupoids are invariant under the two groupoid operations.
\end{proof}

\subsection{Change of Basis with Band Matrices}

The following property shows that the inverse of change of basis matrices formed from truncations of polynomials are band matrices. 
They are upper- or lower-band matrices with bandwidth of $1$.
This simplifies finding the coefficient function for changes of bases where the range is a basis formed truncations of a polynomial.

\begin{lemma} \label{truncband}
Let $f(x)$ be a polynomial over $\mathbb{Q}[x]$ with minimum degree that does not exceed $m$, and degree that is at least $n$ where $0 \leq m \leq n$.

The basis $t$ is either an ascending basis of the form
\[
t = \{f(x)_{n,m}, f(x)_{n,m+d}, \ldots , f(x)_{n,n-2d}, f(x)_{n,n-d}, f(x)_{n, n}\}
\]
or a descending basis of the form
\[
t = \{f(x)_{m,m}, f(x)_{m+d,m}, f(x)_{m+2d, m}, \ldots, f_{n, m}(x)\}
\] 
where either $d=1$ or $d=2$.

Let $M_{vt}$ be a change of basis matrix where $t$ and $v$ are both ascending or descending bases with elements $m_{i,j}$ where $0 \leq i, j \leq \frac{n-m}d$.

The change of basis matrix $M_{tv}$ is a lower band matrix if $t$ is an ascending basis, or an upper band matrix if $t$ is a descending basis. 

Let $M_{tv}$ have elements $m'_{i, j}$.
If $t$ is an ascending basis,
\[
m'_{i,j} = 
\left\{
\begin{array}{ll} 
\frac{1}{m_{i,j}} & \mbox{if $i = j$}\\
\\
-\frac{1}{m_{j, j}} & \mbox{if $i=j+1$}\\
\\
0 & \mbox{otherwise.}
\end{array}
\right.
\]

If $t$ is a descending basis,
\[
m'_{i,j} = 
\left\{
\begin{array}{ll} 
\frac{1}{m_{i,j}} & \mbox{if $i = j$}\\
\\
-\frac{1}{m_{j,j}} & \mbox{if $i= j-1$}\\
\\
0 & \mbox{otherwise.}
\end{array}
\right.
\]

\end{lemma}

\begin{proof}
We give the proof step for the case that $t$ is an ascending basis without definite parity. The three other cases are similar.
The change of basis matrix $M_{vt}$ has the form of Figure~\ref{MDAscNP} if $v$ and $t$ do not have definite parity, or Figure~\ref{MDAscP} if they do have definite parity. In both cases, all elements in the same row in the lower triangle are equal, because each column with index $j$ of $M_{vt}$ is the transposed coordinate vector of $f(x)_{n, m +jd}$ with respect to the basis $v$ where $0 \leq j \leq \frac{n-m}d$.

The proof is by induction in forming the inverse of $M_{vt}$ by forward substitution. In the base case, we have that $m'_{0,0}$ of $M_{tv}$ is $\frac{1}{\alpha(n,m,n-m)}$, as required. The inverse of $m_{1, 0}$ of $M_{vt}$ is the solution of 

\begin{equation*}
\frac{\alpha(n,m,n-m-1)}{\alpha(n,m,n-m)} + \alpha(n, m+1, n-m-1) m'_{1,0} = 0.
\end{equation*}

We have that $\alpha(n, m+1, n-m-1) = \alpha(n, m+1, n-m)$, so that $m'_{1,0} = -\frac{1}{\alpha(n,m,n-m)}$, as required.
Suppose the result holds for $1 < i \leq h < n-m$. The inverse of $m_{h,0} = \alpha(n,m, n-m-h)$ is the solution of 
\begin{equation*}
\frac{\alpha(n,m,n-m-h)}{\alpha(n,m,n-m)} - \frac{\alpha(n,m+1,n-m-h)}{\alpha(n,m,n-m)} + \alpha(n, m+h, n-m-h) m'_{h,0} = 0
\end{equation*} because $\alpha(n, m+i, n-m-i) = 0$ where $1 < i < h$ by the induction hypothesis.

It follows that $m'_{h,0} = 0$ because $\alpha(n,m+1,n-m-h) = \alpha(n,m,n-m-h)$ and the result holds for $i = h$. By the principle of mathematical induction, it holds for all $i: 0 \leq i \leq n-m$. The induction proof can be repeated for the second column to form the second column of the inverse matrix, and all subsequent columns. The proof is similar when $v$ has definite parity.

If $t$ is a descending basis, the proof steps are similar except that back substitution is used in the proof by induction.
\end{proof}

As an example of Lemma~\ref{truncband}, let $f(x)$ be the Laguerre polynomial $L_{10}(x)$, $m = 3$ and $n = 6$. We have
\[
L_{10}(x)_{6,3} = \frac7{24}x^6 - \frac{21}{10}x^5 + \frac{35}4 x^4 -20 x^3.
\]

The ascending basis $t = \{ L_{10}(x)_{6,3}, L_{10}(x)_{6,4}, L_{10}(x)_{6,5}, L_{10}(x)_{6,6}\}$.
If $v = \{x^3, x^4, x^5, x^6\}$, then the change of basis matrix 

\begin{center}
$M_{vt} =$
\begin{doublespace}
\noindent\(\left[
\begin{array}{cccc}
-20 & \phantom{-}0 & \phantom{-}0 & 0 \\
\phantom{-}\frac{35}4 & \phantom{-}\frac{35}4 & \phantom{-}0 &0 \\
-\frac{21}{10} &-\frac{21}{10}& -\frac{21}{10}& 0\\
\phantom{-} \frac{7}{24} & \phantom{-}\frac{7}{24} & \phantom{-}\frac{7}{24} & \frac{7}{24} \\
\end{array}
\right]\)
\end{doublespace}
\end{center}

The coefficient function for the elements of this matrix is
\[
m_{i,j} = 
\left\{
\begin{array}{ll} 
\alpha(6, 3+j, 3 - i) & \mbox{if $i \geq j$}\\
\\
0 & \mbox{if $i > i$}\\
\end{array}
\right.
\] where $0 \leq i, j \leq 3$ and $\alpha(6, m, k) = \frac{(-1)^{6-k}}{(6-k)!} \binom{10}{6-k}$. This matrix has the inverse
\begin{center}
$M_{tv} =$
\begin{doublespace}
\noindent\(\left[
\begin{array}{cccc}
-\frac{1}{20} & \phantom{-}0 & \phantom{-}0 & 0 \\
\phantom{-}\frac1{20} & \phantom{-}\frac4{35} & \phantom{-}0 &0 \\
\phantom{-}0 &-\frac{4}{35}& -\frac{10}{21}& 0\\
\phantom{-}0 & \phantom{-}0 & \phantom{-}\frac{10}{21} & \frac{24}{7} \\
\end{array}
\right]\)
\end{doublespace}
\end{center}

For example, from the third column, we have
\begin{align*}
x^5 =& \frac{10}{21}L_{10}(x)_{6,6} - \frac{10}{21}L_{10}(x)_{6,5}\\
=&\frac{10}{21} \left(\frac7{24}x^6 \right)- \frac{10}{21} \left(\frac7{24}x^6 - \frac{21}{10}x^5\right).
\end{align*}
\section{Bernstein Polynomials}

Bernstein polynomials can be defined by
\begin{equation}
b_m^n(x) = \sum\limits_{l = m}^{n} \binom{n}{l}\binom{l}{m} (-1)^{l - m} x ^l
\end{equation} where $m, n \in \mathbb{N}_0$ and $0 \leq m \leq n$. In general, they do not have definite parity and are basis polynomials. 

It follows that the coefficient function for the mapping from Bernstein polynomials to the monomials is given by
\begin{equation} \label{MB}
\alpha(n, m, k) = \binom{n}{n-k}\binom{n-k}{m} (-1)^{n - m-k}
\end{equation} where $0 \leq k \leq n-m$.

A polynomial in Bernstein form~\cite[equation (7)]{farouki} is a parametric vector equation
\begin{equation}
p(t) = \sum\limits_{k=0}^n c_k b_{k}^n(t)
\end{equation} where $c_k$ are $n+1$ B{\'e}zier control points that are vectors with coordinates $(x_k, y_k)$, and $t \in [0,1]$. The basis polynomials in $p(t)$ are Bernstein polynomials. They each have degree $n$, and form the scalar coefficients of the control points.

The change of basis matrix from $\{b_{m}^n(x), b_{m+1}^n(x), \ldots b_{n}^n(x)\}$ to the monomials $\{x^m, x^{m+1}, \ldots, x^n\}$ is a lower triangular matrix with zero entries above the main diagonal. It has entries
$m_{i,j} = \alpha(n, j, n- m-i)$ where $0 \leq i,j \leq n-m$ and $i \geq j$ and $0$ otherwise. The coefficient function, $\alpha$, is given in equation~(\ref{MB}).

For example, if $m = 0$ and $n=3$, the matrix is
\begin{center}
\begin{doublespace}
\noindent\(\left[
\begin{array}{cccc}
\alpha(3,0,3) & 0 & 0 & 0 \\
\alpha(3,0,2) & \alpha(3,1,2) & 0 & 0 \\
\alpha(3,0,1)& \alpha(3,1,1) & \alpha(3,2,1) & 0 \\
\alpha(3,0,0) &\alpha(3,1,0) &\alpha(3,2,0) & \alpha(3,3,0)
\end{array}
\right]\)
$=$
\(\left[
\begin{array}{cccc}
\phantom{-}1 & \phantom{-}0 & \phantom{-}0 & 0 \\
-3& \phantom{-}3 & \phantom{-}0 & 0 \\
\phantom{-}3 & -6 & \phantom{-} 3 & 0 \\
-1 &\phantom{-}3 & -3 & 1
\end{array}
\right]\)
\end{doublespace}
\end{center} and we have
\begin{align*}
b_{0,3}(x) =& -x^3 + 3 x^2 - 3x + 1\\
b_{1,3}(x) =& 3x^3 - 6x^2 + 3x.
\end{align*}

The inverse mapping~\cite[equation (B5)]{mathar} is given by
\begin{equation} \label{invB1}
x^k = \sum_{i=0}^{n-k} \binom{n-k}{i} \frac{1}{\binom{n}{i}} b_{n-i}^n (x)
\end{equation} It expresses $x^k$ in terms of the ascending basis
$\{b_{k}^n(x), b_{k+1}^n(x), \ldots b_{n}^n(x)\}$. This gives the coefficient function
\begin{equation} \label{invB1a}
\alpha(n,m,k) = \frac{\binom{n-m}{k}}{\binom{n}{k}}.
\end{equation}

For example,
\[
x^3 = \alpha(5,3,0) b_{5}^5(x) + \alpha(5, 3, 1) b_{4}^{5}(x) + \alpha(5,3,2) b_{3}^{5}(x).
\]

Mathar~\cite{mathar} did not give a proof of Equation~(\ref{invB1}). It can be verified as follows.
\begin{theorem}
$\alpha(n,m,k) = \frac{\binom{n-m}{k}}{\binom{n}{k}}$ for all $k: 0 \leq k \leq n-m$.
\end{theorem}

\begin{proof}
The proof is by mathematical induction by forming the inverse of the matrix $M_{sv}$ by forward substitution where $s = \{x^m, x^{m+1}, \ldots, x^n\}$ and
$v$ is the ascending basis $\{b_m^n(x), b_{m+1}^n(x), \ldots, b_n^n(x)\}$. Please see Figure~\ref{MDAscNP} above for the form of $M_{sv}$.

For clarity, we rename the coefficient function in equation~(\ref{MB}) to $\alpha_1$. The base case is for $k = n-m$. We have
\begin{align*}
\alpha(n,m,n-m) = & \frac{1}{\alpha_1(n,m,n-m)}\\
=& \frac{1}{\binom{n}{m}}
\end{align*}
From equation~(\ref{invB1a}), we have $\alpha(n,m,n-m) = \frac{1}{\binom{n}{n-m}}$, as required.

The induction hypothesis is that result holds for all $k = n-m-l$ where $0 \leq l < h \leq n-m$.
It suffices to show that 
\[
\sum\limits_{l=0}^h \alpha_1(n, m+l, n-m-h) \alpha(n,m, n-m-l) = 0.
\] We apply Gosper's algorithm using the Mathematica package `fastZeil' version 3.61~\cite{paulezbf} that shows the sum is $0$ provided none of $\{h+m, n, -m+n\}$ is negative and $h \not= 0$. These conditions are satisfied.
The result holds for $k = n-m-l-1$, i.e., $h=l+1$. Hence, by the principle of mathematical induction, it holds for all $k: 0 \leq k \leq n-m$, as required.
\end{proof}

\subsection{The Descending Basis}

The monomial basis $s = \{x^m, x^{m+1}, \ldots, x^n\}$ is unchanged, however the basis of Bernstein polynomials is a descending basis
$v = \{b_{k}^{k}(x), b_{k}^{k+1}(x), \ldots, b_{k}^n(x)\}$. We find the coefficient function $\alpha$ for the entries of the change of basis matrix $M_{vs}$.

This coefficient function is
\begin{equation} \label{adesc}
\alpha(n, m, k) =(-1)^{n- m -k} \frac{\binom{n}{k}}{\binom{n}{m}}
\end{equation} where $0 \leq k \leq n-m$, so that

\begin{equation}
x^k = \sum\limits_{l=0}^{k-m} \alpha( k, m, l) b_m^{k-l} (x).
\end{equation}

For example, when $n =6$ and $m = 3$, the change of basis matrix is 
\begin{center}
\begin{doublespace}
\(
\left[
\begin{array}{cccc}
\alpha(3,3,0) & \alpha(4,3,1) & \alpha(5,3,2) & \alpha(6,3,3) \\
0 & \alpha(4,3,0) & \alpha(5,3,1) & \alpha(6,3,2) \\
0 & 0 & \alpha(5,3,0) & \alpha(6,3,1) \\
0 &0 &0 & \alpha(6,3,0)
\end{array}
\right]
=
\left[
\begin{array}{cccc}
1 & \phantom{-}1 & \phantom{-}1 & \phantom{-}1 \\
0 & -\frac14 & -\frac12 & -\frac34 \\
0 & \phantom{-}0 & \phantom{-}\frac1{10} & \phantom{-}\frac3{10} \\
0 &\phantom{-}0 &\phantom{-}0 & -\frac{1}{20}
\end{array}
\right]\)
\end{doublespace}
\end{center}
so that
\begin{align*}
x^6 =& -\frac1{20} b_{3}^{6}(x) + \frac3{10}b_{3}^{5}(x) - \frac34 b_{3}^{4}(x) + b_{3}^{3}(x)\\
\\
x^5 =& \alpha(5, 3, 0) b_{3}^{5}(x) + \alpha(5, 3, 1) b_{3}^{4}(x) + \alpha(5, 3, 2)b_{3}^{3}(x)\\
=& \frac1{10}b_{3}^{5}(x) - \frac12 b_{3}^{4}(x) + b_{3}^{3}(x).
\end{align*}

\begin{theorem} \label{invdesa}
The coefficient function for the mapping from $\{x^{m}, x^{m + 1}, \dots, x^n\}$ where $0 \leq m \leq n$ to $\{b_{m}^{m}(x), b_{m}^ {m+1}(x), \ldots, b_{m}^{n}(x)\}$ is
\[
\alpha(n, m, k) =(-1)^{n- m -k} \frac{\binom{n}{k}}{\binom{n}{m}}.
\]
\end{theorem}

\begin{proof}
The proof is by mathematical induction and uses back substitution to form the inverse matrix of the matrix with entries $m_{i, j} = \binom{m + j}{m + i} \binom{m + i}{m} (-1)^i$ where $0 \leq i, j \leq n-m$.

The equation has the following form (see Figure~\ref{MDesNP} above):

\begin{gather*}
\resizebox{\textwidth}{!}{%
$\displaystyle
\begin{bmatrix}
\vdots & \cdots & \vdots & \vdots & \vdots & \vdots \\
0 & \cdots & \alpha_1(n-3,m, 0) & \alpha_1(n-2, m, 1) & \alpha_1(n-1,m, 2)& \alpha_1(n ,m, 3)\\
0 & \cdots & 0 & \alpha_1(n-2,m, 0) & \alpha_1(n-1, m, 1)& \alpha_1(n, m ,2)\\
0 & \cdots &0 & 0 & \alpha_1(n-1,m, 0)& \alpha_1(n, m,1)\\
0 & \cdots & 0 & 0 & 0 & \alpha_1(n, m, 0)
\end{bmatrix}
\begin{bmatrix}
\vdots \\
b_3 \\
b_2\\
b_1 \\
b_0
\end{bmatrix}
= 
\begin{bmatrix}
\vdots \\
0 \\
0\\
0\\
1
\end{bmatrix}$}
\end{gather*} where the row and column indices are $0 \leq i, j \leq n - m$ and $\alpha_1$ is defined in equation (\ref{MB}), and $b_i = \alpha(n,m,i)$.

The base case is $b_0 = \frac{1}{\alpha_1(n, m, 0)}$. This is $\frac{1}{\binom{n}{m}(-1)^{n - m}} = (-1)^{n - m - k} \frac{\binom{n}{k}}{\binom{n}{m}}$, as required.

The induction hypothesis is that the result holds for all $l: 0 \leq l < h \leq n - m$. We need to show that

\[
\alpha(n,m,h) = -\frac{\sum\limits_{l=0}^{h-1} \alpha_1(n-l, m, h-l) \; \alpha(n,m,l)}{\alpha_1(n-h,m,0)}.
\]

We apply Gosper's algorithm using the Mathematica package `fastZeil' version 3.61~\cite{paulezbf} to the equivalent sum
\[
\sum\limits_{l=0}^{h} \alpha_1(n-l, m, h-l) \; \alpha(n,m,l).
\]
It evaluates to $0$ provided that $h$ is a natural number and $(-1)^{h+2n} h \binom{n}{m} \not = 0$, i.e., $n \geq m \geq 0$. These conditions are satisfied.
The result holds for $l+1$.
Hence, by the principle of mathematical induction, it holds for all $l: 0 \leq l \leq n-m$, as required.
\end{proof}

\begin{remark}\rm
We give a non-algorithmic proof of the induction step of the proof of Theorem~\ref{invdesa} in appendix~\ref{appfor4}.
\end{remark}

\section{Zernike Radial Polynomials}

The Zernike Radial polynomials are orthogonal polynomials, but not classical ones.
They have definite parity and can be defined by
\begin{equation} \label{defR}
R_n^m(x)= \sum_{k=0}^{\frac{n-m}2}\binom{n-k}{k}\binom{n-2k}{\frac{n-m}{2} - k} (-1)^k x^{n - 2k} 
\end{equation} $ 0 \leq m \leq n$ and $n-m$ is even. The degree of $R_n^m(x)$ is $n$, and $m$ is the least degree of its terms when the polynomial is expressed in the monomial basis.

This gives the following coefficient function for the mapping from the Zernike Radial polynomials to the monomials.

\begin{equation} \label{MR}
\beta(n, m, k) = \binom{n-k}{k}\binom{n-2k}{\frac{n-m}{2} - k} (-1)^k
\end{equation} where $0 \leq k \leq \frac{n-m}2$, and $n-m$ is even.

For the mapping from the monomials, there are two inverse coefficient functions depending on whether the range basis is a descending basis or an ascending one.

An example of a descending polynomial basis is 
\[
v =\{R_3^3(x), R_5^3(x), R_7^3(x), R_9^3(x)\}
\] so that $s = \{x^3, x^5, x^7, x^9\}$, $d = 2$, $m=3$ and $n = 9$.
The change of basis matrix $M_{sv}$ is a $4 \times 4$ upper triangular matrix
\begin{gather*}
\begin{bmatrix}
\beta(3,3,0) & \beta(5,3,1)& \beta(7,3,2) & \beta(9,3,3) \\
\\
0 & \beta(5,3,0)& \beta(7,3,1)& \beta(9,3,2)\\
\\
0 & 0 & \beta(7,3,0)& \beta(9,3,1) \\
\\
0& 0 & 0&\beta(9,3,0)
\end{bmatrix}
=
\begin{bmatrix}
1 & -4 & \phantom{-}10 & -20 \\
\\
0 & \phantom{-}5& -30& \phantom{-}105\\
\\
0 & \phantom{-}0 & \phantom{-}21& -168 \\
\\
0& \phantom{-}0 & \phantom{-}0&\phantom{-}84 
\end{bmatrix}
\end{gather*} 
When $j=2$, the third column is the transpose of the  coordinate vector of $R_7^3(x) =  21 x^7 - 30x^5 + 10x^3$ with respect to $s$, i.e., $(10, -30, 21, 0)^T$.

\smallskip
An example of the inverse mapping is the change of basis matrix from $s = \{1, x^2, x^4, x^6\}$ to $v =\{R_0^0(x), R_2^0(x), R_4^0(x), R_6^0(x)\}$, i.e.,
\begin{gather*}
M_{vs} =
\begin{bmatrix}
1 & \frac12 & \frac13 & \frac14\\
\\
0 & \frac12& \frac12 & \frac{9}{20} \\
\\
0 & 0 & \frac16& \frac14 \\
\\
0& 0 & 0& \frac1{20}
\end{bmatrix}
\end{gather*} where the element $m_{i,j} = \beta(2j, 0, j-i)$ and $0 \leq i \leq j \leq 3$. 
From the third and fourth columns, we have
\begin{align*}
x^4 =& \beta(4, 0, 0) R_4^0(x) + \beta(4, 0, 1) R_2^0(x) + \beta(4, 0, 2)R_0^0(x)\\
=&\frac16 R_4^0(x) + \frac12 R_2^0(x) + \frac13 R_0^0(x)\\
x^6 =& \frac1{20} R_6^0(x) + \frac14 R_4^0(x) + \frac{9}{20}R_2^0(x) + \frac{1}{4}R_0^0(x).
\end{align*}

The following theorem provides the inverse mapping when the range basis is a descending basis. 
\begin{theorem} \label{RMD}
Let 
\[
s = \{x^m, x^{m+2}, \ldots, x^n\} \: \mbox{\rm  and }  v = \{R_{m}^{m}(x), R^{m}_{m+2}(x), \ldots, R^{m}_n(x)\}
\] where $n \geq m \geq 0$ and $n$ and $m$ have the same parity.
Then the coefficient function for the mapping from $s$ to $v$ is
\begin{equation} \label{RM-desc}
\beta(n, m, k) = \frac{n - 2k + 1}{n - k + 1}\frac{\binom{n}{k}}{\binom{n}{\frac{n-m}2}}
\end{equation} where $0 \leq k \leq \frac{n-m}2$.
\end{theorem}

\begin{proof}
The proof is by mathematical induction and uses back substitution to find the inverse of the upper triangular change of basis matrix from the basis $v = \{R_m^m(x), R_{m+2}^m(x), \ldots, R_n^m(x)\}$ to $s=\{x^m , x^{m+2}, \ldots, x^n\}$. Please see Figure~\ref{MDesP} above for the form of the change of basis matrix $M_{sv}$. The base case is when $k=0$.

We have from equation~(\ref{MR}), that $\beta(n, m, 0) = \frac{1}{\binom{n}{\frac{n-m}2}}$, as required.
The induction hypothesis is that the result holds for all $l: 0 \leq l < h \leq \frac{n - m}2$. We need to show that 
\[
\frac{n - 2h + 1}{n - h + 1}\frac{{n \choose h}}{{n \choose {\frac{n-m}2}}} = -\frac{\sum\limits_{l=0}^{h-1} \beta_1(n-2l, m, h-l) \; \beta(n,m,l)}{\beta_1(n-2h, m, 0)}
\] where $\beta_1$ is defined in equation~(\ref{MR}). 

This is equivalent to showing that
\[
\sum\limits_{l=0}^{h} \beta_1(n-2l, m, h-l) \; \beta(n,m,l) = 0.
\]

We apply Gosper's algorithm using the Mathematica package `fastZeil' version 3.61~\cite{paulezbf} to the sum on the left side of the equation. It evaluates to $0$ provided that $h$ is a natural number, none of $\{n, -2h + n \}$ is a negative integer and $h {n \choose{\frac12 (-m+n}} \not= 0$. These conditions are satisfied because $h> 0$ by the induction hypothesis, the largest value of $2h$ is $n-m \geq 0$ and $n-m \leq n$, $n \geq 0$ and $\frac{n-m}2 \geq 0$.

The result holds for $h = l+1$. Hence, by the principle of mathematical induction, it holds for all $l: 0 \leq l \leq \frac{n - m}2$, as required.
\end{proof}

\begin{remark}\rm
We give a non-algorithmic proof of the induction step of the proof of Theorem~\ref{RMD} in appendix~\ref{appfor5}.
\end{remark}

\subsection{The Ascending Basis}
The other inverse mapping we consider is when the range basis is an ascending basis in Zernike Radial polynomials. It is given by the following theorem.

\begin{theorem} \label{beta-RM}
The coefficient function for the mapping
\[ 
\{x^m, x^{m+2}, \ldots, x^n\} \rightarrow \{ R_n^n(x), R_n^{n-2}(x), \ldots , R_n^{m+2}, R_n^m(x)\}
\]
where $n \geq m \geq 0$ and $n$ and $m$ have the same parity is
\begin{equation}
\beta(n, m, k) = (-1)^k \frac{(m+1)_{v-k-1} (n- 2k)}{(v-k)! {{v+m} \choose v}} \label{b4}
\end{equation} where $v = \frac{n-m}2$, $0 \leq k \leq v$,
\end{theorem}

\begin{proof}
The proof is by mathematical induction and it uses forward substitution to find the inverse of the lower triangular change of basis matrix from the basis $\{R_n^n(x), R_{n}^{n-2}(x), \ldots, R_n^{m+2}(x), R_n^m(x)\}$ to $\{x^m , x^{m+2}, \ldots, x^n\}$. Please see Figure~\ref{MDAscP} above for the form of the change of basis matrix. The elements of the inverse matrix are found column-wise from left to right.

The base case has $h=0$ and $k=v-h$. From equation (\ref{b4}), 
\begin{align*}
\beta(n, m, v) =& (-1)^{\frac{n-m}2} \frac{(m+1)_{-1} m}{{{\frac{n+m}2} \choose {\frac{n-m}2}}}\\
=& (-1)^{\frac{n-m}2} \frac{1}{{{\frac{n+m}2} \choose {\frac{n-m}2}}}
\end{align*}

We rename the function $\beta$ in equation (\ref{MR}) to $\beta_1$ to prevent ambiguity. From this equation, the entry in the first row and last column of the change of basis matrix is
\begin{align*}
\beta_1(n,m,\frac{n-m}2) =& {{n- \frac{n-m}2 } \choose {\frac{n-m}2}} {{n-2 \left( \frac{n-m}2 \right) } \choose {\frac{n-m}{2} - {\frac{n-m}2}}} (-1)^{\frac{n-m}2} \\
=& {{\frac{n+m}2 } \choose {\frac{n-m}2}} (-1)^{\frac{n-m}2}
\end{align*}
We know that $\beta_1(n,m,v) \beta(n, m, v) = 1$, and from the above, $\beta(n, m, v) = \frac{1}{\beta_1(n,m,\frac{n-m}2) }$, as required.

The induction hypothesis is that the result holds for all $l: 0 \leq l < h \leq \frac{n - m}2$.
We need to show that
\[
\sum\limits_{l=0}^{h} \beta_1(n, m+2l, v-h) \; \beta(n,m,v-l)= 0.
\] 

We apply Mathematica's {\tt Simplify} function to the summand formula of the equation. The left side becomes
\[
\sum\limits_{l=0}^{h} (-1)^{-h-l-m+n} \frac{2h+m}{(h-l)!} {{m + 2l} \choose l} (m + 2l +1)_{h-l-1}
\]
which evaluates to $0$, as required.

The result holds for $h = l+1$. Hence, by the principle of mathematical induction, it holds for all $l: 0 \leq l \leq \frac{n - m}2$, as required.
\end{proof}

\begin{remark}\rm
We give a non-algorithmic proof of the induction step of the proof of Theorem~\ref{beta-RM} in appendix~\ref{appfor6}.
\end{remark}

The coefficient function $\beta(n, m, k)$ of equation~(\ref{RM-desc}) in Theorem~\ref{beta-RM} provides the mapping from monomials to Zernike Radial polynomials in ascending representation. Specifically, this is
\begin{equation}
x^m = \sum\limits_{k=0}^{\frac{n-m}2} \beta(n, m, k) R_n^{n-2k}(x).
\end{equation}

For example,
\begin{align*}
x^4 =& \beta(8, 4, 0) R_8^8(x) + \beta(8, 4, 1) R_8^6(x) + \beta(8, 4, 2) R_8^4(x)\\
=& \frac{4}3 x^8 - \frac25 (8x^8 - 7x^6) + \frac1{15} (28 x^8 - 42 x^6 + 15 x^4)\\
\\
x^4 =& \beta(6, 4, 0) R_6^6(x) + \beta(6, 4, 1) R_6^4(x)\\
=& \frac65 x^6 - \frac15 (6 x^6 - 5x^4)\\
\\
x^3 =& \beta(9, 3, 0) R_9^9(x) + \beta(9, 3, 1) R_9^7(x) + \beta(9, 3, 2) R_9^5(x) + \beta(9, 3, 3) R_9^3(x)\\
=& \frac32 R_9^9(x) - \frac7{10} R_9^7(x) + \frac14 R_9^5(x) - \frac1{20} R_9^3(x).
\end{align*}

\part{Composing Coefficient Functions}

From Theorem~\ref{tri-inv}, for every $v,t \in S$, each change of basis matrix $M_{vt}$ has one of the forms in figures Figures~\ref{MDesNP} -- \ref{MDAscP}, depending on whether which of the four kinds of bases are those of the bases in $S$.

If change of basis matrices satisfy the change of basis equation
\begin{equation} \label{CoB}
M_{vt} = M_{vr} M_{r t}
\end{equation} then $M_{vr}$ and $M_{rt}$ could be upper and lower triangular matrices, or vice versa. Their product is not necessarily a triangular matrix.

Given the coefficient functions for $M_{vr}$ and $M_{r t}$, we consider defining the coefficient function for $M_{vt}$ from them.
In general, if $v, r$ and $t$ do not have definite parity, then we call $\alpha$ the coefficient of $M_{vt}$, $\alpha_1$ the coefficient function of $M_{vr}$, and $\alpha_2$ that of $M_{rt}$. If $v, r$ and $t$ have definite parity, we call $\beta$ the coefficient of $M_{vt}$, $\beta_1$ the coefficient function of $M_{vr}$, and $\beta_2$ that of $M_{rt}$. 

\section{Descending Bases without Definite Parity}
The composition of $\alpha_1$ and $\alpha_2$ to form $\alpha$ is based on equation (6) of Wolfram~\cite{wolfram2107} which is for the special case where $m=0$:
\begin{equation} \label{DND}
\alpha(n,m,k) = \sum\limits_{v=0}^k \alpha_1(n-v,m,k-v) \; \alpha_2(n,m,v)
\end{equation} where $0 \leq k \leq n-m$.

This is the dot product of the row with index $n-m-k$ of $M_{vr}$ with the last column of $M_{rt}$. The row and column indices of these matrices are $0 \leq i,j \leq n-m$.

Examples of this include compositions of coefficient functions for mappings between bases of classical orthogonal polynomials that do not have definite parity. For these polynomials, $m=0$. They include Chebyshev polynomials of the third kind $V_n(x)$, generalized Laguerre polynomials $L^{(\alpha)}_n(x)$, and shifted Legendre polynomials $P^{\ast}_n(x)$. Wolfram~\cite{wolfram2108} gives coefficient functions for mappings between these polynomials and the monomials.

A specific example is where $t = \{P^{\ast}_4(x)_{4,4}, P^{\ast}_5(x)_{5,4}, P^{\ast}_6(x)_{6,4}, P^{\ast}_7(x)_{7,4}\}$ of truncated shifted Legendre polynomials, $r =\{x^4, x^5, x^6, x^7\}$ and 
\[v = \{b_4^4(x), b_4^5(x), b_4^6(x), b_4^7(x)\}
\] which is a descending basis of Bernstein polynomials.

The coefficient function for $M_{rt}$ is based on Wolfram\cite[equation (36)]{wolfram2108}:
\[
\alpha_2(n,m, k) = 2^{-k} \sum\limits_{v=0}^{\lfloor \frac{k}2 \rfloor} {{n-2v} \choose {k-2v}} {{2n-2v}\choose n} {n \choose v} (-1)^{k-v}.
\]
where $0 \leq k \leq n-m$.
The coefficient function for $M_{vr}$ is from equation~(\ref{adesc}).
\[
\alpha_1(n, m, k) =(-1)^{n- m -k} \frac{{n \choose k}}{{n \choose m}}
\] where $0 \leq k \leq n-m$.
Applying equation~(\ref{DND}) gives the coefficient function for $M_{vt}$. Equation~(\ref{mDND}) for its elements is
\[
m_{i,j} = 
\left\{
\begin{array}{ll} 
\alpha(m+j, m, j-i ) & \mbox{if $j \geq i$}\\
0 & \mbox{if $j< i$}
\end{array}
\right.
\] where $0 \leq i,j \leq n-m$.

We obtain,
\begin{center}
\begin{doublespace}
$M_{vt} =$
\noindent\(\left[
\begin{array}{cccc}
70 & -378 & 1302 & -3498 \\
0 & -\frac{252}{5} & \frac{924}{5} & -\frac{2904}{5} \\
0 & 0 & \frac{308}{5} & -\frac{572}{5} \\
0 & 0 & 0 & -\frac{3432}{35} \\
\end{array}
\right]\)
\end{doublespace}
\end{center}

From the third column, we have
\begin{align*}
P_6^{\ast}(x)_{6, 4}
=& \frac{308}5 b^6_4(x) +\frac{924}5 b^5_4(x) +1302 b_4^4(x)\\
=&(924 x^6 -1848 x^5 + 924 x^4) + (-924 x^5 + 924 x^4) +1302 x^4\\
=& 924 x^6 -2772 x^5 + 3150 x^4.
\end{align*}

\section{Ascending Bases without Definite Parity}

The composition is found the dot product of the row with index $n-m-k$ of $M_{vr}$ with the first column of $M_{rt}$:
\begin{equation} \label{AND}
\alpha(n,m,k) = \sum\limits_{v=0}^{n-m-k} \alpha_1(n, m+v,k) \; \alpha_2(n,m,n-m-v)
\end{equation} where $0 \leq k \leq n-m$.

An example of this is the mapping from a basis of truncated Chebyshev polynomials of the third kind $t = \{V_6(x)_{6, 3}, V_6(x)_{6,4}, V_6(x)_{6,5}, V_6(x)_{6,6}\}$ to $v = \{b_3^6(x), b_4^6(x), b_5^6(x), b_6^6(x)\}$, where $r = \{x^3, x^4, x^5, x^6\}$, $n=6$ and $m=3$.

We have 
\begin{align*}
V_6(x) =& 64 x^6 - 32 x^5 - 80 x^4 + 32x^3 + 24 x^2 - 6x - 1\\
=& U_6(x) - U_5(x).
\end{align*}

A coefficient function for the mapping from Chebyshev polynomials of the third kind to the monomials is given by Wolfram~\cite[equation (23)]{wolfram2108}:
\[
\alpha(n,k) = \frac{2^{n}}{{{2n} \choose n}} \frac{(-1)^{k}}{ (n-k)!} \sum_{l=0}^{k} 2^l \frac{(1 +n)_{n-l}(\frac{1}{2} - n)_{l}}{(k-l)! \; l!}. 
\]

From equation~(\ref{NDtruncalpha}), we have $\alpha_2(n, m, k) = \alpha(n, k)$ where $0 \leq k \leq n -m$. from equation~(\ref{invB1a}) we have
\[
\alpha_1(n,m,k) = \frac{{{n-m} \choose k}}{{n \choose k}}
\] for the mapping from $r$ to $v$.

Applying equation~(\ref{AND}) gives the required coefficient function,
\[
\alpha(n,m,k) = \sum\limits_{v=0}^{n-m-k} \alpha_1(n, m+v,k) \; \alpha_2(n,m,n-m-v)
\] and we obtain the change of basis matrix by using equation~(\ref{mAND}) for its elements: 
\[
m_{i,j} = 
\left\{
\begin{array}{ll} 
\alpha(m+j, m, j-i ) & \mbox{if $j \geq i$}\\
0 & \mbox{if $j< i$}
\end{array}
\right.
\] where $0 \leq i,j \leq n-m$.

\begin{center}
\begin{doublespace}
$M_{vt} =$
\noindent\(\left[
\begin{array}{cccc}
\frac{8}{5} & 0 & 0 & 0 \\
\frac{16}{15} & -\frac{16}{3} & 0 & 0 \\
-16 & -32 & -\frac{16}{3} & 0 \\
-16 & -48 & 32 & 64 \\
\end{array}
\right]\)
\end{doublespace}
\end{center}

From the second column,
\begin{align*}
V_6(x)_{6, 4} 
=& -48 b^6_6(x) - 32 b^6_5(x) - \frac{16}3 b^6_4(x)\\
=& -48 x^6 -32(-6x^6 + 6x^5) - \frac{16}3 (15 x^6 -30 x^5 +15 x^4)\\
=& 64 x^6 - 32 x^5 - 80 x^4.
\end{align*}

\section{Descending Bases with Definite Parity}
The composition of $\beta_1$ and $\beta_2$ to form $\beta$ is based on Wolfram~\cite{wolfram2107,wolfram2108} which is for the special case where $m=0$:
\begin{equation} \label{DDD-comp}
\beta(n,m,k) = \sum\limits_{v=0}^{\frac{n-m}2} \beta_1(n -2v, m, k-v) \; \beta_2(n,m,v)
\end{equation} where $0 \leq k \leq \frac{n-m}2$.

For example, let $t =\{R_3^3(x), R_5^3(x), R_7^3(x), R_9^3(x)\}$, $r = \{x^3, x^5, x^7, x^9\}$ and $v = \{T_3(x)_{3,3}, T_5(x)_{5, 3}, T_7(x)_{7,3}, T_9(x)_{9,3}\}$. These are all bases of the same vector space.

From equation~(\ref{MR}), 
\begin{gather*}
M_{rt} =
\begin{bmatrix}
1 & -4 & \phantom{-}10& -20 \\
\\
0 & \phantom{-}5 & -30& \phantom{-}105\\
\\
0 & 0 & \phantom{-}21& -168 \\
\\
0& 0 & 0& \phantom{-}84
\end{bmatrix}
\end{gather*}

From \cite[equation (2.14) ]{handscomb}, we have
\begin{gather*}
M_{vr} =
\begin{bmatrix}
\frac1{4} & \frac5{16} & \frac{21}{64} & \frac{21}{64} \\
\\
0 & \frac1{16}& \frac7{64} & \frac9{64}\\
\\
0 & 0 & \frac1{64}& \frac{9}{256} \\
\\
0& 0 & 0&\frac1{256} 
\end{bmatrix}
\end{gather*}

Therefore,
\begin{gather*}
M_{vt} =
\begin{bmatrix}
\frac14 & \frac{9}{16} & \frac{1}{64}& \frac{1}{4} \\
\\
0 & \frac{5}{16} & \frac{27}{64} & 0\\
\\
0 & 0 & \frac{21}{64}& \frac{21}{64} \\
\\
0& 0 & 0& \frac{21}{64}
\end{bmatrix}
\end{gather*}

From the fourth column of $M_{vt}$, we have
\begin{align*}
R_9^3 (x) =& \frac{21}{64} T_9(x)_{9,3} + \frac{21}{64}T_7(x)_{7,3} + \frac1{4}T_3(x)_{3,3}\\
=& \frac{21}{64} (256 x^9 - 576 x^7 + 432 x^5 - 120x^3) + \frac{21}{64}(64x^7 - 112 x^5 + 56 x^3) +  x^3\\
=& 84x^9 - 168 x^7 + 105 x^5 - 20 x^3.
\end{align*}

Other examples of this composition are for coefficient functions for mappings between bases of classical orthogonal polynomials that have definite parity. For these polynomials, $m=0$ or $m=1$. They include Chebyshev polynomials of the second kind $U_n(x)$, Legendre polynomials $P_n(x)$, and Hermite polynomials $H_n(x)$. Wolfram~\cite{wolfram2108} gives coefficient functions for mappings between these polynomials and the monomials.

\section{Ascending Bases with Definite Parity}
In this case, the composition is similar to equation~(\ref{AND})
\begin{equation} \label{AD}
\beta(n,m,k) = \sum\limits_{v=0}^{\frac{n-m}2-k} \beta_1(n, m+2v,k) \; \beta_2(n,m,\frac{n-m}2 -v)
\end{equation} where $0 \leq k \leq \frac{n-m}2$.

An example is  $t= \{R_9^3(x), R_9^5(x), R_9^7(x), R_9^9(x) \}$,  $r = \{x^3, x^5, x^7, x^9\}$ and $v = 
\{H_9(x)_{9,3},H_9(x)_{9,5},H_9(x)_{9,7}, H_9(x)_{9,9}\}$, i.e., an ascending basis of truncated Physicist's Hermite polynomials where $n =9$.

The coefficient function for $M_{rt}$ is given by equation~(\ref{MR})
\[
\beta_2(n, m, k) = {{n-k} \choose k}{{n-2k} \choose {\frac{n-m}{2} - k}} (-1)^k 
\] where $0 \leq k \leq \frac{n-m}2$

The coefficient function for $M_{vr}$ is derived from Koornwinder et al.~\cite[equation 18.5.13]{dlmf} and Lemma~\ref{truncband}:
\[
\beta_1(n, m, k) = \left\{
\begin{array}{ll} 
(-1)^{k} 2^{2k-n} \frac{k!(n-2k)!}{n!} & \mbox{if $k = \frac{n-m}2$}\\
\\
-\beta_1(n,m,k+1) & \mbox{if $k = \frac{n-m}2 -1$}\\
\\
0 & \mbox{otherwise.}
\end{array}
\right.
\]

Applying equation~(\ref{AD}) gives the required coefficient function $\beta(n, m, k)$. The elements of $M_{vt}$ are given by 
the equation~(\ref{mAD}):
\[
m_{i,j} = 
\left\{
\begin{array}{ll} 
\beta(n, m+2j, l- i ) & \mbox{if $i \geq j$}\\
0 & \mbox{if $i < j$}
\end{array}
\right.
\] where $0 \leq i,j \leq 3$.
This gives
\begin{center}
\begin{doublespace}
$M_{vt} =$
\(\left[
\begin{array}{cccc}
\frac{1}{4032} & 0 & 0 & 0 \\
\frac{31}{16128} & \frac{1}{2304} & 0 & 0 \\
\frac{37}{2304} & \frac{13}{2304} & \frac{1}{1152} & 0 \\
\frac{7}{48} & \frac{37}{576} & \frac{77}{4608} & \frac{1}{512} \\
\end{array}
\right]\)
\end{doublespace}
\end{center}
From the second column of $M_{vt}$, we have $R_9^5 (x)$
\begin{align*}
=& \frac{37}{576} H_9(x)_{9,9} + \frac{13}{2304}H_9(x)_{9,7} + \frac1{2304}H_9(x)_{9,5}\\
=& \frac{37}{576} 512 x^9 + \frac{13}{2304}(512 x^9 - 9216x^7) + \frac1{2304}(512x^9 - 9216x^7 + 48384x^5)\\
=& 36 x^9 - 56 x^7 + 21 x^5.
\end{align*}

\section{Ascending and Descending Bases}
If $M_{vr}$ is a change of basis matrix for ascending bases and $M_{rt}$ is one for descending bases, or vice versa, then $r$ is the subset of the monomials that spans the same vector space as $v$ and $t$. In general, $M_{vt}$ is not a triangular matrix.

If $v,r$ and $t$ do not have definite parity, and $v$ is an ascending basis then the composition of $\alpha_1$ and $\alpha_2$ is given by
\begin{equation} \label{compDA}
\alpha(n,l,k) = \sum\limits_{v=0}^{\min\{n-m-k, l-m\}} \alpha_1(n, m+v,k) \; \alpha_2(l,m,l-m-v)
\end{equation} where $0 \leq j, k \leq n-m$ and $l = m+j$. This is the dot product of the row with index $i = n-m-k$ of $M_{vr}$ with the column with index $j$ of $M_{rt}$. The elements of $M_{vt}$ are 
\begin{equation} \label{mADN}
m_{i, j} = \alpha(n, m+j, n-m-i).
\end{equation}
\begin{figure}[ht]
\resizebox{\textwidth}{!}{%
$\displaystyle
\left[
\begin{array}{rrrrrr}
\alpha(n,m, n-m) & \cdots & \alpha(n, n-3, n-m) & \alpha(n,n-2, n-m) & \alpha(n, n-1, n-m) & \alpha(n,n, n-m)\\
\vdots & \vdots & \vdots & \vdots & \vdots & \vdots \\
\alpha(n,m, 3) & \cdots & \alpha(n, n-3, 3) & \alpha(n, n-2, 1) & \alpha(n, n-1, 3)& \alpha(n ,n, 3)\\
\alpha(n,m, 2) &\cdots & \alpha(n, n-3, 2) & \alpha(n, n-2, 2) & \alpha(n, n-1, 2)& \alpha(n, n, 2)\\
\alpha(n,m, 1) & \cdots & \alpha(n, n-3, 1) & \alpha(n, n-2, 1) & \alpha(n, n-1, 1)& \alpha(n, n, 1)\\
\alpha(n,m, 0) & \cdots &\alpha(n, n-3, 0) & \alpha(n, n-2, 0) & \alpha(n, n-1, 0) & \alpha(n, n, 0)
\end{array} 
\right]
$}
\caption{$M_{vt}$ for a Descending to Ascending Basis Without Definite Parity} \label{MDesAscNP}
\end{figure}
\medskip

An example of this form of coefficient function is discussed in detail by Farouki~\cite{farouki2000} for the mapping between shifted Legendre polynomials and an ascending basis of Bernstein polynomials. We consider this mapping in \S \ref{BLst} below.

If $v,r$ and $t$ do not have definite parity, and $v$ is a descending basis then the composition of $\alpha_1$ and $\alpha_2$ is given by
\begin{equation} \label{compAD}
\alpha(l, m, k) = \sum\limits_{v=0}^{\min\{k, n-l\}} \alpha_1(n-v, m,k-v) \; \alpha_2(n, l, v )
\end{equation} where $0 \leq j, k \leq n-m$ and $l = m+j$. This is the dot product of the row with index $i = n-m-k$ of $M_{vr}$ with the column with index $j$ of $M_{rt}$. The elements of $M_{vt}$ are 
\begin{equation}\label{mDAN}
m_{i, j} = \alpha(m+j, m, k).
\end{equation}
\begin{figure}[ht]
\resizebox{\textwidth}{!}{%
$\displaystyle
\left[
\begin{array}{rrrrrr}
\alpha(m,m, n-m) & \cdots & \alpha(n-3, m, n-m) & \alpha(n-2, m, n-m) & \alpha(n-1, m, n-m) & \alpha(n,m, n-m)\\
\vdots & \vdots & \vdots & \vdots & \vdots & \vdots \\
\alpha(m,m, 3) & \cdots & \alpha(n-3, m, 3) & \alpha(n-2, m, 1) & \alpha(n-1, m, 3)& \alpha(n ,m, 3)\\
\alpha(m,m, 2) &\cdots & \alpha(n-3, m, 2) & \alpha(n-2, m, 2) & \alpha(n-1, m, 2)& \alpha(n, m, 2)\\
\alpha(m,m, 1) & \cdots & \alpha(n-3, m, 1) & \alpha(n-2, m, 1) & \alpha(n-1, m, 1)& \alpha(n, m, 1)\\
\alpha(m,m, 0) & \cdots &\alpha(n-3, m, 0) & \alpha(n-2, m, 0) & \alpha(n-1, m, 0) & \alpha(n, m, 0)
\end{array} 
\right]
$}
\caption{$M_{vt}$ for an Ascending to Descending Basis Without Definite Parity} \label{MAscDesNP}
\end{figure}
\medskip

If these bases have definite parity, and $v$ is an ascending basis this equation becomes
\begin{equation} \label{compDAD}
\beta(n,l,k) = \sum\limits_{v=0}^{\min\{\frac{n-m}2 -k, \frac{l-m}2\}} \beta_1(n, m+2v,k) \; \beta_2(l,m, \frac{l-m}2-v)
\end{equation} where $0 \leq j, k \leq \frac{n-m}2$ and $l= m + 2j$. This is the dot product of the row with index $i = \frac{n-m}2-k$ of $M_{vr}$ with the column with index $j$ of $M_{rt}$. The elements of $M_{vt}$ are 
\begin{equation} \label{mADD}
m_{i, j} = \beta(n, m+2j, k).
\end{equation}
\begin{figure}[ht]
\resizebox{\textwidth}{!}{%
$\displaystyle
\left[
\begin{array}{rrrrrr}
\beta(n,m, \frac{n-m}2) & \cdots & \beta(n, n-6, \frac{n-m}2) & \beta(n,n-4, \frac{n-m}2) & \beta(n, n-2, \frac{n-m}2) & \beta(n,n, \frac{n-m}2)\\
\vdots & \vdots & \vdots & \vdots & \vdots & \vdots \\
\beta(n,m, 3) & \cdots & \beta(n, n-6, 3) & \beta(n, n-4, 1) & \beta(n, n-2, 3)& \beta(n ,n, 3)\\
\beta(n,m, 2) &\cdots & \beta(n, n-6, 2) & \beta(n, n-4, 2) & \beta(n, n-2, 2)& \beta(n, n, 2)\\
\beta(n,m, 1) & \cdots & \beta(n, n-6, 1) & \beta(n, n-4, 1) & \beta(n, n-2, 1)& \beta(n, n, 1)\\
\beta(n,m, 0) & \cdots &\beta(n, n-6, 0) & \beta(n, n-4, 0) & \beta(n, n-2, 0) & \beta(n, n, 0)
\end{array} 
\right]
$}
\caption{$M_{vt}$ for a Descending to Ascending Basis With Definite Parity} \label{MDesAscP}
\end{figure}
\medskip

If $v,r$ and $t$ have definite parity, and $v$ is a descending basis then 
\begin{equation} \label{ADP}
\beta(l, m, k) = \sum\limits_{v=0}^{\min\{k, \frac{n-l}2\}} \beta_1(n-2v, m,k-v) \; \beta_2(n,l, v )
\end{equation} where $0 \leq j, k \leq \frac{n-m}2$ and $l = m+2j$. This is the dot product of the row with index $i = \frac{n-m}2-k$ of $M_{vr}$ with the column with index $j$ of $M_{rt}$. The elements of $M_{vt}$ are 
\begin{equation} \label{mDAP}
m_{i, j} = \beta(m+2j, m, k).
\end{equation}
\begin{figure}[H]
\resizebox{\textwidth}{!}{%
$\displaystyle
\left[
\begin{array}{rrrrrr}
\beta(m,m, \frac{n-m}2) & \cdots & \beta(n-6, m, \frac{n-m}2) & \beta(n-4, m, \frac{n-m}2) & \beta(n-2, m, \frac{n-m}2) & \beta(n,m, \frac{n-m}2)\\
\vdots & \vdots & \vdots & \vdots & \vdots & \vdots \\
\beta(m,m, 3) & \cdots & \beta(n-6, m, 3) & \beta(n-4, m, 1) & \beta(n-2, m, 3)& \beta(n ,m, 3)\\
\beta(m,m, 2) &\cdots & \beta(n-6, m, 2) & \beta(n-4, m, 2) & \beta(n-2, m, 2)& \beta(n, m, 2)\\
\beta(m,m, 1) & \cdots & \beta(n-6, m, 1) & \beta(n-4, m, 1) & \beta(n-2, m, 1)& \beta(n, m, 1)\\
\beta(m,m, 0) & \cdots &\beta(n-6, m, 0) & \beta(n-4, m, 0) & \beta(n-2, m, 0) & \beta(n, m, 0)
\end{array} 
\right]
$}
\caption{$M_{vt}$ for an Ascending to Descending Basis With Definite Parity} \label{MAscDesP}
\end{figure}

The mapping from truncated Chebyshev polynomials of the first kind to Chebyshev polynomials of the first kind provides an example.
Suppose that $t =\{T_7(x)_{7,1}, T_7(x)_{7, 3}, T_7(x)_{7,5}, T_7(x)_{7,7}\}$ and $v = \{T_1(x), T_3(x), T_5(x), T_7(x)\}$. It follows that 
\begin{center}
\begin{doublespace}
\( M_{vt} = \left[
\begin{array}{cccc}
0 & 7 & -35 & 35 \\
0 & 0 & -14 & 21 \\
0 & 0 & 0 & 7 \\
1 & 1 & 1 & 1 \\
\end{array}
\right]\)
\end{doublespace}
\end{center}
From the third column, we have,
\begin{align*}
T_7(x)_{7, 5} =& T_7(x) - 14 T_3(x) - 35 T_1(x)\\
=& (64x^7 -112 x^5 +56 x^3 -7x) -14 (4x^3 - 3x) -35x\\
=& 64x^7 -112 x^5.
\end{align*}

The coefficient functions $\beta_1$ and $\beta_2$ were summarized in Wolfram~\cite{wolfram2107} from Abramowitz and Stegun~\cite[equations 22.4.4 and 22.3.6]{as} and Mason and Handscomb~\cite[equation (2.14)]{handscomb}. In this example, $n=7$, $m=1$ and $0 \leq i, j \leq 3$. For $\beta_2$ for the mapping from $t$ to $\{x, x^3, x^5, x^7\}$, the function has no dependence on $m$. From equation~(\ref{ADP}), we have

\begin{equation} 
\beta(l, m, k) = \sum\limits_{v=0}^{\min\{\frac{n-l}2, k\}} (-1)^v \frac{n}{n-v} {{n-v} \choose k} {k \choose v}
\end{equation} where $0 \leq j, k \leq \frac{n-m}2$ and $n$ is odd. Let $V = \min\{\frac{n-l}2, k\}$ where $l = m+2j$. By applying Gosper's algorithm, we can find that 
\begin{equation}
\beta(l, m, k) =
\left\{
\begin{array}{ll} 
1 & \mbox{if $k = 0$}\\
0 & \mbox{if $k = V$ and $k > 0$}\\
(-1)^V n ( \frac1{V-n} - \frac{V}{k^2 - kn} ) {{n-V} \choose k} {k \choose V} & \mbox{if $k \not=V$}.
\end{array}
\right.
\end{equation}
\section{Summary}
Table \ref{CoBM} summarises eight kinds of products of change of basis matrices. They depend on whether the domain or range basis is an ascending or descending basis, and the parity of these bases. 
In the table, $U$ means an upper-triangular matrix, $L$ means a lower-triangular matrix, $L$, an $M$ means an invertible matrix.

\begin{table}[H]
\centering
\resizebox{\textwidth}{!}{%
\begin{tabular}{|l|l|l|c|c|c|c|}\hline 
{\bf Domain} & {\bf Range} & {\bf Parity} & {\bf Matrix} & {\bf Figure} & {\bf Elements} & {\bf Composition}\\ \hline
Descending & Descending & Not Definite & U& \ref{MDesNP} & (\ref{mDND}) & (\ref{DND}) \\ \hline
Ascending & Ascending & Not Definite & L& \ref{MDAscNP} & (\ref{mAND}) & (\ref{AND}) \\ \hline
Descending & Descending & Definite & U& \ref{MDesP} & (\ref{mDD}) & (\ref{DDD-comp}) \\ \hline
Ascending & Ascending & Definite & L & \ref{MDAscP} & (\ref{mAD}) & (\ref{AD}) \\ \hline
Descending & Ascending & Not Definite & M& \ref{MDesAscNP} & (\ref{mADN}) & (\ref{compDA}) \\ \hline
Ascending & Descending & Not Definite & M & \ref{MAscDesNP} & (\ref{mDAN}) & (\ref{compAD}) \\ \hline
Descending & Ascending & Definite & M& \ref{MDesAscP} & (\ref{mADD}) & (\ref{compDAD}) \\ \hline
Ascending & Descending & Definite & M& \ref{MAscDesP} & (\ref{mDAP}) & (\ref{ADP}) \\ \hline
\end{tabular}

\caption{Products of Change of Basis Matrices and Compositions of Coefficient Functions} \label{CoBM}

}
\end{table}

\section{Examples with Ascending and Descending Bases} 
\label{BLst}

We consider two examples from the literature. The first has a mapping from a descending basis to an ascending basis. The second has two descending bases.

\subsection{Shifted Legendre Polynomials to Bernstein Polynomials}
Farouki~\cite{farouki2000} gave results for change of basis matrices between shifted Legendre polynomials and an ascending basis of Bernstein polynomials using a different method.

Instead, by multiplying change of basis matrices, an example where $n = 5$ is
\begin{doublespace}
\[\left[
\begin{array}{cccccc}
1 & 0 & 0 & 0 & 0 & 0 \\
1 & \frac{1}{5} & 0 & 0 & 0 & 0 \\
1 & \frac{2}{5} & \frac{1}{10} & 0 & 0 & 0 \\
1 & \frac{3}{5} & \frac{3}{10} & \frac{1}{10} & 0 & 0 \\
1 & \frac{4}{5} & \frac{3}{5} & \frac{2}{5} & \frac{1}{5} & 0 \\
1 & 1 & 1 & 1 & 1 & 1 \\
\end{array}
\right] \left[
\begin{array}{cccccc}
1 & -1 & 1 & -1 & 1 & -1 \\
0 & 2 & -6 & 12 & -20 & 30 \\
0 & 0 & 6 & -30 & 90 & -210 \\
0 & 0 & 0 & 20 & -140 & 560 \\
0 & 0 & 0 & 0 & 70 & -630 \\
0 & 0 & 0 & 0 & 0 & 252 \\
\end{array}
\right] =
\]
\[
\left[
\begin{array}{cccccc}
1 & -1 & 1 & -1 & 1 & -1 \\
1 & -\frac{3}{5} & -\frac{1}{5} & \frac{7}{5} & -3 & 5 \\
1 & -\frac{1}{5} & -\frac{4}{5} & \frac{4}{5} & 2 & -10 \\
1 & \frac{1}{5} & -\frac{4}{5} & -\frac{4}{5} & 2 & 10 \\
1 & \frac{3}{5} & -\frac{1}{5} & -\frac{7}{5} & -3 & -5 \\
1 & 1 & 1 & 1 & 1 & 1 \\
\end{array}
\right]\]
\end{doublespace}

The unsolved problem discussed there~\cite{farouki2000} is to find a ``closed form'' representation for the elements of the change of basis matrix and its inverse in terms of $j$, $k$ and $n$. Equation (20) of Farouki~\cite{farouki2000}, for the elements of the change of basis matrix has a sum over three binomials:
\[
\frac{1}{{n \choose j}} \sum_{ i = \max \{0, j+k-n\}}^{\min\{j,k\}} (-1)^{k+i} {k \choose i} {k \choose i} {{n-k} \choose {j-i}}.
\]

We show using the techniques above that the corresponding coefficient function can be expressed as an indefinite hypergeometric sum in equation (\ref{alpha1}). We assume that ``indefinite'' means that the lower bound of the sum is a known element of $\mathbb{N}_0$ and the upper bound can be a variable or other expression.
From Lemma~\ref{closed-alpha}, Gosper's Algorithm~\cite{paulezbf} does not find a closed form for this sum which provides a solution to one part of the open problem. This assumes that if the sum has a required closed form expression, the expression is one that can be  produced by Gosper's Algorithm.

A hypergeometric series for this coefficient function was also given~\cite{farouki2000}, however it is only defined in the upper skew triangle. We give another hypergeometric series for the remaining elements and derive a hypergeometric series that is more general simpler than them.

\smallskip

The shifted Legendre polynomials can be defined~\cite[equation 18.7.10]{dlmf} by
\[
P^{\ast}_n(x) = P_n(2x -1)
\] where $P_n(x)$ is a Legendre polynomial and $n \geq 0$. The shifted Legendre polynomials do not have definite parity.

The coefficient function from the shifted Legendre polynomials to the monomials is given by
\begin{equation} \label{basealpha}
\alpha(n, k) = 2^{-k} \sum\limits_{v=0}^{\lfloor \frac{k}2 \rfloor} {{n-2v} \choose {k-2v}} {{2n-2v}\choose n} {n \choose v} (-1)^{k-v}.
\end{equation}
This is equation (36) of Wolfram~\cite{wolfram2108}. Since the minimum degree of every shifted Legendre polynomial is $0$, the second parameter has been omitted. In this context, we use $\alpha(n, 0,k)$. 
\begin{lemma}
The coefficient function of equation~(\ref{basealpha}) is equal to
\begin{equation} \label{alphasimplified}
\alpha(n, 0,k) = (-1)^k \frac{(2n -k)!}{k! ((n-k)!)^2}.
\end{equation}
\end{lemma}

\begin{proof}
We use Zeilberger's algorithm with equation~(\ref{basealpha}) to derive recurrences that this sum satisfies, and then verify that equation~(\ref{alphasimplified}) is the solution of the recurrences.

In the case that $k$ is even, 
we apply Zeilberger's algorithm using the Mathematica package `fastZeil' version 3.61~\cite{paulezbf} to the right-hand side of equation~(\ref{basealpha}) with $k$ replaced by $2l$. This gives
\begin{equation} \label{receven}
-2(-1 + 2l -2n)(-1 + l -n)\mbox{ SUM}[n] + (-1 + 2l -n)^2 \mbox{ SUM}[1+n] = 0
\end{equation} provided that $l$ is a natural number and none of $\{-2(l -n), n , -2l +n\}$ is a negative integer.
We have that $-2(l-n) \geq 0$ because $n \geq l$, and $n \geq 2l$ because $n \geq k$.

The initial condition is $\mbox{SUM}[0]$ which is $\alpha(0, 0, 0) = 1$. This is correct because $P^{\ast}_0(x) = 1$ and $P^{\ast}_0(x) = \alpha(0,0,0) . 1$.
Substituting $\alpha(n, 0, 2l)$ from equation~(\ref{alphasimplified}) for $\mbox{SUM}[n]$ in equation~(\ref{receven}) simplifies to $0$.

If $k$ is odd, let $k = 2l +1$. 
Applying Zeilberger's algorithm to the right-hand side of equation~(\ref{basealpha}) with $k$ replaced by $2l +1$ gives
\begin{equation} \label{recodd}
-2(-1 + 2l -2n)( l -n)\mbox{ SUM}[n] + (2l-n)^2 \mbox{ SUM}[1+n] = 0
\end{equation} provided that $l$ is a natural number and none of $\{-2(l -n), n , -2l +n\}$ is a negative integer. These conditions are satisfied. The initial condition is $\mbox{SUM}[1]$ which is $\alpha(1, 0, 1) = -1$. This is correct because $P^{\ast}_1(x) = 2 x -1$ and $P^{\ast}_1(x) = \alpha(1,0,0) . x + \alpha(1,0,1) . 1$.
Substituting $\alpha(n, 0, 2l+1)$ from equation~(\ref{alphasimplified}) for $\mbox{SUM}[n]$ in equation~(\ref{recodd}) simplifies to $0$, as required.
\end{proof}

The coefficient function in equation~(\ref{alphasimplified}) is for the mapping from shifted Legendre polynomials to the monomials. The shifted Legendre polynomials form the descending basis
\[
\{P^{\ast}_0(x), P^{\ast}_1(x), \ldots, P^{\ast}_n(x) \}
\] where $n \geq 0$.

The ascending basis of Bernstein polynomials in this example has the form
\[
\{b_0^n(x), b_1^n(x), \ldots, b_n^n(x)\}.
\] 
The coefficient function from the monomials to these Bernstein polynomials is given in equation~(\ref{invB1}).

These coefficient functions can be composed to give the mapping from the descending basis of shifted Legendre polynomials to the ascending basis of Bernstein polynomials by applying equation~(\ref{compDA}) to give
\begin{equation} \label{alpha1}
\alpha(n, j, k) = \frac{(-1)^j}{{n \choose k}} \sum\limits_{v=0}^{\min\{n-k,j\}} (-1)^{-v} {{n-v} \choose k} \frac{(j+v)!}{(j-v)! ((v!)^2}
\end{equation}

\begin{lemma} \label{alphaM}
$\alpha(n, j, n-i) = M_{i,j}$ where $0 \leq i,j \leq n$.
\end{lemma}

\begin{proof}
Applying Zeilberger's algorithm to the right-hand side of equation~(\ref{alpha1}) gives the recurrence 
\begin{align} \label{reccol}
&(1+j)(2+j+n)\mbox{SUM}[j] - (3+2j)(2i -n) \mbox{SUM}[1+j] \nonumber\\
&- (2+j) (1+j -n)\mbox{SUM}[2+j] = 0.
\end{align} 
The variable $i$ is the row index and $i = n-k$. Given three adjacent columns in the change of basis matrix if two are known, the elements of the third can be found from the recurrence for elements in row $i$.

When $i$ is used as the upper limit, the conditions found by the algorithm are that $i$ is a natural number and none of $\{n, -i +n\}$ is a negative integer. This is satisfied because $n \geq i$.

When $j$ is the upper limit, the conditions are that $j$ is a natural number and none of $\{n, -2 -j +n\}$ is a negative integer. This is satisfied provided that $0 \leq j \leq n-2$, i.e., there are at least three columns and the recurrence variable does not exceed $n-2$.

Equation (20) of Farouki~\cite{farouki2000}, the variable $j$ is used as the row index, $k$ is used as the column index, and $i$ is used as the summation variable. After renaming these variables by $i$, $j$ and $v$ respectively, we obtain
\[
\frac{1}{{n \choose i}} \sum_{ v = \max \{0, i+j-n\}}^{\min\{i,j\}} (-1)^{j+v} {j \choose v} {j \choose v} {{n-j} \choose {i-v}}.
\]

Applying Zeilberger's algorithm to this expression gives the same recurrence multiplied by $-1$
for all four possible combinations of limits:
\begin{itemize}
\item $v = 0$ to $i$ with the condition $i$ is a natural number
\item $v = 0$ to $j$ with the condition $j$ is a natural number
\item $v = i+j -n$ to $i$ with the condition $-1-j+n$ is a natural number and none of $\{j, n, -2-j+n\}$ is a negative integer
\item $v = i+j-n$ to $j$ with the condition $-i-n$ is a natural number and none of $\{j, n, -2-j+n\}$ is a negative integer.
\end{itemize}

It follows that with the condition that $0 \leq i, j \leq n$ and $ j \leq n-2$, both the right-hand side of equation~(\ref{alpha1}) and that of equation (20) of Farouki~\cite{farouki2000} satisfy the same recurrence.

It remains to show that these equations yield the same element values for two adjacent columns, i.e., the initial conditions of the recurrence for each row. We choose the columns with indices $0$ and $1$. It is straightforward to check that for all $i: 0 \leq i \leq n$, $\alpha(n, 0, n-i) = 1$ and $M_{i, 0} = 1$. 

When $j=1$, we have $\alpha(n, 1, n-i) = \frac{2i}n - 1$, and $M_{i, 1} = \alpha(n, 1, n-i)$.
The equation for $\alpha$ has two cases where the limits on the summation are $v=0$ to $0$ when $i = 0$, and $v = 0$ to $1$ when $i > 0$. The latter follows with three cases where the limits on the summation are $v=0$ to $0$ when $i=0$; $v= 0$ to $1$ when $1 \leq i \leq n-1$; and $v=1$ to $1$ when $i = n$.

Hence, both equations give the same result for the elements of the change of basis matrix where the column index is $0$ or $1$.

It follows that both equations have the same initial conditions and satisfy the recurrence above from Zeilberger's algorithm provided that $0 \leq i, j \leq n$ and $ j \leq n-2$. Hence, they are equal for all $0 \leq i, j \leq n$, as required.
\end{proof}

\subsubsection*{Simplifying the Coefficient Function}

Equation~(\ref{alpha1}) can be simplified by Mathematica to a generalized hypergeometric function
\begin{align}
\alpha(n, j, k) =& (-1)^j \sum_{v=0}^{\min\{n-k,j\}} \frac{(-j)_v (1+j)_v (k-n)_v}{1_v (-n)_v} \frac{1}{v!} \nonumber\\
=& (-1)^j \pFq{3}{2}{-j, 1+j, k-n}{1, -n}{1} 
\end{align}

The change of basis matrix has elements 
\[
m_{i,j} = \alpha(n, j, n-i)
\] where $0 \leq i,j \leq n$, so that

\begin{equation}
m_{j,k} = (-1)^k \pFq{3}{2}{-k, 1+k, -j}{1, -n}{1} \label{pfq0}
\end{equation} where $0 \leq j,k \leq n$ using the index variables of Farouki~\cite{farouki2000}.

\begin{lemma} \label{closed-alpha}
Gosper's Algorithm does not find a closed form for $\alpha(n, j, k)$ of equation~(\ref{alpha1}).
\end{lemma}
\begin{proof}
Applying Gosper's algorithm from the `fastZeil' package version 3.61~\cite{paulezbf} to the right-hand side of equation~(\ref{alpha1}) with an upper bound of $n-k$ or $j$ fails.
\end{proof}

Farouki~\cite{farouki2000} derives the following expression for the elements of the change of basis matrix:
\begin{equation}
M_{jk} = \frac{1}{{n \choose k}} \sum_{ i = \max \{0, j+k -n\}}^{\min\{j,k\}} (-1)^{k+i} {j \choose i} {k \choose i} {{n-j} \choose {k-i}}
\end{equation} where $ 0 \leq j,k \leq n$. In our notation, the variable $j$ is $n-k$, i.e., the row index. The variable $k$ is $j$ which is the column index, and the summation variable $i$ is $v$.

This expression is also written using a hypergeometric function~\cite{farouki2000} as
\begin{equation}
M_{jk} = (-1)^k \frac{{{n-k} \choose j}}{{n \choose j}} \pFq{3}{2}{-k, -k, -j}{1, n-k-j+1}{1} \label{pfq1}
\end{equation} however this holds only for the upper skew triangle of $M_{jk}$ where $n \geq j + k$ and is indeterminate for the other elements.

Equation~(\ref{pfq0}) is simpler and more generally defined than equation~(\ref{pfq1}).

\subsubsection*{Closed Forms by Row and Column}

Equation~(\ref{reccol}) is a recurrence that can be used to find a hypergeometric closed form for the coefficient functions for a particular column of the change of basis matrix. Given any three adjacent columns, two of which have known coefficient functions, the one for the unknown column can be found. The first column has the coefficient function, $\alpha(n, 0, n-i) = 1$ where $0 \leq i \leq n$, and the second column has the coefficient function $\alpha(n, 1, n-i) = \frac{2i}n -1$. Applications of equation~(\ref{reccol}) provide $\alpha(n, 2, n-i)$, $\alpha(n, 3, n-i)$ and so on.

Similarly to the proof of Lemma~\ref{alphaM}, we can find the following recurrence on rows of the change of basis matrix.
\begin{align} \label{recrow}
&-(1+i)(i-n)\mbox{SUM}[i] + (1 + 4i + 2i^2 +j +j^2 - 3n - 2i n) \mbox{SUM}[1+i] \nonumber\\
&- (2+i) (1+i -n)\mbox{SUM}[2+i] = 0.
\end{align} provided that $0 \leq i \leq n-2$ in the recurrence.

It is not difficult to show that for the first row, $\alpha(n, j, n) = (-1)^j$ and for the second, $\alpha(n, j, n-1) = (-1)^j (1 - \frac{j(j+i)}n)$. The recurrence of equation~(\ref{recrow}) can be used to find the other $n-1$ formulas.

In both cases, the same formula can be found by applying equation~(\ref{pfq0}).

It is known that for the last column $\alpha(n, n, n-i) = (-1)^{n+i}{n \choose i}$, e.g., Farouki~\cite[equation (23)]{farouki2000}. Equation~(\ref{pfq0}) gives for this column

\begin{equation}
\alpha(n, n, n-i) = (-1)^{n} \sum_{v=0}^i \frac{ (1+n)_v (-i)_v}{(v!)^2} 
\end{equation}

This is the interpolation polynomial in the Lagrange form on the domain $[0, n]$. 
For example, with $n=4$ we have
\begin{align*}
\alpha(4, 4, 4-i) =& 1 - 5i - \frac{15}2(1-i)i - \frac{35}6(1-i)(2-i)i - \frac{35}{12}(1-i)(2-i)(3-i)i\\
=& \sum_{l=0}^4 (-1)^{4+l}{n \choose l} \prod_{\mathclap{\substack{v = 0 \\ v \not= l}}}^n \frac{(i -v)}{(l-v)}.
\end{align*}

This applies to all columns of the change of basis matrix.
\begin{lemma} \label{lagrangecol}
Given the elements $m_{i,j} = (-1)^j \pFq{3}{2}{-j, 1+j, -i}{1, -n}{1}$ of the change of basis matrix from shifted Legendre polynomials to Bernstein polynomials where $0 \leq i,j \leq n$ and the Lagrange interpolation polynomials
\[
l_k(i) = \prod_{\mathclap{\substack{v = 0 \\ v \not= k}}}^n \frac{(i -v)}{(k-v)}
\] where $0 \leq k \leq n$, then for each $j$,
\begin{equation} \label{colfun}
(-1)^j \pFq{3}{2}{-j, 1+j, -i}{1, -n}{1} = \sum_{v=0}^n m_{i,j} l_v(i).
\end{equation}
\end{lemma}

\begin{proof}
In equation~(\ref{colfun}), $j$ is a column index, and the right-hand side of the equation is a polynomial that is a linear combination of terms in the set
\[ L= \{l_0(i), l_1(i), \ldots , l_n(i)\}
\] with coefficients that are elements of the column with index $j$. The set $L$ is a basis that spans the same vector space as $\{1, i, i^2, \ldots, i^n\}$. The elements are linearly independent. Otherwise, suppose that $l_h(i)$ where $0 \leq h \leq n$ is a linear combination of the other Lagrange polynomials in $L$. It follows that $l_h(h) = 1$ and all of the other polynomials in $L$ evaluated at $h$ are zero. This is a contradiction. Hence, the polynomials of $L$ are linearly independent. 

The terms in the Lagrange polynomials in $L$ can be written as polynomials with respect to $\{1, i, i^2, \ldots, i^n\}$. Each has terms of degree from $1$ to $n$ except for $l_0(i)$ that has a constant term. The columns of the change of basis matrix from the Lagrange polynomials to the monomials can be formed from the transposed coordinate vectors of these polynomials.

The left-hand side of the equation is a polynomial that is a linear combination of
\[
I = \{1, -i, -i(-i +1), -i(-i +1)(-i +2), \ldots, (-i)_n \}.
\]
A similar proof step shows that these polynomials form a basis that spans the same the vector space as one with the basis $\{1, i, i^2, \ldots, i^n\}$. The columns of the change of basis matrix from $I$ to the second basis can be similarly formed. The product of the inverse of this matrix and the previous matrix is the change of basis matrix from the Lagrange polynomials to $I$, see Wolfram~\cite[Theorem 1]{wolfram2107}.

Extensionally, the right-hand side of equation~(\ref{colfun}) equals the left-hand side because for each $i$, we have $l_v(i) = \delta_{vi}$, and $m_{i, j} \delta_{ii}$ is equal to the left-hand side from equation~(\ref{pfq0}). It follows that the polynomial representations of both sides of equation~(\ref{colfun}) in any basis that spans $\{1, i, i^2, \ldots, i^n\}$ are equal.
\end{proof}

As an example of the matrices described in the proof of Lemma~\ref{lagrangecol}, let $n=5$ and $j = 3$. We have
\begin{align*}
l_0(i) = & \frac1{120} (-i+1)(-i+2) (-i+3)(-i+4)(-i+5)\\
=& -\frac{i^5}{120} + \frac{i^4}8 - \frac{17 i^3}{24} + \frac{15 i^2}8 - \frac{137 i}{60} + 1\\
l_3(i) = & -\frac1{12} -i (-i+1)(-i+2)(-i+4)(-i+5)\\
=& \frac{i^5}{12} - i^4 + \frac{49i^3}{12} - \frac{13 i^2}2 + \frac{10 i}3.
\end{align*}

The change of basis matrix from the Lagrange interpolation polynomials to the monomials is

\begin{center}
\begin{doublespace}
\noindent\(\left[
\begin{array}{cccccc}
1 & 0 & 0 & 0 & 0 & 0 \\
-\frac{137}{60} & 5 & -5 & \frac{10}{3} & -\frac{5}{4} & \frac{1}{5} \\
\frac{15}{8} & -\frac{77}{12} & \frac{107}{12} & -\frac{13}{2} & \frac{61}{24} & -\frac{5}{12} \\
-\frac{17}{24} & \frac{71}{24} & -\frac{59}{12} & \frac{49}{12} & -\frac{41}{24} & \frac{7}{24} \\
\frac{1}{8} & -\frac{7}{12} & \frac{13}{12} & -1 & \frac{11}{24} & -\frac{1}{12} \\
-\frac{1}{120} & \frac{1}{24} & -\frac{1}{12} & \frac{1}{12} & -\frac{1}{24} & \frac{1}{120} \\
\end{array}
\right]\)
\end{doublespace}
\end{center}

An example of a polynomial in the basis $I$ is
\begin{align*}
-i(-i+1)(-i+2) = & -i^3 + 3i^2 - 2i.\\
\end{align*}

In general, $-i(-i+1) \cdots (-i+n-1)$ is $(-1)^n i^{\underline{n}}$ where $i^{\underline{n}}$ is the falling factorial.
The change of basis matrix from the monomials to $I$ is

\begin{center}
\begin{doublespace}
\noindent\({\left[
\begin{array}{cccccc}
1 & 0 & 0 & 0 & 0 & 0 \\
0 & -1 & -1 & -2 & -6 & -24 \\
0 & 0 & 1 & 3 & 11 & 50 \\
0 & 0 & 0 & -1 & -6 & -35 \\
0 & 0 & 0 & 0 & 1 & 10 \\
0 & 0 & 0 & 0 & 0 & -1 \\
\end{array}
\right]}^{-1} = 
\left[
\begin{array}{cccccc}
1 & 0 & 0 & 0 & 0 & 0 \\
0 & -1 & -1 & -1 & -1 & -1 \\
0 & 0 & 1 & 3 & 7 & 15 \\
0 & 0 & 0 & -1 & -6 & -25 \\
0 & 0 & 0 & 0 & 1 & 10 \\
0 & 0 & 0 & 0 & 0 & -1 \\
\end{array}
\right] \) 
\end{doublespace}
\end{center}

The required change of basis matrix is the product of the matrices.

\begin{center}
\begin{doublespace}
\(\left[
\begin{array}{cccccc}
1 & 0 & 0 & 0 & 0 & 0 \\
0 & -1 & -1 & -1 & -1 & -1 \\
0 & 0 & 1 & 3 & 7 & 15 \\
0 & 0 & 0 & -1 & -6 & -25 \\
0 & 0 & 0 & 0 & 1 & 10 \\
0 & 0 & 0 & 0 & 0 & -1 \\
\end{array}
\right].\left[
\begin{array}{cccccc}
1 & 0 & 0 & 0 & 0 & 0 \\
-\frac{137}{60} & 5 & -5 & \frac{10}{3} & -\frac{5}{4} & \frac{1}{5} \\
\frac{15}{8} & -\frac{77}{12} & \frac{107}{12} & -\frac{13}{2} & \frac{61}{24} & -\frac{5}{12} \\
-\frac{17}{24} & \frac{71}{24} & -\frac{59}{12} & \frac{49}{12} & -\frac{41}{24} & \frac{7}{24} \\
\frac{1}{8} & -\frac{7}{12} & \frac{13}{12} & -1 & \frac{11}{24} & -\frac{1}{12} \\
-\frac{1}{120} & \frac{1}{24} & -\frac{1}{12} & \frac{1}{12} & -\frac{1}{24} & \frac{1}{120} \\
\end{array}
\right]\)
\end{doublespace}
\end{center}

which is

\begin{center}
\begin{doublespace}
\noindent\(\left[
\begin{array}{cccccc}
1 & 0 & 0 & 0 & 0 & 0 \\
1 & -1 & 0 & 0 & 0 & 0 \\
\frac{1}{2} & -1 & \frac{1}{2} & 0 & 0 & 0 \\
\frac{1}{6} & -\frac{1}{2} & \frac{1}{2} & -\frac{1}{6} & 0 & 0 \\
\frac{1}{24} & -\frac{1}{6} & \frac{1}{4} & -\frac{1}{6} & \frac{1}{24} & 0 \\
\frac{1}{120} & -\frac{1}{24} & \frac{1}{12} & -\frac{1}{12} & \frac{1}{24} & -\frac{1}{120} \\
\end{array}
\right]\)
\end{doublespace}
\end{center}

The shifted Legendre polynomial to Bernstein polynomial change of basis matrix with $n=5$ is the matrix
\begin{center}
\begin{doublespace}
\(\left[
\begin{array}{cccccc}
1 & -1 & 1 & -1 & 1 & -1 \\
1 & -\frac{3}{5} & -\frac{1}{5} & \frac{7}{5} & -3 & 5 \\
1 & -\frac{1}{5} & -\frac{4}{5} & \frac{4}{5} & 2 & -10 \\
1 & \frac{1}{5} & -\frac{4}{5} & -\frac{4}{5} & 2 & 10 \\
1 & \frac{3}{5} & -\frac{1}{5} & -\frac{7}{5} & -3 & -5 \\
1 & 1 & 1 & 1 & 1 & 1 \\
\end{array}
\right]\)
\end{doublespace}
\end{center}

The product of the previous change of basis matrix with the column of this matrix with index $j = 3$ is
\[
[ -1, -\frac{12}5, -\frac32, -\frac13, 0, 0 ]^{T}
\]

It follows that
\begin{align*}
\alpha(5, i, 3) =& -1 +\frac{12}5 i + \frac32 i (-i + 1) + \frac13 i (-i+1)(-i+2)\\
=& \frac{i^3}3 - \frac{5 i^2}2 + \frac{137 i}{30} - 1\\
=& \sum_{v=0}^5 m_{v, 3} l_v(i).
\end{align*}

Another interesting property is the coefficient function of the penultimate column.
\begin{lemma} \label{lemm-pen}
The coefficient function of the column with index $n-1$ of the change of basis matrix is
\begin{equation} \label{alpha-pen}
\alpha(n, n-1, n-i) = (-1)^{n+i} \frac{2i-n}n {n \choose i} \: \mbox{where $n > 0$.}
\end{equation} 
\end{lemma}

\begin{proof}
The proof shows that 
\begin{equation} \label{rtp-pen}
P^{\ast}_{n-1}(x) = \sum_{v=0}^n \alpha(n, n-1, v) b_{n-v}^n(x)
\end{equation}

From equation~(\ref{alpha-pen}), the right-hand side of equation~(\ref{rtp-pen})
is
\[
\sum_{v=0}^n (-1)^{2n -v} \frac{n-2v}n {n \choose v} b_{n-v}^n(x).
\]

This equals
\begin{align*}
& \sum_{v=0}^n (-1)^{2n -v} \frac{n-2v}n {n \choose v} { n\choose v} x^{n-v} (1-x)^{v}\\
= 2^{-n} & \sum_{v=0}^n \frac{n-2v}n {n \choose v}^2 (y+1)^{n-v} (y-1)^{v} \: \mbox{where $y = 2x -1$.}\\
\end{align*}

When $n=1$, this expression evaluates to $1 = P_0(y)$ and when $n=2$ it evaluates to $y = P_1(y)$, as required.

We apply Zeilberger's algorithm to this expression by using the Mathematica package `fastZeil' version 3.61~\cite{paulezbf}. This gives the recurrence
\[
(2n +1) y \mbox{ SUM}[1+ n] -n \mbox{ SUM}[n] + (-n-1)\mbox{ SUM}[2+n] = 0
\] provided that $n$ is a natural number. This is a recurrence relation that can be used to define the Legendre polynomials~\cite[Table 18.19.1]{dlmf} where $\mbox{SUM}[n] = P_{n-1}(y)$. Hence, equations~(\ref{rtp-pen}) and (\ref{alpha-pen}) hold.
\end{proof}

\begin{cor}
\begin{equation}
P_n(x) - P_{n-1}(x) = \frac{2^{-n+1}}n \sum_{v=0}^n v {n \choose v}^2 (x+1)^{n-v} (x-1)^{v}
\end{equation} where $n \geq 1$.
\end{cor}

\begin{cor}
\begin{equation}
\alpha(n,1,n-i) = \frac{\alpha(n, n-1, n-i)}{\alpha(n,n,n-i)}
\end{equation} i.e., in the same row, the element in the second column is the ratio of the elements of the penultimate and last columns.
\end{cor}

\begin{proof}
The second column has the coefficient function
\[
\alpha(n, 1, n-i) = \frac{2i}n -1
\] from equation~(\ref{pfq0}).
\end{proof}

A summary of coefficient formulas for this change of basis matrix is given in Appendix~\ref{app1}.

\subsection{Shifted Chebyshev  to Bernstein Polynomials}

AlQudah~\cite{alqudah} considers mappings between shifted Chebyshev polynomials of the second kind, i.e., $U_n(2x -1)$ and Bernstein polynomials, and other mappings involving a generalization of these Chebyshev polynomials.

Equation (6) in~\cite{alqudah} is for Chebyshev polynomials of the second kind.
\[
U_n(x) = \frac{(n+1)(2n)!!}{(2n+1)!!} \sum_{k=0}^n {{n + \frac12} \choose {n-k}} {{n + \frac12} \choose {k}} {\left(\frac{x+1}2\right)}^{n-k} {\left(\frac{x-1}2\right)}^{k}.
\]

This enables the shifted form to be expressed in terms of Bernstein polynomials:
\begin{align}
U_n(2x-1) =& \frac{(n+1)(2n)!!}{(2n+1)!!} \sum_{k=0}^n {{n + \frac12} \choose {n-k}} {{n + \frac12} \choose {k}} {x}^{n-k} (x-1)^{k} \nonumber\\
=& \frac{(n+1)(2n)!!}{(2n+1)!!} \sum_{k=0}^n (-1)^{k}{{n + \frac12} \choose {n-k}} {{n + \frac12} \choose {k}} {x}^{n-k} (1-x)^{k}\nonumber\\
=& \frac{(n+1)(2n)!!}{(2n+1)!!} \sum_{l=0}^n (-1)^{n-l} \frac{{{n + \frac12} \choose {l}} {{n + \frac12} \choose {n-l}}}{{n \choose l}} b^n_l(x). \label{mapBUst}
\end{align}

This gives a coefficient function
\begin{equation*} 
\alpha(n, k) = \frac{(n+1)(2n)!!}{(2n+1)!!} (-1)^{k} \frac{{{n + \frac12} \choose {n-k}} {{n + \frac12} \choose {k}}}{{n \choose k}}.
\end{equation*}

A small correction needs to be made to equation (7) in~\cite{alqudah}. The exponent of the factor $-1$ should be $n-k$ instead of $n+1$. We can show that 
\begin{equation}
{{n+ \frac12} \choose n} (2n)!! = (2n+1)!!
\end{equation} from the definitions, so that the coefficient function becomes
\begin{equation} \label{coeff1}
\alpha(n, k) = (-1)^{k} (n+1) \frac{{{n + \frac12} \choose {n-k}} {{n + \frac12} \choose {k}}}{{n \choose k}{{n + \frac12} \choose n} }.
\end{equation}

This example is unusual because  there is no single change of basis matrix from the domain $\{U_0^{\ast}(x), \ldots, U_n^{\ast}(x)\}$  to the range basis $\{b^n_0(x), \ldots, b^n_n(x)\}$. The range basis depends on the degree $n$ of $U^{\ast}_n(x)$.

Instead, the coefficient function can be found from the last column of the products of two change of basis matrices to form the mapping
\[
\{U_0^{\ast}(x), \ldots, U_n^{\ast}(x)\} \:\: \mbox{to} \:\: \{b^n_0(x), \ldots, b^n_n(x)\}.
\] More specifically, we shall apply equation~(\ref{compDA}).
The coefficient function $\alpha_2$ for the mapping from the descending basis of shifted Chebyshev polynomials of the second kind to monomials is given in equation (33) of Wolfram~\cite{wolfram2107}:
\[
\alpha_2(n, k) = 
\sum_{v=0}^{\lfloor \frac{k}2 \rfloor} {{n-2v} \choose {k-2v}} {{n-v}\choose v} (-1)^{k-v} 2^{2(n-v)-k}. 
\]

The coefficient function $\alpha_1$ for the mapping from monomials to the ascending basis of Bernstein polynomials is equation~(\ref{invB1}):
\begin{equation*} 
\alpha_1(n,m,k) = \frac{{{n-m} \choose k}}{{n \choose k}}
\end{equation*} where $0 \leq m \leq n$.

Applying equation~(\ref{compDA}) gives
\begin{equation} 
\alpha_3(n,j,k) =\sum\limits_{v=0}^{\min\{n-k, j\}} \frac{{{n-v} \choose k}}{{n \choose k}} \sum_{l=0}^{\lfloor \frac{j-v}2 \rfloor} {{j-2l} \choose {j-v-2l}} {{j-l}\choose l} (-1)^{j-v-l} 2^{j -2l +v}
\end{equation} where $0 \leq j, k,m \leq n$. When $j=n$, we obtain

\begin{equation} \label{a3}
\alpha_3(n,n,k) =(-2)^{n}\sum\limits_{v=0}^{n-k} (-2)^{v}\frac{{{n-v} \choose k}}{{n \choose k}} \sum_{l=0}^{\lfloor \frac{n-v}2 \rfloor} (-4)^{-l} {{n-2l} \choose {v}} {{n-l}\choose l} 
\end{equation}

This is equal to the coefficient function of equation~(\ref{coeff1}) as the following lemma verifies.

\begin{lemma}
\begin{equation}
\alpha_3(n,n,k)= \alpha(n, k) \:\: \mbox{where $0 \leq k \leq n$.}
\end{equation} 
\end{lemma}
\begin{proof}
We apply Zeilberger's algorithm from the Mathematica package `fastZeil' version 3.61~\cite{paulezbf} to the right-hand side of equation~(\ref{a3}). This gives
\begin{equation} \label{reck}
(-1+2k-2n)\mbox{ SUM}[k] + (-3-2k) \mbox{ SUM}[1+k] = 0
\end{equation} provided that $-1-k+n$ is a natural number and none of $\{k, n\}$ is a negative integer. These conditions are satisfied because $0 \leq k \leq n$, and $n > k$ in the recurrence.

With Mathematica, we can show that the right-hand side of equation~(\ref{coeff1}) satisfies equation~(\ref{reck}) where $\mbox{ SUM}[k] = \alpha(n,k)$. The initial conditions are $\mbox{ SUM}[0] = n +1$ and $\mbox{ SUM}[1] = -\frac13 (n+1)(2n+1)$.

The right-hand side of equation~(\ref{a3}) simplifies\footnote{Mathematica 13.0 gives a result that is half of this expression. I have reported this as a bug.} to 
\[
\alpha(n, n, k) = (-1)^n (n+1)\frac{ \sqrt{\pi} \; \Gamma(-\frac12 - k)}{2 \Gamma(-\frac12 -n) \; \Gamma(\frac32 -k + n) }.
\] This satisfies equation~(\ref{reck}). 
We can show that $\mbox{ SUM}[0] = \alpha(n,n,0)$ and $\mbox{ SUM}[1] = \alpha(n,n,1)$.
\end{proof}

\part{Truncation, Alternation, and Superposition}

Truncation, Alternation, and superposition are methods to produce polynomial bases from existing ones. Alternation and superposition form bases that do not have definite parity from ones that do. 

This has applications where the domain or range basis does not have definite parity and we wish to find a mapping between it and the Zernike Radial polynomials or other polynomials with definite parity such as Chebyshev polynomials of the first or second kinds.

Alternation and superposition cannot be combined, but each of them can be combined with truncation.

\section{Truncation}
\label{truncsec}

We now consider truncation from Definition~\ref{parcoeff} and change of basis matrices.

The following matrix is the change of basis matrix for Laguerre polynomials up to degree $x^5$ to $\{1, x, x^2, x^3, x^4, x^5\}$.
\begin{center}
$M_{ML} =$
\begin{doublespace}
\noindent\(\left[
\begin{array}{cccccccc}
1 & \phantom{-}1 & \phantom{-}1 & \phantom{-}1 & \phantom{-}1 & \phantom{-}1 \\
0 & -1 & -2 & -3& -4 & -5 \\
0 & \phantom{-}0 & \phantom{-}\frac12 & \phantom{-}\frac32 & \phantom{-}3 & \phantom{-}5 \\
0 & \phantom{-}0 & \phantom{-}0 & -\frac16 & -\frac{2}3 & -\frac53 \\
0 & \phantom{-}0 & \phantom{-}0 & \phantom{-}0 & \phantom{-}\frac{1}{24} & \phantom{-}\frac{5}{24} \\
0 & \phantom{-}0 & \phantom{-}0 & \phantom{-}0 & \phantom{-}0 & -\frac{1}{120} \\
\end{array}
\right]\)
\end{doublespace}
\end{center} By excluding the first two rows and columns and the last row and column, we form the change of basis matrix from 
\[
\{L_2(x)_{2,2}, L_3(x)_{3,2}, L_4(x)_{4,2}\} \:\; \mbox{\rm  to } \{x^2, x^3, x^4\}.
\]
 This gives the upper triangular matrix
\begin{center}
\begin{doublespace}
\noindent\(\left[
\begin{array}{ccc}
\frac12 & \phantom{-}\frac{3}2 & \phantom{-}3 \\
0 & -\frac{1}{6} & -\frac{2}{3} \\
0 & \phantom{-}0 & \phantom{-}\frac{1}{24} \\
\end{array}
\right]\)
\end{doublespace}
\end{center}
and from the third column, we have
\begin{align*}
L_4(x)_{4,2} =& \frac1{24} x^4 - \frac23 x^3 + 3 x^2.
\end{align*}

The exclusion of the first $k_1$ rows and columns of a change of basis matrix, and the exclusion of the last $k_2$ rows and columns are two optimizations discussed in Wolfram~\cite[\S 5.1]{wolfram2107}. This matrix could be used to provide a change of basis matrix from 
$\{L_2(x)_{2,2}, L_3(x)_{3,2}, L_4(x)_{4,2}\}$ to $\{b_2^2(x), b_2^3(x), b_2^4(x)\}$.

\begin{lemma} \label{counttrunc}
Let $M$ be an upper or lower triangular change of basis matrix with $b$ rows and $b$ columns. There are $\frac12 b(b+1)-1$ ways to truncate the matrix by excluding the first $k_1$ rows and columns, and the last $k_2$ rows and columns where $0< k_1 + k_2 < b$, to form another change of basis matrix.
\end{lemma}

\begin{proof}
Excluding the first $k_1$ rows and columns, and the last $k_2$ rows and columns from $M$ yields a matrix with $b - k_1 - k_2$ rows and columns and whose main diagonal is the same as $m_{k_1, k_1} \cdots m_{b-1 - k_2, b-1 -k_2}$ in $M$. There are $b$ ways to form a $1 \times 1$ matrix through truncation, $b-1$ ways to form a $2 \times 2$ such matrix, and $2$ ways to form a $(b-1) \times (b-1)$ matrix.
The result follows from 
\[
\sum_{i=2}^{b} i = \frac12 b(b+1)-1.
\]
\end{proof}

\begin{definition}
Let $M$ be an upper or lower triangular $b \times b$ matrix. The term $\mbox{tr}_{b_1, b_2} M$ where $0< k_1 + k_2 < b$
means the truncated matrix formed from $M$ by excluding its first $k_1$ rows and columns, and its last $k_2$ rows and columns.
\end{definition}

\begin{theorem} \label{local}
Let $M$ be an upper or lower triangular $b \times b$ matrix and $M' = \mbox{tr}_{b_1, b_2} M$ where $0< k_1 + k_2 < b$. Then the following hold:
\begin{enumerate}
\item $M'$ is also an upper or lower triangular matrix, respectively.
\item If $N$ is also an upper or lower triangular $b \times b$ matrix, respectively, then 
\[
\mbox{tr}_{k_1, k_2} (M N) = (\mbox{tr}_{k_1, k_2} M) (\mbox{tr}_{k_1, k_2} N)
\]
\item If $M$ is an invertible matrix, then
\[
\mbox{tr}_{k_1, k_2} (M^{-1}) = (\mbox{tr}_{k_1, k_2} M)^{-1}
\]
\end{enumerate}
\end{theorem}
\begin{proof}
The first property follows from the proof of Lemma~\ref{counttrunc}, above. For the second and third properties, we consider upper triangular matrices because the proofs for lower triangular matrices are similar.

The second property holds for the elements below the main diagonal because each of them is $0$ from the first property and Theorem~\ref{tri-inv}. Suppose that $i$ is the row index and $j$ is the column index of an element $v_{i,j}$ in the upper triangle of $M N$, so that $0 \leq i \leq j < b$. We have
\[
v_{i, j} = \sum\limits_{l = 0}^{b-1} m_{i, l} n_{l,j}
\] by the definition of matrix multiplication. Elements below the main diagonals are zero in both $M$ and $N$, and this equation simplifies to
\[
v_{i, j} = \sum\limits_{l = i}^{j} m_{i, l} n_{l,j}
\] This implies that elements in the first $i$ rows and columns of $M$ and $N$, and elements in the last $b-j-1$ rows and columns of $M$ and $N$ have no effect on this sum and the only elements that are used are $m_{i,i}, \ldots, m_{i,j}$ and $n_{i,j}, \ldots, n_{j,j}$. It follows that these rows and columns can be excluded where $k_1 \leq i \leq j \leq k_2$, as required.

For the proof of the third property, since $M$ is invertible, let $N = M^{-1}$. We have that $\mbox{tr}_{k_1, k_2} (M N)$ is the identity matrix with the same dimensions as $\mbox{tr}_{k_1, k_2} M$. From property 2, it follows that $\mbox{tr}_{k_1, k_2} N$ is the inverse of $\mbox{tr}_{k_1, k_2} M$, as required.
\end{proof}

The third property of Theorem~\ref{local} shows that truncation and inversion commute. This implies that given a change of basis matrix $M_{vt}$ and the coefficient function $\alpha$ or $\beta$ for its inverse $M_{tv}$, the same coefficient function can be used for every truncation of $M_{tv}$. 

Returning to the earlier example with Laguerre polynomials, the coefficient function for the inverse matrix $M_{LM}$ is given by
\[
\alpha(n, k) = (-n)_{n-k} (n -k+1)_k
\] from \cite[equation 18.18.19]{dlmf}, where $0 \leq k \leq n \leq 5$. This gives

\begin{center}
$M_{LM} =$
\begin{doublespace}
\noindent\(\left[
\begin{array}{cccccccc}
1 & \phantom{-}1 & \phantom{-}2 & \phantom{-}6 & \phantom{-}24 & \phantom{-}120 \\
0 & -1 & -4 & -18 & -96 & -600 \\
0 & \phantom{-}0 &\phantom{-} 2 & \phantom{-}18 & \phantom{-}144 & \phantom{-}1200 \\
0 & \phantom{-}0 & \phantom{-}0 & -6 & -96 & -1200 \\
0 & \phantom{-}0 & \phantom{-}0 &\phantom{-} 0 & \phantom{-}24 & \phantom{-}600 \\
0 & \phantom{-}0 &\phantom{-} 0 & \phantom{-}0 & \phantom{-}0 & -120 \\
\end{array}
\right]\)
\end{doublespace}
\end{center}

We have 
\[
\mbox{tr}_{2,1} M_{LM} = \left[\begin{array}{ccc}
2 & \phantom{-}18 & \phantom{-}144 \\
\\
0 & -6 & -96 \\
\\
0 & \phantom{-}0 & \phantom{-}24
\end{array} \right]
\] 
Hence, 
\begin{align*}
x^4 =& 24 L_4(x)_{4,2} -96 L_3(x)_{3,2} + 144 L_2(x)_{2,2}\\
=& (x^4 - 16x^3 + 72 x^2) - 16(-x^3 + 9x^2) + 72 x^2
\end{align*}

Similarly, from the last column of $\mbox{tr}_{1,1} M_{LM}$, we have
\begin{align*}
x^4 =& 24 L_4(x)_{4,1} -96 L_3(x)_{3,1} + 144 L_2(x)_{2,1} -96 L_1(x)_{1,1}\\
=& (x^4 - 16x^3 + 72 x^2 -96x) - 16(-x^3 + 9x^2 -18x ) + 72 (x^2 - 4x) + 96 x\\
\end{align*}

The matrix $\mbox{tr}_{1,1} M_{LM}$ is the inverse of
\[
\mbox{tr}_{2,1} M_{ML} = 
\left[
\begin{array}{ccc}
\frac12 & \phantom{-}\frac{3}2 & \phantom{-}3 \\
\\
0 & -\frac{1}{6} & -\frac{2}{3} \\
\\
0 & \phantom{-}0 & \phantom{-}\frac{1}{24} \\
\end{array}
\right]
\]

The matrix $\mbox{tr}_{2,1} M_{LM}$ is the change of basis matrix from the exchange basis of $\{x^2, x^3, x^4\}$ to the basis of truncated Laguerre polynomials 
\[
\{L_2(x)_{2,2}, L_3(x)_{3,2}, L_4(x)_{4,2}\}.
\]

We can use this matrix to change basis from $B = \{b_2^2(x), b_2^3(x), b_2^4(x)\}$ to these truncated Laguerre polynomials by matrix multiplication
\[
M_{LM} M_{MB} = 
\left[\begin{array}{ccc}
2 & \phantom{-}18 & \phantom{-}144 \\
\\
0 & -6 & -96 \\
\\
0 & \phantom{-}0 & \phantom{-}24
\end{array} \right]
\left[\begin{array}{ccc}
1 & \phantom{-}3 & \phantom{-}6 \\
\\
0 & -3 & -12 \\
\\
0 & \phantom{-}0 & \phantom{-}6
\end{array} \right]
\]

This product is
\[
M_{LB} = 
\left[\begin{array}{ccc}
2 &-48 & \phantom{-}660 \\
\\
0 & \phantom{-}18 & -504 \\
\\
0 & \phantom{-}0 & \phantom{-}144
\end{array} \right]
\]
and we have, for example, from the third column that
\begin{align*}
b_2^4(x) =& 6x^4 -12 x^3 + 6x^2\\
=& 144 L_4(x)_{4,2} - 504 L_3(x)_{3,2} + 660 L_2(x)_{2,2}\\
=& \frac{144}{24}(x^4 - 16 x^3 + 72 x^2) - \frac{504}{6}(-x^3 + 9x^2) + \frac{660}{2} x^2.
\end{align*}

\section{Alternation}

Alternation creates a polynomial basis without definite parity from polynomials that have definite parity. This enables changes of basis from Bernstein polynomials to Zernike Radial polynomials, or Chebyshev polynomials of the first kind to Bernstein polynomials, for example.

An example of alternation starts with the descending basis 
\[
\{R_2^2(x), R_4^2(x), R_6^2(x), R_8^2(x)\}
\] which has the following change of basis matrix to $\{x^2, x^4, x^6, x^8\}$:

\begin{doublespace}
\begin{equation*}
\begin{bmatrix}
\phantom{-}1 & -3 & \phantom{-}6 & -10 \\
\phantom{-}0 & \phantom{-}4 & -20 & \phantom{-}60 \\
\phantom{-}0 & \phantom{-}0 & \phantom{-}15 & -105 \\
\phantom{-}0 & \phantom{-}0 & \phantom{-}0 & \phantom{-}56 \\
\end{bmatrix}
\end{equation*}
\end{doublespace}

For example, $R_6^2(x) = 15 x^6 - 20 x^4 + 6 x^2$ from the third column.
Each polynomial in the basis has even parity. Alternation includes polynomials with odd parity in the basis, specifically, $\{R_3^3(x), R_5^3(x), R_7^3(x)\}$ and the change of basis matrix for the mapping to $\{x^2, x^3, x^4, x^5, x^6, x^7, x^8\}$ is

\begin{center}
\begin{doublespace}
\noindent\(\left[
\begin{array}{ccccccc}
1 & 0 & -3 & 0 & 6 & 0 & -10 \\
0 & 1 & 0 & -4 & 0 & 10 & 0 \\
0 & 0 & 4 & 0 & -20 & 0 & 60 \\
0 & 0 & 0 & 5 & 0 & -30 & 0 \\
0 & 0 & 0 & 0 & 15 & 0 & -105 \\
0 & 0 & 0 & 0 & 0 & 21 & 0 \\
0 & 0 & 0 & 0 & 0 & 0 & 56 \\
\end{array}
\right]\)
\end{doublespace}
\end{center}
The columns of this matrix alternate between the original basis and the basis with polynomials of odd parity. 

\begin{definition}
Suppose that $f(x)_{u,l}$ is a polynomial over $\mathbb{R}(x)$ with definite parity. An alternating basis is a basis in these polynomials that spans the same vector space as
\[
\{x^m, x^{m+1}, \ldots, x^{n-1}, x^n\}.
\] where $0 \leq m \leq n$ and $m$ and $n$ can have the same or different parity.

A descending alternating basis is either
\[\{f(x)_{m,m}, f(x)_{m+1, m+1}, f(x)_{m+2, m}, f(x)_{m+3, m+1}, \ldots, f(x)_{n-1, m+1}, f(x)_{n,m}\}
\] if $n$ and $m$ have the same parity, or
\[\{f(x)_{m,m}, f(x)_{m+1, m+1}, f(x)_{m+2, m}, f(x)_{m+3, m+1}, \ldots, f(x)_{n-1, m}, f(x)_{n,m+1}\}
\] if $n$ and $m$ do not have the same parity.

\medskip
An ascending alternating basis is either 
\[
\{f(x)_{n,m}, f(x)_{n-1,m+1}, f(x)_{n,m+2}, f(x)_{n-1,m+3}\ldots, f(x)_{n-1, n-1}, f(x)_{n,n}\}
\] if $n$ and $m$ have the same parity, or
\[
\{f(x)_{n-1,m}, f(x)_{n,m+1}, f(x)_{n-1,m+2}, f(x)_{n,m+3}\ldots, f(x)_{n-1, n-1}, f(x)_{n,n}\}
\] if $n$ and $m$ do not have the same parity.
\end{definition}
\medskip

In the previous example, the alternating basis is 
\[
\{R_2^2(x), R_4^2(x), R_6^2(x), R_8^2(x)\} \cup \{R_3^3(x), R_5^3(x), R_7^3(x)\}
\] which is a descending basis with $n = 8$ and $m =2$.

\medskip

The following definition names change of basis matrices where the domain or range basis is an alternating basis. 
\begin{definition}
A change of basis matrix is an alternating matrix if its domain basis or its range basis is an alternating basis.
\end{definition}

An alternating matrix has zero elements when $| i -j|$ is odd, where $i$ is the row index, $j$ is the column index and $0 \leq i,j \leq n-m$. The elements of its leading diagonal are non-zero.

For example, the change of basis matrix from the ascending alternating basis $\{R_4^0(x), R_5^1(x), R_4^2(x), R_5^3(x), R_4^4(x), R_5^5(x)\}$ to the descending alternating basis $\{T_0(x), T_1(x), \ldots, T_5(x)\}$ is the alternating matrix

\begin{center}
\begin{doublespace}
\noindent\(\left[
\begin{array}{cccccc}
\frac{1}{4} & 0 & 0 & 0 & \frac{3}{8} & 0 \\
0 & \frac{1}{4} & 0 & \frac{1}{8} & 0 & \frac{5}{8} \\
0 & 0 & \frac{1}{2} & 0 & \frac{1}{2} & 0 \\
0 & \frac{1}{8} & 0 & \frac{9}{16} & 0 & \frac{5}{16} \\
\frac{3}{4} & 0 & \frac{1}{2} & 0 & \frac{1}{8} & 0 \\
0 & \frac{5}{8} & 0 & \frac{5}{16} & 0 & \frac{1}{16} \\
\end{array}
\right]\)
\end{doublespace}
\end{center} 
From the fourth column, we have
\begin{align*}
R_5^3(x) =& \frac{5}{16} T_5(x) + \frac{9}{16} T_3(x) + \frac18 T_1(x)\\
=& 5x^5 - 4x^3.
\end{align*}

\subsubsection*{Descending Bases}

We consider the general forms of change of basis matrices $M_{vt}$ where either $v$ or $t$ is a descending alternating basis that spans the same vector space as $\{x^m, \ldots, x^n\}$ where $0 \leq m \leq n$ and $t$ or $v$, respectively, also spans this vector space.

If $n$ and $m$ have the same parity, let $l = \frac{n-m}2$. The general form of the change of basis matrix $M_{vt}$ in this case is as follows:

\begin{figure}[H]
\begin{gather*}
\begin{bmatrix}
\beta(m,m, 0) & \cdots & \beta(n-2,m, l -1) &0 & \beta(n,m, l)\\
0 & \ldots & 0& \beta(n-1, m+1, l -1) & 0\\
\ldots & \ldots & \ldots & \ldots & \ldots \\
0 & \cdots & 0 & \beta(n-1, m+1, 1)& 0\\
0 &\cdots & \beta(n-2, m, 0) & 0& \beta(n ,m, 1)\\
0 & \cdots & 0 & \beta(n-1, m+1, 0)& 0\\
0 & \cdots & 0 & 0 & \beta(n, m, 0)
\end{bmatrix}
\end{gather*} 
\caption{Change of Basis Matrix: Descending Basis with Equal Parities} \label{MDAltsame}
\end{figure}

The elements of the matrix are
\[
m_{i,j} = 
\left\{
\begin{array}{ll} 
\beta(m+j , m, \frac{j-i}2 ) & \mbox{if $j \geq i$, $i$ is even and $j$ is even}\\
\beta(m+j, m+1,\frac{j-i}2 ) & \mbox{if $j \geq i$, $i$ is odd and $j$ is odd}\\
0 & \mbox{if $j-i$ is odd}\\
0 & \mbox{if $j< i$}
\end{array}
\right.
\] where its indices are $0 \leq i,j \leq n - m$
\bigskip

If $n$ and $m$ do not have the same parity, let $l = \frac{n-m -1}2$. The general form of the change of basis matrix $M_{vt}$ in this case is as follows:
\begin{figure}[H]
\begin{gather*}
\begin{bmatrix}
\beta(m,m, 0) & \cdots & 0 &\beta(n-1, m, l ) & 0\\
0 & \ldots & \beta(n-2,m+1, l -1)& 0 & \beta(n,m+1, l)\\
\ldots & \ldots & \ldots & \ldots & \ldots \\
0 & \cdots & 0 & \beta(n-1, m, 1)& 0\\
0 &\cdots & \beta(n-2, m+1, 0) & 0& \beta(n ,m+1, 1)\\
0 & \cdots & 0 & \beta(n-1, m, 0)& 0\\
0 & \cdots & 0 & 0 & \beta(n, m+1, 0)
\end{bmatrix}
\end{gather*} 
\caption{Change of Basis Matrix: Descending Basis with Unequal Parities} \label{MDAltdiffer}
\end{figure}

The elements of this matrix are
\[
m_{i,j} = 
\left\{
\begin{array}{ll} 
\beta(m+j , m+1, \frac{j-i}2 ) & \mbox{if $j \geq i$, $i$ is odd and $j$ is odd}\\
\beta(m+j, m,\frac{j-i}2 ) & \mbox{if $j \geq i$, $i$ is even and $j$ is even}\\
0 & \mbox{if $j-i$ is odd}\\
0 & \mbox{if $j< i$}
\end{array}
\right.
\] where its indices are $0 \leq i,j \leq n - m$
\medskip

In both cases, the upper triangle is a `chequered pennant', i.e., $m_{i,j} = 0$ if $j-i$ is odd, and so every other diagonal above the main diagonal has elements that are zero. 

\begin{equation}
\alpha_3(n, m, k)= \sum_{\mathclap{\substack{l = 0 \\k - l \: \mbox{\small even}}}}^{k} \beta_1(n - l, m + (l \bmod 2), \frac{k-l}2) r(l)
\end{equation} where $0 \leq k \leq n-m$, 
where $0 \leq i,j \leq n-m$.

\subsubsection*{Ascending Bases}

We define the general forms of change of basis matrices $M_{vt}$ where either $v$ or $t$ is an ascending alternating basis that spans the same vector space as $\{x^m, \ldots, x^n\}$ where $0 \leq m \leq n$ and $t$ or $v$, respectively, also spans this vector space.

If $n$ and $m$ have the same parity, let $l = \frac{n-m}2$. The form of the change of basis matrix in this case is as follows.

\begin{figure}[H]
\begin{gather*}
\begin{bmatrix}
\beta(n,m, l) & 0& \cdots & 0& 0 \\
0 &\beta(n-1,m+1, l-1) & \cdots & 0& 0 \\
\vdots & \vdots & \vdots & \vdots & \vdots \\
0 & \beta(n-1, m+1, 1) & \cdots & 0& 0\\
\beta(n,m,1) & 0& \cdots & 0 & 0 \\
0 & \beta(n-1, m+1, 0) & \cdots & \beta(n-1, n-1, 0)& 0\\
\beta(n,m,0) & 0 & \cdots & 0 & \beta(n, n, 0)
\end{bmatrix}
\end{gather*} 
\caption{Change of Basis Matrix: Ascending Basis with Equal Parities} \label{MAAltsame}
\end{figure}

The elements of this matrix are
\begin{equation}
m_{i,j} = 
\left\{
\begin{array}{ll} 
\beta(n, m+j, \lfloor \frac{n-m-i}2 \rfloor) & \mbox{if $i-j$ is even and $j$ is even}\\
\beta(n-1, m+j, \lfloor \frac{n-m-i}2 \rfloor ) & \mbox{if $i-j$ is even and $j$ is odd}\\
0 & \mbox{if $i-j$ is odd}\\
0 & \mbox{if $j> i$}
\end{array} \label{AAeq}
\right.
\end{equation} where $0 \leq i,j \leq n-m$.
\smallskip 

For example, if $n=9$, $m = 3$ and $\beta$ is the coefficient function of Theorem~\ref{beta-RM} for the mapping from the monomials to Zernike Radial polynomials, then the alternating change of basis matrix is
\begin{center}
\begin{doublespace}
\noindent\(\left[
\begin{array}{ccccccc}
-\frac{1}{20} & 0 & 0 & 0 & 0 & 0 & 0 \\
0 & \frac{1}{15} & 0 & 0 & 0 & 0 & 0 \\
\frac{1}{4} & 0 & \frac{1}{21} & 0 & 0 & 0 & 0 \\
0 & -\frac{2}{5} & 0 & -\frac{1}{7} & 0 & 0 & 0 \\
-\frac{7}{10} & 0 & -\frac{1}{3} & 0 & -\frac{1}{8} & 0 & 0 \\
0 & \frac{4}{3} & 0 & \frac{8}{7} & 0 & 1 & 0 \\
\frac{3}{2} & 0 & \frac{9}{7} & 0 & \frac{9}{8} & 0 & 1 \\
\end{array}
\right]\)
\end{doublespace}
\end{center}

From the third column,
\begin{align*}
x^5 =& \frac97 R_9^9(x) - \frac13 R_9^7(x) + \frac1{21} R_9^5(x)\\
=& \frac97 x^9 - \frac13 (9 x^9 - 8x^7) + \frac1{21} (36 x^9 - 56 x^7 + 21 x^5).
\end{align*}

If $n$ and $m$ have the different parities, let $l = \frac{n-m -1}2$. The form of the change of basis matrix in this case is as follows.

\begin{figure}[ht]
\begin{gather*}
\begin{bmatrix}
\beta(n-1,m, l) & 0& \cdots & 0& 0 \\
0 &\beta(n,m+1, l) & \cdots & 0& 0 \\
\vdots & \vdots & \vdots & \vdots & \vdots \\
\beta(n-1,m,1) & 0& \cdots & 0& 0\\
0 & \beta(n, m+1, 1) & \cdots & 0 & 0 \\
\beta(n-1,m,0) &0 & \cdots & \beta(n-1, n-1, 0)& 0\\
0 & \beta(n, m+1, 0) & \cdots & 0 & \beta(n, n, 0)
\end{bmatrix}
\end{gather*} 
\caption{Change of Basis Matrix: Ascending Basis with Unequal Parities} \label{MAAltdiffer}
\end{figure}

The elements of this matrix are
\begin{equation}
m_{i,j} = 
\left\{
\begin{array}{ll} 
\beta(n-1, m+j, \lfloor\frac{n-m-i}2 \rfloor) & \mbox{if $i-j$ is even and $j$ is even}\\
\beta(n, m+j, \lfloor \frac{n-m-i}2 \rfloor ) & \mbox{if $i-j$ is even and $j$ is odd}\\
0 & \mbox{if $i-j$ is odd}\\
0 & \mbox{if $j> i$} \label{AAneq}
\end{array}
\right.
\end{equation} where $0 \leq i,j \leq n-m$.

\begin{equation} \label{alpha-new2}
\alpha_3(n, m, k)= \sum_{\mathclap{\substack{l = 0 \\k - l \: \mbox{\small even}}}}^{v-k} \beta_1(n - (l \bmod 2), m+l, \lfloor \frac{k}2\rfloor) r(n-m- l)
\end{equation} where $0 \leq k \leq n-m$, 

\subsubsection*{The Inverse Matrix}

The following theorem shows that if the inverse $\beta^{-1}$ of a coefficient function is known for a descending or ascending basis with definite parity, then $\beta^{-1}$ can be used to find the inverse of an alternating change of basis matrix.

\begin{theorem} \label{inv-alt}
Let $M_{vt}$ be a change of basis matrix for a basis $t$ that is either a descending alternating basis, or an ascending alternating basis. The general forms of these matrices are in Figures~\ref{MDAltsame} -- \ref{MAAltdiffer}. 

The inverse change of basis matrix $M_{tv}$ has the same form as $M_{vt}$ except that the coefficient function $\beta$ is replaced by its inverse, the mapping from $v$ to $t$.
\end{theorem}

\begin{proof}
It suffices to consider a descending alternating basis and an upper triangular change of basis matrix as in Figures~\ref{MDAltsame} and \ref{MDAltdiffer}. A similar proof applies for an ascending alternating basis and a lower diagonal change of basis matrix.

From Theorem~\ref{tri-inv}, $M_{tv}$ exists and is an upper triangular matrix. Consider a change of basis matrix between two descending bases with definite parity without alternation. This matrix has the following form.
\medskip

\begin{figure}[ht]
\resizebox{\textwidth}{!}{%
$\displaystyle
\left[
\begin{array}{rrrrr}
\beta(m,m, 0) & \cdots & \beta(n-4,m, l -2) & \beta(n-2, m, l -1)& \beta(n,m, l)\\
0 & \ldots & \beta(n-4, m, l -1)& \beta(n-2, m, l -2) & \beta(n ,m, l-1)\\
\ldots & \ldots & \ldots & \ldots & \ldots \\
0 & \cdots & \beta(n-4, m, 1) & \beta(n-2, m, 2)& \beta(n ,m, 3)\\
0 &\cdots & \beta(n-4, m, 0) & \beta(n-2, m, 1)& \beta(n ,m, 2)\\
0 & \cdots & 0 & \beta(n-2, m, 0)& \beta(n ,m, 1) \\
0 & \cdots & 0 & 0 & \beta(n, m, 0)
\end{array} 
\right]
$} 
\end{figure}
\medskip

To find the inverse coefficient function $\beta^{-1}$, we can use back substitution. We shall show from back substitution that the inverse of 
an alternating change of basis matrix has the same inverse coefficient function. The zero elements in the upper triangle of the alternating change of basis matrix, $m_{i,j} = 0$ where $j-i$ is odd and $j \geq i$, have no effect on back substitution.

This can be shown by mathematical induction. 
In the base case, suppose $n$ and $m$ have the same parity. We have
$\beta(n,m,0) m_{j,j} = 1$ where $m_{j,j}$ is an element of $M_{tv}$. This gives $m_{j,j} = \frac1{\beta(n,m,0)}$. This is equal to $\beta^{-1}(n,m,0)$ from finding the inverse of the matrix above. If $n$ and $m$ do not have the same parity, then $n$ and $m+1$ have the same parity. From the change of basis matrix in Figure~\ref{MDAltdiffer}, we have that $\mbox{tr}_{1,0} M_{vt}$ is a change of basis matrix for a descending basis with equal parities. By property 3 of Theorem~\ref{local} this truncation and inversion are commutative. Hence, the proof step for equal parities applies and $m_{j,j} = \beta^{-1}(n,m+1, 0)$, as required.

The induction hypothesis is that the result holds for all $l: 0 \leq l < h \leq n-m$. We first consider the case that $n$ and $m$ have equal parity. The element $x = m_{n-m-h,n-m}$ of $M_{tv}$ is an unknown.
If $h$ is odd, then the element in row $n-m-h$ and column $n-m$ in $M_{vt}$ is zero, and so is every other element in that row in columns with even column indices. In column $n-m$ of $M_{tv}$, elements in rows with odd row indices are zero by the induction hypothesis. Hence the dot product of row $n-m-h$ of $M_{vt}$ with column $n-m$ in $M_{tv}$ is zero.

If $h$ is even, then the element in row $n-m-h$ and column $n-m$ in $M_{vt}$ is $\beta(n,m, \frac{h}2 )$. The dot product of row $n-m-h$ of $M_{vt}$ with column $n-m$ in $M_{tv}$ is

\[
\beta(n,m, 0) \; x + \sum\limits_{k=0}^{\frac{h}2 -1} \beta(n,m, \frac{h}2 -k) \: \beta^{-1}(n,m,k) = 0
\] by the induction hypothesis. The elements in row $n-m-h$ and column $n-m$ that have different parities and are zero occur in the same terms in the dot product. This sum is zero because it is the entry in row $n-m-h$ and column $n-m$ of the identity matrix, and $h > 0$ by the induction hypotheses.

This is the same expression as that for finding the element in the inverse of the above matrix by back substitution in row $ \frac{n-m-h}2 $ and column $\frac{n-m}2$. This element is $x= \beta^{-1}(n,m, \frac{h}2 )$, as required.

The result holds for $l=0$ and assuming the induction hypothesis, it holds for $l = h$. Hence, by the principle of mathematical induction it holds for all $l: 0 \leq l \leq n-m$.

If $n$ and $m$ have different parities, then $\mbox{tr}_{1,0} M_{vt}$ is a change of basis matrix for a descending basis with equal parities. By property 3 of Theorem~\ref{local} truncation and inversion are commutative. Hence, the induction proof step for equal parities applies for $\mbox{tr}_{1,0} M_{vt}$ and $m_{n-m-1, n-m-1}$ of $M_{tv}$ is $\beta^{-1}(n,m+1, \frac{h}2)$, as required.

The proof can be repeated for column $n-m-1$ of $M_{vt}$ from $\mbox{tr}_{0,1}M_{vt}$ by truncating the last column and last row of $M_{vt}$ to form another similar upper triangular change of basis matrix with the other parity. This truncation has no effect on finding the remainder of the inverse from the third property of Theorem~\ref{local}. After $n-m$ such truncations in total, the inverse of the $1 \times 1$ matrix containing $\beta(m,m,0)$ is found which has the element $\beta^{-1}(m,m,0)$.
\end{proof}

\begin{cor} \label{corto2}
With the assumptions of Theorem~\ref{tri-inv} and $d=1$, let the $k_1$ ascending bases or $k_2$ descending bases be alternating bases.

The set of change of basis matrices of the form $M_{vt}$ where $v,t \in S$ is a set of alternating matrices and it is a connected groupoid of order $|S|^2$ with matrix multiplication and matrix inverse as operations.

The set of change of basis matrices of the form $M_{vt}$ where both $v$ and $t$ are ascending alternating bases is a connected sub-groupoid of order $k_1^2$ of lower-triangular alternating matrices. If both $v$ and $t$ are descending alternating bases, the set of change of basis matrices is a connected sub-groupoid of order $k_2^2$ of upper-triangular alternating matrices.
\end{cor}

\begin{proof}
From Theorem~\ref{inv-alt}, an inverse of a lower- or upper-triangular alternating matrix is also a lower- or upper-triangular alternating matrix. It is straightforward to show that the product of alternating matrices is an alternating matrix, and lower- or upper-triangularity is invariant under matrix product.
Alternation is also invariant under inverse in general. From the change of basis equation (\ref{CoB}), we can always form
\begin{equation*}
M_{vt} = M_{vr} M_{r t}
\end{equation*} where $r = \{x^m, x^{m+1}, \ldots, x^{n}\}$ and $M_{vr}$ and $M_{r t}$ are triangular alternating matrices. It follows that $M_{vr}^{-1}$ and $M_{r t}^{-1}$ are alternating triangular matrices and so $M_{vt}^{-1} = M_{r t}^{-1} M_{vr}^{-1}$ is an alternating matrix.
\end{proof}

The previous example of a change of basis with a descending alternating basis had $n=8$ and $m=2$, which have equal parities. The inverse of this change of basis matrix is for the mapping from $\{x^2, x^3, \ldots, x^7, x^8\}$ to 
\[
\{R_2^2(x), R_4^2(x), R_6^2(x), R_8^2(x)\} \cup \{R_3^3(x), R_5^3(x), R_7^3(x)\}
\] and it is

\begin{center}
\begin{doublespace}
\noindent\(\left[
\begin{array}{ccccccc}
1 & 0 & \frac{3}{4} & 0 & \frac{3}{5} & 0 & \frac{1}{2} \\
0 & 1 & 0 & \frac{4}{5} & 0 & \frac{2}{3} & 0 \\
0 & 0 & \frac{1}{4} & 0 & \frac{1}{3} & 0 & \frac{5}{14} \\
0 & 0 & 0 & \frac{1}{5} & 0 & \frac{2}{7} & 0 \\
0 & 0 & 0 & 0 & \frac{1}{15} & 0 & \frac{1}{8} \\
0 & 0 & 0 & 0 & 0 & \frac{1}{21} & 0 \\
0 & 0 & 0 & 0 & 0 & 0 & \frac{1}{56} \\
\end{array}
\right]\)
\end{doublespace}
\end{center}

From equation~(\ref{RM-desc}), the inverse mapping $\beta^{-1}$ is 
\[
\beta^{-1}(n, m, k) = \frac{n - 2k + 1}{n - k + 1}\frac{{n \choose k}}{{n \choose {\frac{n-m}2}}}
\] where $0 \leq k \leq \frac{n-m}2$, and we have that the matrix elements are defined by
\[
m_{i,j} = 
\left\{
\begin{array}{ll} 
\beta^{-1}(2+j , 2, \frac{j-i}2 ) & \mbox{if $j-i$ is even and $6-j$ is even}\\
\beta^{-1}(2+j, 3,\frac{j-i}2 ) & \mbox{if $j-i$ is even and $6-j$ is odd}\\
0 & \mbox{if $j-i$ is odd}\\
0 & \mbox{if $j< i$}
\end{array}
\right.
\] where $0 \leq i,j \leq 6$, following Theorem~\ref{inv-alt}. The elements original matrix can be found by replacing $\beta^{-1}$ with $\beta$ from equation~(\ref{MR}).

\subsection{Composing Coefficient Functions with Alternating Bases}

Finding the product from the change of basis equation (\ref{CoB})
\begin{equation*}
M_{vt} = M_{vr} M_{r t}
\end{equation*} can be optimized if $M_{vr}$ or $M_{rt}$ is an alternating change of basis matrix. This was defined in Wolfram~\cite{wolfram2107} for the case where $r, t$ and $v$ are descending bases. The product is formed from the dot products of the rows of $M_{vr}$ with the columns of $M_{rt}$. When one or both is an alternating matrix, at least one of the vectors in each dot product has zero elements in known positions. These elements can be excluded from the calculation of the dot product.

This leads to an optimization of the equations for the composition of the coefficient functions for the matrices. 

The example below is for the three cases of composition of two change of basis matrices that are for ascending bases that do not have definite parity where at least one matrix is an alternating matrix. This could be done for the nine other cases where at least one of the bases is a descending basis.

An optimized version of equation~(\ref{AND}) for ascending bases without definite parity where $M_{vr}$ is an alternating matrix is
\begin{equation} 
\alpha_3(n,m,k) = \sum\limits_{\mathclap{\substack{v = 0 \\ n-m-k-v \mbox{\small \;even}}}}^{n-m-k} \alpha_1(n, m+v,k) \; \alpha_2(n,m,n-m-v) \label{prodAlt1}
\end{equation} where $0 \leq k \leq n-m$. The number of elements in row $n-m -k$ in the lower triangle of $M_{vr}$ is $n-m-k+1$ and $v$ is the column index. Elements in a row with index $n-m-k$ in the lower triangle where $n-m-k$ and $v$ do not have the same parity are $0$ and can be excluded from the sum. 

Equation~(\ref{prodAlt1}) can be expressed in terms of the coefficient function $\beta$ for the alternating elements in the lower triangle of $M_{vr}$ by using equations~(\ref{AAeq}) and (\ref{AAneq}). If $n$ and $m$ have the same parity, then we have

\begin{equation} 
\alpha_3(n,m,k) =
\left\{\:\:\:\:\:
\begin{array}{ll} 
\sum\limits_{\mathclap{\substack{v = 0 \\ n-m-k-v \mbox{\small \;even}}}}^{n-m-k} \beta_1(n, m+v, \frac{k}2) \; \alpha_2(n,m,n-m-v) & \mbox{$k$ even}\\\\
\sum\limits_{\mathclap{\substack{v = 0 \\ n-m-k-v \mbox{\small \;even}}}}^{n-m-k} \beta_1(n-1, m+v, \frac{k-1}2) \; \alpha_2(n,m,n-m-v) & \mbox{$k$ odd}
\end{array} 
\right.
\end{equation} where $0 \leq k \leq n-m$. 

If $n$ and $m$ have different parity, then
\begin{equation} 
\alpha_3(n,m,k) =
\left\{\:\:\:\:\:
\begin{array}{ll} 
\sum\limits_{\mathclap{\substack{v = 0 \\ n-m-k-v \mbox{\small \;even}}}}^{n-m-k} \beta_1(n-1, m+v, \frac{k}2) \; \alpha_2(n,m,n-m-v) & \mbox{$k$ even}\\\\
\sum\limits_{\mathclap{\substack{v = 0 \\ n-m-k-v \mbox{\small \;even}}}}^{n-m-k} \beta_1(n, m+v, \frac{k-1}2) \; \alpha_2(n,m,n-m-v) & \mbox{$k$ odd}
\end{array}
\right.
\end{equation}

If $M_{rt}$ is also an alternating matrix, then the sum can be optimized further.

If $n$ and $m$ have the same parity, then we have

\begin{equation} 
\alpha_4(n,m,k) =
\left\{\:\:\:\:\:
\begin{array}{ll} 
\sum\limits_{\mathclap{\substack{v = 0 \\ n-m-k-v \mbox{\small \;even}}}}^{n-m-k} \beta_1(n, m+v, \frac{k}2) \; \beta_2(n,m,\frac{n-m-v}2) & \mbox{$k$ even}\\\\
0 & \mbox{$k$ odd.}
\end{array}
\right.
\end{equation} 

If $n$ and $m$ have different parity, then
\begin{equation} 
\alpha_4(n,m,k) =
\left\{\:\:\:\:\:
\begin{array}{ll} 
0 & \mbox{$k$ even}\\\\
\sum\limits_{\mathclap{\substack{v = 0 \\ n-m-k-v \mbox{\small \;even}}}}^{n-m-k} \beta_1(n, m+v, \frac{k-1}2) \; \beta_2(n,m,\frac{n-m-v-1}2) & \mbox{$k$ odd.}
\end{array}
\right.
\end{equation}

If $M_{vr}$ is not an alternating matrix and $M_{rt}$ is an alternating matrix, then
\begin{equation} 
\alpha_5(n,m,k) = 
\left\{\:\:\:\:\:
\begin{array}{ll} 
\sum\limits_{\mathclap{\substack{v = 0 \\ v \mbox{\small \;even}}}}^{n-m-k} \alpha_1(n, m+v,k) \; \beta_2(n,m,\frac{n-m-v}2) & \mbox{$n-m$ even}\\\\
\sum\limits_{\mathclap{\substack{v = 0 \\ v \mbox{\small \;even}}}}^{n-m-k} \alpha_1(n, m+v,k) \; \beta_2(n,m,\frac{n-m-v-1}2) & \mbox{$n-m$ odd.}
\end{array}
\right.
\end{equation}

\section{Superposition}

Superposition involves adding two polynomials with different parities to produce one that does not have definite parity.

\begin{definition} \label{superposed}
Let $g(x)_{n,m}$ and $h(x)_{n,m}$ be polynomials over $\mathbb{R}[x]$ which have definite parity and where $n \geq m \geq 0$. 

If a polynomial of the form $g(x)_{u,l}$ has degree $u$ and a polynomial of the form $h(x)_{u,l}$ has degree $u$, then the descending
superposed polynomial is defined by
\begin{equation} \label{gdef}
f(x)_{n,m} = 
\left\{ 
\begin{array}{ll} 
g(x)_{m,m} & \mbox{if $m=n$}\\
g(x)_{n,m} + h(x)_{n-1,m+1}& \mbox{if $n> m$ and $n \equiv m \pmod 2$}\\
g(x)_{n-1, m} + h(x)_{n,m+1}& \mbox{if $n> m$ and $n \not\equiv m \pmod 2$}
\end{array} \right.
\end{equation} 

If a polynomial of the form $g(x)_{u,l}$ has minimum degree $l$ and a polynomial of the form $h(x)_{u,l}$ has minimum degree $l$, then the ascending superposed polynomial is defined by
\begin{equation} \label{gadef}
f(x)_{n,m} = 
\left\{ 
\begin{array}{ll} 
g(x)_{n,n} & \mbox{if $n=m$}\\
g(x)_{n,m} + h(x)_{n-1,m+1}& \mbox{if $n> m$ and $n \equiv m \pmod 2$}\\
g(x)_{n, m+1} + h(x)_{n-1,m}& \mbox{if $n> m$ and $n \not\equiv m \pmod 2$}
\end{array} \right.
\end{equation} 
\end{definition}

Two examples of descending superposed polynomials are the Chebyshev polynomials of the third and fourth kinds. They can be defined~\cite[equations (1.17)--(1.18)]{handscomb} by
\begin{align*}
V_n(x) =& U_n(x) - U_{n-1}(x)\\
W_n(x) =& U_n(x) + U_{n-1}(x)
\end{align*} where $V_0(x) = 1$ and $W_0(x) = 1$. 

From Definition~\ref{superposed}, we have $m=0$ because the least degree of $V_n(x)$ nad $W_n(x)$ is $0$ where $n \geq 0$. For $V_n(x)$, $g_{n,0}(x) = U_n(x)$ where $n$ is even, and $g_{n,1}(x) = U_n(x)$ where $n$ is odd. The polynomial $h$ is $h_{n-1, 1}(x) = -U_{n-1}(x)$ where $n-1$ is odd, and $h_{n-1, 0}(x) = - U_{n-1}(x)$ where $n-1$ is even. The definition for $W_n(x)$ is similar except that $h$ has the other sign.

Another example, is the change of basis matrix from $\{x^2, x^3, x^4, x^5\}$ to the ascending basis
$\{ R_5^3(x) + R_4^2(x), R_5^3(x) + R_4^4(x), R_5^5(x) + R_4^4(x), R_5^5(x)\}$ that spans the same vector space:

\begin{center}
\begin{doublespace}
\noindent\(\left[
\begin{array}{cccc}
-\frac13 & \phantom{-}0 & \phantom{-}0& 0 \\
\phantom{-}\frac13& -\frac14 & \phantom{-}0 & 0 \\
\phantom{-}1 & \phantom{-}\frac14 & \phantom{-}1 & 0 \\
-1 & \phantom{-}1 & -1 & 1 \\
\end{array}
\right]\)
\end{doublespace}
\end{center}

Its product with a change of basis matrix for truncated Laguerre polynomials 
from $\{{L_3}_(x)_{2,2}, {L_4}(x)_{3,2}, {L_5}(x)_{4,2}, {L_6(x)}_{5,2}\}$ to $\{x^2, x^3, x^4, x^5\}$ gives

\begin{doublespace}
\begin{equation*}
\begin{bmatrix}
-\frac13 & \phantom{-}0 & \phantom{-}0& 0 \\
\phantom{-}\frac13& -\frac14 & \phantom{-}0 & 0 \\
\phantom{-}1 & \phantom{-}\frac14 & \phantom{-}1 & 0 \\
-1 & \phantom{-}1 & -1 & 1 \\
\end{bmatrix}
\begin{bmatrix}
\phantom{-}\frac32 & \phantom{-}3 & \phantom{-}5 & \phantom{-}\frac{15}2 \\
\phantom{-}0 & -\frac23 & -\frac{5}3 & -\frac{10}3 \\
\phantom{-}0 & \phantom{-}0 & \phantom{-}\frac{5}{24} & \phantom{-}\frac{5}{8} \\
\phantom{-}0 & \phantom{-}0 & \phantom{-}0 & -\frac{1}{20} \\
\end{bmatrix}
=
\begin{bmatrix}
-\frac12 & -1 & -\frac53& -\frac52 \\
\phantom{-}\frac12& \phantom{-}\frac76 & \phantom{-}\frac{25}{12} & \phantom{-}\frac{10}{3} \\
\phantom{-}\frac32 & \phantom{-}\frac{17}{6} & \phantom{-}\frac{115}{24} & \phantom{-} \frac{175}{24} \\
-\frac32 & -\frac{11}3 & -\frac{55}{8} & -\frac{1381}{120} \\
\end{bmatrix}
\end{equation*}
\end{doublespace}

From the second column of the product,
\begin{align*}
{L_4}(x)_{3,2} =& -\frac{11}{3} R_5^5(x) + \frac{17}{6} ( R_5^5(x) + R_4^4(x)) + \frac{7}{6} (R_5^3(x) + R_4^4(x))\\
& - (R_5^3(x) + R_4^2(x))\\
=& \frac1{24}(- 16 x^3 + 72 x^2).
\end{align*}

In general, from Definition \ref{superposed}, if the range basis of a mapping comprises superposed polynomials, the even and odd components can be separated in terms of the polynomials $g$ and $h$. In the previous example, the odd component of ${L_4}(x)_{3,2}$ is
\[
-\frac23 x^3 = (-\frac{11}{3} + \frac{17}{6}) R_5^5(x) + (\frac{7}{6} -1) R_5^3(x).
\]

\subsubsection*{Superposition and Alternation}

\begin{theorem} \label{conv-supalt}
Superposition and alternation are converse operations. 
\end{theorem}

\begin{proof}
Let $v$ and $t$ be descending alternating bases that span the same vector space as $\{x^m, x^{m+1}, \ldots , x^n\}$.
The matrix $M_{vt}$ is a change of basis matrix for a descending alternating basis $t$:
\[
\{g_{m,m}(x), g_{m+2,m}(x), \ldots g_{k, m}(x)\} \cup \{h_{m+1, m+1}(x), h_{m+3, m+1}(x) , \ldots, h_{l, m+1}(x)\}
\] where
\begin{itemize}
\item $k = n$ and $l=n-1$ if $n$ and $m$ have equal parity
\item $k=n-1$ and $l =n $ if $n$ and $m$ do not have equal parity.
\end{itemize}

From Theorem~\ref{tri-inv}, the basis $v$ has the following form.
\[
\{p_{m,m}(x), p_{m+2,m}(x), \ldots p_{k, m}(x)\} \cup \{p_{m+1, m+1}(x), p_{m+3, m+1}(x) , \ldots, p_{l, m+1}(x)\}
\]

Consider the change of basis matrix for a descending alternating basis with unequal parity, e.g., from Figure~\ref{MDAltdiffer}, and let $w = \frac{n-m-1}2$:
\begin{figure}[ht]
\begin{gather*}
\begin{bmatrix}
\beta(m,m, 0) & \cdots & \beta(n-1, m, w) & 0 \\
0 & \cdots & 0 & \beta(n,m+1, w) \\
\cdots & \cdots & \cdots & \cdots \\
0 & \cdots & \beta(n-1, m, 1)& 0 \\
0 &\cdots & 0& \beta(n ,m+1, 1) \\
0 & \cdots & \beta(n-1,m,0) & 0 \\
0 & \cdots & 0& \beta(n, m+1, 0) \\
\end{bmatrix}
\end{gather*} 
\end{figure}

Column $j$ where $0 \leq j \leq n-m$ of this matrix is the transposed coordinate vector of a basis polynomial of $t$ with respect to $v$. 
If $j$ has the same parity as $m$, then
\begin{equation} \label{gsum}
g(x)_{m+j,m} = \sum\limits_{i=0}^{\frac{j}2} \beta(m+j,m,i) p(x)_{m+j - 2 i , m}.
\end{equation}

If $j$ does not have the same parity as $m$, then
\begin{equation} \label{hsum}
h(x)_{m+j,m+1} = \sum\limits_{i=0}^{\frac{j-1}2} \beta(m+j,m+1,i) p(x)_{m+j - 2 i , m+1}.
\end{equation}

We can form an upper triangular matrix from it using the following equations on columns.
\begin{align*}
d_0 =& c_0\\
d_j =& c_j + c_{j-1}
\end{align*} 

The columns $d_j$ where $0 \leq j \leq m-n$ are the columns of the new matrix and $c_j$ are the columns of the original matrix where $0 \leq j \leq m-n$.
This gives a change of basis matrix for a superposed polynomial basis.
\begin{figure}[H]
\resizebox{\textwidth}{!}{%
$\displaystyle
\left[
\begin{array}{rrrrr}
\beta(m,m, 0) & \cdots & \beta(n-3,m, w-1) &\beta(n-1, m, w) & \beta(n-1, m, w)\\
0 & \ldots & \beta(n-2,m+1, w-1)& \beta(n-2,m+1, w-1) & \beta(n,m+1, w)\\
\ldots & \ldots & \ldots & \ldots & \ldots \\
0 & \cdots & \beta(n-3, m, 0) & \beta(n-1, m, 1)& \beta(n-1, m, 1)\\
0 &\cdots & \beta(n-2, m+1, 0) & \beta(n-2, m+1, 0)& \beta(n ,m+1, 1)\\
0 & \cdots & 0 & \beta(n-1, m, 0)& \beta(n-1, m, 0)\\
0 & \cdots & 0 & 0 & \beta(n, m+1, 0)
\end{array}
\right]
$}
\end{figure}

We show that the columns $d_j$ of this matrix are the transposed coordinate vectors of 
\[
\{f(x)_{m,m}, f(x)_{m+1,m+1}, \ldots, f(x)_{n-1, m+1}, f(x)_{n,m}\}.
\] in the basis $v$.

When $j=0$, we have $g(x)_{m,m} = \beta(m,m,0) p(x)_{m,m}$, and from equation~(\ref{gsum}), $g(x)_{m,m} = f(x)_{m,m}$.

If $j > 0$ and $j$ is even, then from equations~(\ref{gsum}) and (\ref{hsum}), we have that column $d_j$ is the transposed coordinate vector of
$g(x)_{m+j,m} + h(x)_{m+j-1, m+1}$ in the basis $v$. From equation~(\ref{gdef}), this sum is $f(x)_{m+j, m}$.

If $j > 0$ and $j$ is odd, then from equations~(\ref{gsum}) and (\ref{hsum}), we have that column $d_j$ is the transposed coordinate vector of $h(x)_{m+j, m+1} + g(x)_{m+j-1, m}$ in the basis $v$. From equation~(\ref{gdef}), this sum is $f(x)_{m+j,m}$, as required.

\medskip

Similarly, for an ascending alternating change of basis matrix which has columns $c_j$, we can form the lower triangular matrix with columns $d_j$ by using the equations
\begin{align*}
d_{n-m} =& c_{n-m}\\
d_j =& c_j + c_{j+1}
\end{align*} where $0 \leq j < n-m$.

Let $v$ and $t$ be ascending alternating bases that span the same vector space as $\{x^m, x^{m+1}, \ldots , x^n\}$.
The matrix $M_{vt}$ is a change of basis matrix for an ascending alternating basis $t$ where 
\[
t = \{g_{n,n}(x), g_{n,n-2}(x), \ldots g_{n, k}(x)\} \cup \{h_{n-1, n-1}(x), h_{n-1, n-3}(x) , \ldots, h_{n-1, l}(x)\}
\] and 
\begin{itemize}
\item $k = m$ and $l=m+1$ if $n$ and $m$ have equal parity
\item $k=m+1$ and $l =m$ if $n$ and $m$ do not have equal parity.
\end{itemize}

From Theorem~\ref{tri-inv}, the basis $v$ has the form
\[
\{p_{n,n}(x), p_{n,n-2}(x), \ldots p_{n, k}(x)\} \cup \{p_{n-1, n-1}(x), p_{n-1, n-3}(x) , \ldots, p_{n-1, l}(x)\}.
\] 

We show that the columns $d_j$ of this matrix are the transposed coordinate vectors of 
\[
\{f(x)_{n,n}, f(x)_{n,n-1}, \ldots, f(x)_{n, m+1}, f(x)_{n,m}\}.
\] in the basis $v$.

Column $j$ where $0 \leq j \leq n-m$ of this matrix is the transposed coordinate vector of a basis polynomial of $t$ with respect to $v$. 
If $j$ has the same parity as $n-m$, then
\begin{equation} \label{gasum}
g(x)_{n,m+j} = \sum\limits_{i=0}^{\frac{n-m-j}2} \beta(n, m+j, i) p(x)_{n, n- 2 i}.
\end{equation}

If $j$ does not have the same parity as $n-m$, then
\begin{equation} \label{hasum}
h(x)_{n-1, m+j} = \sum\limits_{i=0}^{\frac{n-m-j-1}2} \beta(n-1, m+j, i) p(x)_{n-1, n-1 - 2 i}.
\end{equation}

When $j=n-m$, we have $g(x)_{n,n} = \beta(n,n,0) p(x)_{n,n}$, and from equation~(\ref{gasum}), $g(x)_{n,n} = f(x)_{n,n}$.

If $j < n-m$ and $j$ has the same parity as $n-m$, then from equations~(\ref{gasum}) and (\ref{hasum}), we have that column $d_j$ is the transposed coordinate vector of
$g(x)_{n,m+j} + h(x)_{n-1, m+j+1}$ in the basis $v$. From equation~(\ref{gadef}), this sum is $f(x)_{n, m+j}$.

If $j < n-m$ and $j$ does not have the same parity as $n-m$, then from equations~(\ref{gasum}) and (\ref{hasum}), we have that column $d_j$ is the transposed coordinate vector of $h(x)_{n-1, m+j} + g(x)_{n, m+j+1}$ in the basis $v$. From equation~(\ref{gadef}), this sum is $f(x)_{n, m+j}$, as required.
\end{proof}

\subsection{From Bernstein to Zernike}

We can use superposition to form an ascending polynomial basis that does not have definite parity from the Zernike Radial Polynomials. This basis is
\[
F = \{f_{7,3}(x), f_{7,4}(x), f_{7,5}(x), f_{7,6}(x), f_{7,7}(x) \}
\] where
\begin{equation*} \label{fdef}
f_{7,m}(x) = 
\left\{
\begin{array}{ll} 
R_7^7(x) & \mbox{if $m=7$}\\
R_7^m(x) + R_6^{m+1}(x)& \mbox{if $m < 7$ and $7 \equiv m \pmod 2$}\\
R_7^{m+1}(x) + R_6^m(x)& \mbox{if $m < 7$ and $7 \not\equiv m \pmod 2$}\\
0 & \mbox{if $m > 7$}
\end{array} \right.
\end{equation*} 

The product of the change of basis matrix from $\{x^3, x^4, \ldots, x^6, x^7\}$ to $F$ with the change of basis matrix from $\{b_3^7(x), b_4^7, b_5^7(x), b_6^7, b_7^7\}$ to $\{x^3, x^4, \ldots, x^6, x^7\}$ is

\begin{center}
\begin{doublespace}
\noindent\(\left[
\begin{array}{ccccc}
\frac{1}{10} & 0 & 0 & 0 & 0 \\
-\frac{1}{10} & -\frac{1}{5} & 0 & 0 & 0 \\
-\frac{2}{5} & \frac{1}{5} & -\frac{1}{6} & 0 & 0 \\
\frac{2}{5} & 1 & \frac{1}{6} & 1 & 0 \\
1 & -1 & 1 & -1 & 1 \\
\end{array} \right]
\left[
\begin{array}{ccccc}
35 & 0 & 0 & 0 & 0 \\
-140 & 35 & 0 & 0 & 0 \\
210 & -105 & 21 & 0 & 0 \\
-140 & 105 & -42 & 7 & 0 \\
35 & -35 & 21 & -7 & 1 \\
\end{array}
\right]
=
\left[
\begin{array}{ccccc}
\frac{7}{2} & 0 & 0 & 0 & 0 \\
\frac{49}{2} & -7 & 0 & 0 & 0 \\
-77 & \frac{49}{2} & -\frac{7}{2} & 0 & 0 \\
-231 & \frac{245}{2} & -\frac{77}{2} & 7 & 0 \\
560 & -280 & 84 & -14 & 1 \\
\end{array}
\right]
\)
\end{doublespace}
\end{center}

From the second column of this matrix, we have 
\begin{align*}
b_4^7(x) =& -280 f_{7, 7}(x) + \frac{245}2 f_{7,6}(x) + \frac{49}2 f_{7, 5}(x) -7 f_{7,4}(x)\\
=& -280 R_7^7(x) + \frac{245}2 (R_7^7(x) + R_6^6(x)) + \frac{49}2(R_7^5(x) + R_6^6(x)) \\
&-7 (R_7^5(x) + R_6^4(x))\\
= & 35 x^4-105 x^5+105 x^6-35 x^7\\
=& {7 \choose 4} (1-x)^3 x^4.
\end{align*}

The components with even parity are
\begin{align*}
147 R_6^6(x) -7 R_6^4(x) =& 35 x^4 + 105 x^6 
\end{align*}

The components with odd parity are
\begin{align*}
(-280 + \frac{245}2) R_7^7(x) + (\frac{49}2 -7) R_7^5(x) =& -105 x^5 -35 x^7.
\end{align*}

\section{Summary with Categories}

We use category theory, e.g., \cite{walters}, to analyse change of basis groupoids. 

\begin{lemma}
A change of basis groupoid is a small category { CB}.
The set 
$\mbox{\rm obj\;CB}$ is a finite set of polynomial bases. The objects of CB are these bases. Each basis in this set spans the same vector space.

Its morphisms are change of basis matrices such that for every $v, t \in \mbox{\rm obj\;{CB}}$ there is a unique $M_{vt} \in \mbox{\rm hom\;{CB}}$ such that $M_{vt}: v \rightarrow t$.
\end{lemma}

\begin{proof}
It is straightforward to check that $\mbox{\rm obj\;{CB}}$ is a small category. Given $M_{vr}, M_{rt} \in \mbox{\rm hom\;{CB}}$ there is a composite morphism $M_{vt} \in \mbox{\rm hom\;{CB}}$ defined by $M_{vr} \circ M_{rt} = M_{vr} M_{rt}$, i.e., matrix product.

For every object $v \in \mbox{\rm obj\;{CB}}$, there is an identity morphism $M_{vv}:v \rightarrow v$. This is the identity change of basis matrix.

The identity laws are satisfied. For every $M_{vt} \in \mbox{\rm hom\;{CB}}$, $M_{vv} \circ M_{vt} = M_{vt}$ and $M_{vt} \circ M_{tt} = M_{vt}$.
The associative law is also satisfied. For every $M_{vr}, M_{rt}, M_{tu} \in \mbox{\rm hom\;{CB}}$, $M_{vr} \circ (M_{rt} \circ M_{tu}) = (M_{vr} \circ M_{rt}) \circ M_{tu})$ because matrix product is associative.
\end{proof}

\begin{lemma}
Give a category { CB}, it has full subcategories
\begin{itemize}
\item {\rm UTr} where $\mbox{\rm obj\;{UTr}}$ are descending bases and $\mbox{\rm hom\;{UTr}}$ are upper-triangu\-lar matri\-ces
\item {\rm LTr} where $\mbox{\rm obj\;{LTr}}$ are ascending bases and $\mbox{\rm hom\;{LTr}}$ are lower-triangular matrices
\item {\rm Alt} where $\mbox{\rm obj\;{Alt}}$ are alternating bases and $\mbox{\rm hom\;{Alt}}$ are alternating matrices
\item {\rm UAlt} where $\mbox{\rm obj\;{UAlt}}$ are descending alternating bases and $\mbox{\rm hom\;{UAlt}}$ are upper-triangular alternating matrices
\item {\rm LAlt} where $\mbox{\rm obj\;{LAlt}}$ are ascending alternating bases and $\mbox{\rm hom\;{LAlt}}$ are lower-triangular alternating matrices
\end{itemize}

The subcategory relations are given in the following diagram.
\begin{center}
\begin{tikzcd}
\rm UTr \arrow[d] & \rm CB \arrow[d] \arrow[l] \arrow[r] & \rm LTr \arrow[d] \\
\rm UAlt & \rm Alt \arrow[r] \arrow[l] & \rm LAlt 
\end{tikzcd}
\end{center}
\end{lemma}

\begin{proof}
The objects of the categories are super-sets of the objects of their sub-categories.
Change of basis matrices between ascending bases and descending bases have lower- and upper-triangular change of basis matrices from Section~\ref{CoBS}.
Morphism inversion and composition preserve upper- and lower-triangularity from Theorem~\ref{tri-inv}. They preserve alternation from Corollary~\ref{corto2}. The sub-categor\-ies are full sub-categories because a morphism between any two objects in a category or sub-category of the category exists and it is unique. 
\end{proof}

\begin{definition}
The small categories $\mbox {\rm UTr}_n$, $\mbox {\rm LTr}_n$, $\mbox {\rm UAlt}_n$, $\mbox {\rm LAlt}_n$ are respectively 
{\rm UTr, LTr, UAlt, LAlt} whose morphisms are $n \times n$ change of basis matrices.
\end{definition}

Truncation of change of basis matrices is a functor from the categories with upper- or lower-triangular change of basis matrices to the categories whose morphisms are truncated forms of these matrices.

\begin{theorem}
Let $A_n$, be one of $\mbox {\rm UTr}_n$, $\mbox {\rm LTr}_n$, $\mbox {\rm UAlt}_n$, or $\mbox {\rm LAlt}_n$. There is a covariant functor $T_{k_1, k_2}$ such that 
\[
T_{k_1, k_2}(A_n) = A_{n - k_1 - k_2}
\] where $0 \leq k_1+ k_2 \leq n$ and
\begin{itemize}
\item for every basis of the form $\{f_1(x), \ldots, f_n(x)\} \in \mbox{\rm obj\;A}_n$ ordered by increasing degree
\[
T_{k_1, k_2}(\{f_1(x), \ldots, f_n(x)\}) = \{f_{k_1+1}(x)_{k_1, k_2}, \ldots, f_{n-k_2}(x)_{k_1, k_2}\}
\]
\item for every morphism $M_{vt} \in \mbox{\rm hom\;{A}}_n$, 
\[
T_{k_1, k_2}(M_{vt}) = tr_{k_1, k_2}(M_{vt}).
\]
\end{itemize}
\end{theorem}

\begin{proof}
Truncation of a polynomial $f(x)_{u, l}$ was defined in Definition \ref{parcoeff}, and truncation of a change of basis matrix was defined in Section~\ref{truncsec}. From Definition~\ref{parcoeff}, we can show that if $t$ has the form $\{f_1(x), \ldots, f_n(x)\}$ and $v$ has the form $\{g_1(x), \ldots, g_n(x)\}$, then $tr_{k_1, k_2}(M_{vt})$ is the morphism that maps 
\[\{f_{k_1+1}(x)_{k_1, k_2}, \ldots, f_{n-k_2}(x)_{k_1, k_2}\} \mbox{ to } \{g_{k_1+1}(x)_{k_1, k_2}, \ldots, g_{n-k_2}(x)_{k_1, k_2}\}.
\]

From Theorem~\ref{local}, we have $T_{k_1, k_2}(M_{vr} M_{rt}) = T_{k_1, k_2}(M_{vr}) \;T_{k_1, k_2}(M_{rt})$ so that $T_{k_1, k_2}$ is covariant. Also, for every identity morphism $M_{vv}$ in $A_n$, we have $T_{k_1, k_2}(M_{vv})$ is the identity morphism from $T_{k_1, k_2}(v)$ to itself in $A_{n - k_1 - k_2}$. Hence $T_{k_1, k_2}$ satisfies the definition of a covariant functor from $A_n$ to $A_{n- k_1 - k_2}$.
\end{proof}

Superposition is not a functor. It is used to define a basis from an alternating basis, or a change of basis matrix from an alternating change of basis matrix. This is a counter-example. Suppose that $S$ is superposition functor from LAlt and we have the following morphisms in hom\;LAlt:

\begin{centering}
\begin{doublespace}
\noindent\(\left[
\begin{array}{ccc}
1 & 0 & 0 \\
0 & 2 & 0 \\
4 & 0 & 3 \\
\end{array}
\right] \left[
\begin{array}{ccc}
2 & 0 & 0 \\
0 & 2 & 0 \\
5 & 0 & 2 \\
\end{array}
\right] = 
\left[
\begin{array}{ccc}
2 & 0 & 0 \\
0 & 4 & 0 \\
23 & 0 & 6 \\
\end{array}
\right].\)
\end{doublespace}
\end{centering}

However,

\begin{centering}
\begin{doublespace}
\noindent\(\left[
\begin{array}{ccc}
1 & 0 & 0 \\
2 & 2 & 0 \\
4 & 3 & 3 \\
\end{array}
\right]\left[
\begin{array}{ccc}
2 & 0 & 0 \\
2 & 2 & 0 \\
5 & 2 & 2 \\
\end{array}
\right] = 
\left[
\begin{array}{ccc}
2 & 0 & 0 \\
8 & 4 & 0 \\
29 & 12 & 6 \\
\end{array}
\right]\)
\end{doublespace}
\end{centering}
and
\begin{centering}
\begin{doublespace}
\(S\left[
\begin{array}{ccc}
2 & 0 & 0 \\
0 & 4 & 0 \\
23 & 0 & 6 \\
\end{array}
\right] \not=
\left[
\begin{array}{ccc}
2 & 0 & 0 \\
8 & 4 & 0 \\
29 & 12 & 6 \\
\end{array}
\right].\)
\end{doublespace}
\end{centering}

\section{Conclusion}

This paper builds on earlier results for change of basis~\cite{wolfram2107,wolfram2108} that mainly concerned bases of classical orthogonal polynomials. In the framework introduced here, these bases are descending bases and each polynomial implicitly has $0$ or $1$ as its minimum degree when expressed using the monomials. 

In Part I, a framework for classifying changes of basis between polynomial bases included ascending bases.  These bases can either have definite parity or have no parity. Unlike classical orthogonal polynomials, the minimum degree of polynomials in an ascending basis is specified explicitly and it can be greater than $1$. The change of basis matrices for mappings between descending bases and ascending bases were then defined. 

We defined four general kinds of change of basis matrices: two between ascending bases depending on parity, and two between descending bases. Mappings between descending bases have upper triangular change of basis matrices, those between ascending bases have lower triangular ones.

The main result is in Theorem~\ref{tri-inv} that given a set of ascending or descending bases that span the same vector space, the sets of all change of basis matrices between ascending bases or between descending bases form connected sub-groupoids of lower- or upper-triangular matrices, respectively.

We then showed that Bernstein polynomials and Zernike Radial polynomials can be used to form ascending and descending bases. Coefficient functions~\cite{wolfram2107} are functions that evaluate to connection coefficients. We give two new coefficient functions for the mappings from monomials to a descending basis of Bernstein polynomials, and to an ascending polynomial basis of Zernike Radial polynomials in equations~(\ref{adesc}) and (\ref{b4}). These mappings are verified in Theorems~\ref{invdesa} and~\ref{beta-RM}.

In Part II, we used all products of the four change of basis matrices defined in Part I to define eight general change of basis matrices. Each product satisfies the change of basis equation~(\ref{CoB}). For simplicity, the exchange basis $r$ is often chosen to be the monomials. A motivation for doing this is to reduce the number of permutations of domain and range bases from $n^2$ to $2n$ by analogy with spoke-hub distribution where the monomials form the ``hub''. 
Equations for the elements of the products of the matrices and the compositions of the coefficient functions were given. There is a summary in Table~\ref{CoBM}.

We give two examples of change of basis with Bernstein polynomials. The main one is from Farouki~\cite{farouki2000} and it concerns the change of basis from shifted Legendre polynomials to Bernstein polynomials. An open problem from this work was to find a closed form expression for a function that evaluates to the connection coefficients. 

Using the techniques introduced in Part II, we show that the change of basis is a product of a descending basis and an ascending basis without definite parity. This factoring enables us to find an expression for the composition of coefficient functions. By using Gosper's algorithm with this expression, we show that it has no general closed form expression, which is consistent with Farouki's conjecture about this problem~\cite[\S 3]{farouki}. 

However, by using Zeilberger's algorithm, we give recurrence relations for the coefficient functions for rows and columns of the change of basis matrix. The techniques also yield a hypergeometric function for the elements of the change of basis matrix. 
This expression is more general and simpler than a previous one~\cite{farouki} which is only defined for the elements of the upper-skew triangle. 

By using this hypergeometric function, we show that the coefficient function for columns is equivalent to the Lagrange interpolation polynomial for the column elements.

In Part III, we defined three methods for defining polynomial bases from others: truncation, alternation and superposition. We showed that truncation and matrix inversion are commutative. This can be used to reduce the complexity of finding the inverse matrix of a truncated change of basis matrix.

We introduced alternating matrices that are change of basis matrices between alternating bases. These have zero elements in alternate positions. Since alternation is preserved by matrix multiplication and inversion, we have, as a corollary of Theorem~\ref{tri-inv}, that change of basis matrices between lower- or upper-triangular alternating matrices are connected sub-groupoids.

We also showed that that alternation and superposition are converse operations. Alternation and superposition produce polynomial bases without definite parity from ones that at have definite parity. The inverse of a change of basis matrix where the domain or range is an alternating basis and both are descending or ascending bases, can be found from the inverse coefficient function of the alternate non-zero elements. 

Truncation can be applied to change of basis matrices formed using alternation and superposition. 

Using category theory, we showed that alternating matrices, and combinations of alternating and upper- and lower-triangular matrices are full subcategories of the small category of the change of basis groupoid. Truncation is a covariant functor. Superposition enables bases to be defined from alternating ones, but it is not a functor.

\section*{Acknowledgment}
I am grateful to the College of Engineering \& Computer Science at The Australian National University for research support.

\begin{appendix}
\section{Coefficient Functions by Column and Row}
\label{app1}

These are the first coefficient functions by column for the change of basis matrix from shifted Legendre polynomials to Bernstein polynomials where $j: 0 \leq j \leq n$ is the column index and $i: 0 \leq i \leq n$ is the row index.

\subsection{Columns}

When $j > 1$, the coefficient functions can be found by applying the recurrence of equation~(\ref{reccol}) found from Zeilberger's algorithm:

\begin{align*}
&(1+j)(2+j+n)\mbox{SUM}[j] - (3+2j)(2i -n) \mbox{SUM}[1+j] \\
&- (2+j) (1+j -n)\mbox{SUM}[2+j] = 0.
\end{align*} 

In the list below, $\alpha(n, j, n-i) = \mbox{SUM}[j]$, 
\begin{itemize}
\item $\alpha(n, 0, n-i) = 1$. This is from equation~(\ref{pfq0}).
\item $\alpha(n, 1, n-i) = \frac{2i}n - 1$. This is from equation~(\ref{pfq0}).
\item $\alpha(n, 2, n-i) = \frac{n^2 - (6i + 1) n + 6i^2}{n(n-1}$
\item $\alpha(n, 3, n-i) = \frac{(n-2i)(n^2 - (10i+3)n +10 i^2 + 2)}{n(n-1)(n-2)}$
\item $\alpha(n,4,n-i) = \frac{n^4 - (20i +6)n^3 + (90 i^2 + 30i + 11)n^2 - (140 i^3 + 30 i^2 + 50 i+6)n +(70i^4 + 50i^2) }{n(n-1)(n-2)(n-3)}$
\end{itemize}

We can also apply the recurrence to find $ \mbox{SUM}[j]$ where $j \leq n-2$. 
\begin{itemize}
\item $\alpha(n, n, n-i) = (-1)^{n+i} {n \choose i}$. This follows from Li and Zhang~\cite[equation (6)]{li} and Farouki~\cite[equation (13)]{farouki2000}.
\item $\alpha(n, n-1, n-i) = (-1)^{n+i}{n \choose i} \frac{2i-n}n $. This is from equation~(\ref{alpha-pen}) of Lemma~\ref{lemm-pen}.
\item $\alpha(n, n-2,n-i) = (-1)^{n+i} \frac{{n \choose i}}{n^2(n-1)}(n^3 - (4i +1)n^2 + 2i(2i +1)n - 2i^2)$
\item $\alpha(n, n-3, n-i) = (-1)^{n+i-1} \frac{{n \choose i}}{n^2(n-1)(n-2)}(n^3 - (4i +3)n^2 + (2i+1)(2i +2)n - 6i^2)(n-2i)$
\end{itemize}

\subsection{Rows}
The following coefficient functions satisfy the recurrence of equation~(\ref{recrow}) found from Zeilberger's algorithm.
\begin{align*} 
&-(1+i)(i-n)\mbox{SUM}[i] + (1 + 4i + 2i^2 +j +j^2 - 3n - 2i n) \mbox{SUM}[1+i] \nonumber\\
&- (2+i) (1+i -n)\mbox{SUM}[2+i] = 0.
\end{align*} 

In the list below, $\alpha(n, j, n-i) = \mbox{SUM}[i]$. These equations follow from equation~(\ref{pfq0}).
\begin{itemize}
\item $\alpha(n, j, n) = (-1)^j$.
\item $\alpha(n, j, n-1) = (-1)^j( 1- \frac{j(j+1)}n)$.
\item $\alpha(n, j, n-2) = (-1)^j( 1- \frac{2j(j+1)}n + \frac{(1-j)j(1+j)(2+j)}{2(1-n)n})$.
\item $\alpha(n, j, n-3) = (-1)^j( 1- \frac{3j(j+1)}n + \frac{3(1-j)j(1+j)(2+j)}{2(1-n)n} -$ \\
$\frac{(1-j)(2-j)j(1+j)(2+j)(3+j)}{6(1-n)(2-n)n}).$
\end{itemize}

From Farouki~\cite{farouki2000}, the rows of the change of basis matrix have the property
\[
\noindent
m_{n-i, j} = (-1)^j m_{i, j} \:\: \mbox{for $0 \leq j \leq n$ and $0 \leq i \leq \lfloor \frac{n}2 \rfloor$}
\]
so that $\alpha(n, j, i) = (-1)^j \alpha(n, j, n-i)$ where $0 \leq j \leq n$ and $0 \leq i \leq \lfloor \frac{n}2 \rfloor$.

\section{Proof of the Induction Step of Theorem~\ref{invdesa}} \label{appfor4}
The right side of the equation is
\begin{align*}
=& -\frac{\sum\limits_{l=0}^{h-1} 
{{n-l} \choose {n -h}}{{n -h}\choose m} (-1)^{n -h - m} \; (-1)^{n- m -l} \frac{{n \choose l}}{{n \choose m}}}{{{n-h} \choose {n-h}}{{n-h}\choose m} (-1)^{n -h - m}}\\
=& -\sum\limits_{l=0}^{h-1} 
{{n-l} \choose {n -h}} \; (-1)^{n- m -l} \frac{{n \choose l}}{{n \choose m}}\\
=& -\frac{ (-1)^{n- m}}{ {n \choose m}} \sum\limits_{l=0}^{h-1} 
{{n-l} \choose {n -h}} \; (-1)^{-l} {n \choose l} \\
=& -\frac{ (-1)^{n- m}}{ {n \choose m}} \frac{n!}{(n-h)!} \sum\limits_{l=0}^{h-1} 
\frac{(-1)^{l}}{l! (h-l)!}
\end{align*}

We use the identity from Gradshteyn and Ryzhik~\cite[equation 0.151.4]{zwillinger}:
\[
\sum\limits_{k=0}^m (-1)^k {n \choose k} = (-1)^m {{n-1} \choose m}\:\: \mbox{ where $n \geq 1$}
\] so that
\begin{align*}
=& -\frac{ (-1)^{n- m}}{ {n \choose m}} \frac{n!}{(n-h)!} \frac{(-1)^{1-h}}{h!}\\
=& (-1)^{n- m - h} \frac{{n \choose h}}{ {n \choose m}}, \: \mbox{ as required.}\\
\end{align*}

\section{Proof of the Induction Step of Theorem~\ref{RMD}} \label{appfor5}
The induction hypothesis is that the result holds for all $l: 0 \leq l < h \leq \frac{n - m}2$. We need to show that
\[
\frac{n - 2h + 1}{n - h + 1}\frac{{n \choose h}}{{n \choose {\frac{n-m}2}}} = -\frac{\sum\limits_{l=0}^{h-1} \beta_1(n-2l, m, h-l) \; \beta(n,m,l)}{\beta_1(n-2h, m, 0)}
\] where $\beta_1$ is defined in equation~(\ref{MR}). 

We shall use the following identity from Gradshteyn and Rizhik~\cite[equation 0.160.2]{zwillinger}:
\[
\sum\limits_{r=0}^n {n \choose r}(-1)^r \frac{\Gamma(r + b)}{\Gamma(r+a)} = \frac{B(n + a - b, b)}{\Gamma(a-b)}.
\] where $B$ is Euler's Beta function. We also use the property
\[
B(x+1, y) = \frac{x}{x+y} B(x, y).
\]

The right side of the equation is
\begin{align*}
& -\frac{\sum\limits_{l=0}^{h-1} {{n- l -h} \choose {h-l}}{{n -2h} \choose {\frac{n - m}{2} - h }} (-1)^{h-l} \; \left(\frac{n - 2l + 1}{n - l + 1}\frac{{n \choose l}}{{n \choose {\frac{n-m}2}}}\right)}{
{{n-2h} \choose {\frac{n-2h-m}{2}}} }\\
=& -\sum\limits_{l=0}^{h-1} {{n- l -h} \choose {h-l}} (-1)^{h-l} \; \left(\frac{n - 2l + 1}{n - l + 1}{n \choose l}\right)\\
\end{align*}
Subtracting the left side of the equation $\frac{n - 2h + 1}{n - h + 1}\frac{{n \choose h}}{{n \choose {\frac{n-m}2}}}$, gives
\begin{align*}
0 =& \sum\limits_{l=0}^{h} {{n- l -h} \choose {h-l}} (-1)^{h-l} \; \left(\frac{n - 2l + 1}{n - l + 1}{n \choose l}\right)\\
=& \sum\limits_{l=0}^{h} {{n- l -h} \choose {h-l}} (-1)^{l} \; \left(\frac{n - 2l + 1}{(n - l + 1)!l!}\right)\\
=& \sum\limits_{l=0}^{h} {h \choose l} \frac{(n-l-h)!}{(n-2h)!} (-1)^{l} \; \left(\frac{n - 2l + 1}{(n - l + 1)!}\right)\\
=& \sum\limits_{l=0}^{h} {h \choose l} (n-l-h)! (-1)^{l} \; \left(\frac{n - 2l + 1}{(n - l + 1)!}\right)\\
=& \sum\limits_{l=0}^{h} \frac{{h \choose l}}{{{n-l+1} \choose {h+1}} } (-1)^{l} \; (n - 2l + 1)\\
=& (n+1) \sum\limits_{l=0}^{h} \frac{{h \choose l}}{{{n-l+1} \choose {h+1}} } (-1)^{l} - 2 \sum\limits_{l=0}^{h} l \frac{{h \choose l}}{{{n-l+1} \choose {h+1}} } (-1)^{l} \\
=& (n+1) \sum\limits_{l=0}^{h} \frac{{h \choose l}}{{{n-l+1} \choose {h+1}} } (-1)^{l} - 2 h (-1)^h\sum\limits_{l=0}^{h-1} \frac{{{h-1} \choose l}}{{{n-l} \choose {h+1}} } (-1)^{l} \\
=& (n+1) (h+1)! (-1)^h \sum\limits_{l=0}^{h} {h \choose l} (-1)^{l} \frac{\Gamma(l+n-2h+1)}{\Gamma(l+n-h+2)} \\
& - 2 h (h+1)! (-1)^h \sum\limits_{l=0}^{h-1} {{h-1} \choose l} (-1)^{l} \frac{\Gamma(l+ n-2h+1)}{\Gamma(l+n-h+2)} \\
=& (n+1) (h+1)! (-1)^h \frac{B(2h +1, n-2h+1)}{\Gamma(h+1)} - \\
&2 h (h+1)! (-1)^h \frac{B(2h, n-2h+1)}{\Gamma(h+1)} \\
=& B(2h, n-2h+1) \left( (n+1) (h+1) (-1)^h \frac{2h}{n+1} -2 h (h+1) (-1)^h \right)\\
=&0.
\end{align*}

\section{Another Proof of Theorem~\ref{RMD}}

We use a formula that relates the Zernike Radial polynomials to Jacobi polynomials~\cite{born}:
\begin{equation} \label{RPeq}
R_n^m(x) = {(-1)}^{\frac{(n-m)}2} x^m P_{\frac{(n-m)}2}^{(m , 0)} (1 - 2 x^2)
\end{equation} where $n$ and $m$ are non-negative integers and have the same parity. We use equation~\ref{RPeq} to give the inverse mapping.

\begin{theorem}
The coefficient function for the mapping
\[ 
\{x^m, x^{m+2}, \ldots, x^n\} \rightarrow \{ R_m^m(x), R_{m+2}^m, \ldots, R_n^m(x)\}
\]
where $n \geq m \geq 0$ and $m$ and $n$ have the same parity is
\begin{equation}
\beta(n, m, k) = (-1)^{k} 2^{-v} \sum_{l=0}^{k} {v \choose l} (-1)^l \beta_1(n - 2l, m, k-l) \label{b3}
\end{equation} where $v = \frac{n-m}2$, $0 \leq k \leq v$, and
\begin{align}
\beta_1(n, m, k) =& (n - 2k + 1)(m+2)_{v-k-1} \nonumber\\
& \left( \sum_{l = 0}^{k} 2^{v-l} {v \choose {v-l}} {(-1)}^l 
\frac{(v+1-k)_{k-l}(k-l+1)_{v-k}}{(m + 2)_{v-l}(v-l+m+2)_{v-k}} \right).
\end{align}
\end{theorem}

\begin{proof}
This mapping has the property
\begin{align}
x^{n} =& \sum_{k=0}^v \beta(n, m, k) R_{n-2k}^m(x).
\end{align}

From equation~(\ref{RPeq}), 
the change of basis mapping from $\{1, x^2, x^4, \ldots, x^{n-m}\}$ to $\{1, (1-2 x^2), (1-2x^2)^2, \ldots (1-2x^2)^v\}$ has the coefficient function given by \cite[equation~(12)]{wolfram2107}. It is a linear shift with $y = x^2$, $c = -2$ and $d = 1$. This gives
$\alpha_2(v, k) = {v \choose k} c^{-v}(-d)^k$ for the mapping from $\{1, y, y^2, \ldots, y^{\frac{n-m}2}\}$ to $\{1, (1-2y), (1-2y)^2, \ldots, (1-2y)^{\frac{n-m}2}\}$.

This simplifies to
\begin{equation} \label{b2}
\beta_2(n, m, k) = {v \choose k} \frac{(-1)^{v-k}}{2^v}.
\end{equation} where $v = \frac{n-m}2$ and $0 \leq k \leq v$.

Equation~(\ref{alpha-zP}) below is the coefficient function for the mapping from basis of the monomials to the Jacobi polynomials from Wolfram~\cite[equation(19)]{wolfram2108}. In this application of the equation, $\alpha = m$, $\beta = 0$ and $v = \frac{n-m}2$. We have
\begin{align} \label{alpha-zP}
\alpha(v, k) =& (n - 2k + 1)(m+2)_{v-k-1} \nonumber\\
& \left( \sum_{l = 0}^{k} 2^{v-l} {v \choose {v-l}} {(-1)}^l 
\frac{(v+1-k)_{k-l}(k-l+1)_{v-k}}{(m + 2)_{v-l}(v-l+m+2)_{v-k}} \right)
\end{align} where $0 \leq k \leq v$. This gives
\[
x^v = \sum_{k=0}^v \alpha(v, k) P_{v-k}^{(m, 0)}(x)
\]

From equation~(\ref{RPeq}), we have
\begin{align} \label{MpR}
(1 -2 x^2)^v =& \sum_{k=0}^v (-1)^{k-v} x^{-m} \alpha(v, k) R_{2(v-k) + m}^m(x) \nonumber\\
=& x^{-m} \sum_{k=0}^v (-1)^{v-k} \beta_1(n, m, k) R_{n-2k}^m(x)
\end{align} where $\beta_1(n, m, k) = \alpha(v, k)$. The function name is changed because the polynomials $R_{n-2k}^m(x)$ have definite parity following the substitution above of $1 -2x^2$ for $x$.

In general, the coefficient function from the composition of $\beta_1$ and $\beta_2$ is given by
\[
\beta_3(n, m, k)= \sum_{l=0}^{k} \beta_1(n - 2l, m, k-l) \beta_2(n, m , l)
\] where $0 \leq k \leq v$, e.g., ~\cite[\S 5.1]{wolfram2107}. Applying this gives
\begin{equation} 
\beta_3(n, m, k)= (-1)^{v} 2^{-v} \sum_{l=0}^{k} {v \choose l} (-1)^l \beta_1(n - 2 l, m, k-l)
\end{equation} where $0 \leq k \leq v$.

We have 
\begin{align}
x^{2v} =& x^{-m} \sum_{k=0}^v (-1)^{k-v} \beta_3(n, m, k) R_{n-2k}^m(x) \mbox{, i.e.,}\nonumber\\
x^{n} =& \sum_{k=0}^v (-1)^{k-v} \beta_3(n, m, k) R_{n-2k}^m(x). \label{b3eq}
\end{align}

This gives,
\begin{align}
\beta(n, m, k) =& (-1)^{k-v} \beta_3(n, m,k) \nonumber \\
=& (-1)^{k} 2^{-v} \sum_{l=0}^{k} {v \choose l} (-1)^l \beta_1(n - 2l, m, k-l) 
\end{align}
and

\begin{align} \nonumber
x^{n} =& \sum_{k=0}^v \beta(n, m, k) R_{n-2k}^m(x).
\end{align} on substitution into equation (\ref{b3eq}), as required.
\end{proof}

\section{Another Proof of Theorem~\ref{beta-RM}} \label{appfor6}
We have
\begin{align*}
&\sum\limits_{l=0}^{h} \beta_1(n, m+2l, \frac{n-m}2 -h) \; \beta(n,m,v-l)\\
=&\sum\limits_{l=0}^{h} {{n-(\frac{n-m}2 -h)} \choose (\frac{n-m}2 -h)}{{n-2(\frac{n-m}2 -h)} \choose {\frac{n-(m+2l)}{2} - (\frac{n-m}2 -h)}} (-1)^{(\frac{n-m}2 -h)} (-1)^{\frac{n-m}2 + l}\\
& \frac{(m+1)_{l-1} (m + 2l)}{l! {{\frac{n+m}2} \choose {\frac{n-m}2}}}\\
= &\frac{(-1)^h}{{{\frac{n+m}2} \choose {\frac{n-m}2}}} {{\frac{n+m}2 +h} \choose {\frac{n-m}2 -h}} \sum\limits_{l=0}^{h}{{m+2h} \choose {h - l}} (-1)^{l} \frac{(m+1)_{l-1} (m + 2l)}{l! }\\
=&\frac{(-1)^h}{{{\frac{n+m}2} \choose {\frac{n-m}2}}} {{\frac{n+m}2 +h} \choose {\frac{n-m}2 -h}} \sum\limits_{l=0}^{h}{{m+2h} \choose {h - l}} (-1)^{l} \frac1{m}{{m+l-1} \choose l} (m + 2l)\\
= &\frac{(-1)^h}{{{\frac{n+m}2} \choose {\frac{n-m}2}}} {{\frac{n+m}2 +h} \choose {\frac{n-m}2 -h}} \sum\limits_{l=0}^{h}{{m+2h} \choose {h - l}} (-1)^{l} \left ({{m+l-1} \choose l} + 2{{m+l-1} \choose {l-1}} \right)\\
=&\frac{(-1)^h}{{{\frac{n+m}2} \choose {\frac{n-m}2}}} {{\frac{n+m}2 +h} \choose {\frac{n-m}2 -h}} \sum\limits_{l=0}^{h}{{m+2h} \choose {h - l}} (-1)^{l} \left ({{m+l} \choose l} + {{m+l-1} \choose {l-1}} \right)\\
\end{align*}

The sum \[
\sum\limits_{l=0}^{h}{{m+2h} \choose {h - l}} (-1)^{l} {{m+l} \choose l}
\]
\begin{align*}
=& \sum\limits_{l=0}^{h} (-1)^{l} \frac{(m+2h)!}{(h-l)! (m+h+l)!} \frac{(m+l)!}{m! l!} \\
=& \frac{(m+2h)!}{h! m!} \sum\limits_{l=0}^{h} (-1)^{l} {{h} \choose l} \frac{\Gamma(l + m +1)}{\Gamma(l + m + h + 1)}\\
=& \frac{(m+2h)!}{h! m!} \frac{B(2h, m+1)}{\Gamma(h)}, \mbox{from~\cite[equation 0.160.2]{zwillinger}}\\
=& \frac{(m+2h)!}{h! m!} \frac{(2h -1)! m!}{(2h+m)! (h-1)!}\\
=& {{2h -1} \choose h}.
\end{align*}

The sum 
\[
\sum\limits_{l=0}^{h}{{m+2h} \choose {h - l}} (-1)^{l} {{m+l-1} \choose {l-1}}
\]
\begin{align*}
=& \sum\limits_{l=1}^{h} (-1)^{l} \frac{(m+2h)!}{(h-l)! (m+h+l)!} \frac{(m+l-1)!}{m! (l-1)!} \\
=& -\frac{(m+2h)!}{m! (h-1)!} \sum\limits_{k=0}^{h-1} (-1)^{k} {{h-1} \choose k} \frac{\Gamma(k+m+1)}{\Gamma(k+m+h+2)} \\
=& -\frac{(m+2h)!}{m! (h-1)!} \frac{B(2h, m+1)}{\Gamma(h+1)}, \mbox{from~\cite[equation 0.160.2]{zwillinger}}\\\\
=& -\frac{(m+2h)!}{m! (h-1)!} \frac{(2h -1)! m!}{(2h+m)! h!}\\
=& - {{2h -1} \choose h}.
\end{align*}

Hence,
\[
\sum\limits_{l=0}^{h}{{m+2h} \choose {h - l}} (-1)^{l} \left ({{m+l} \choose l} + {{m+l-1} \choose {l-1}} \right) = 0
\] and 
\[
\sum\limits_{l=0}^{h} \beta_1(n, m+2l, \frac{n-m}2 -h) \; \beta(n,m,v-l) = 0.
\]

\end{appendix}

\end{document}